\documentclass[11pt,twoside]{amsart}

      \usepackage{amssymb}
      \usepackage{graphicx}
      \usepackage{dcolumn,indentfirst,cite,color,hyperref}
     \usepackage{pgfplots}
      \usepackage{pgf}
    \usepackage{tikz}
    \usetikzlibrary{arrows, decorations.pathmorphing, backgrounds, positioning, fit, petri, automata}
 
\definecolor{yellow1}{rgb}{1,0.8,0.2} 

\theoremstyle{plain}
 \newtheorem{thm}{Theorem}[section]
 
\newtheorem{lem}[thm]{Lemma}

\newtheorem{pro}[thm]{Proposition}
\newtheorem{rmk}[thm]{Remark}

\newtheorem{defi}[thm]{Definition}
\newcommand{\bess}{\begin{eqnarray*}}
\newcommand{\eess}{\end{eqnarray*}}

\input xy
\xyoption{all}

\usepackage{fancyhdr} %Ò³ÃŒºÍÒ³œÅ
\begin{document}

\author{\textsc{Xiaoguang Wang}}
\address{School of Mathematical Sciences, Zhejiang University, Hangzhou, 310027, China}
\email{wxg688@163.com}

%%%%%%%%%%%%%%%%%%%%%%%%%%%%%%%%%%%%%%%%%%%%%%%%%%%%%

   \title[]{Hyperbolic components and cubic polynomials} % with a periodic critical point}

   \begin{abstract} In the space of cubic polynomials, Milnor defined a notable curve $\mathcal S_p$,
   consisting of cubic polynomials with a periodic critical point, whose period is exactly $p$.
  In this paper, we show that for any integer $p\geq 1$, any bounded hyperbolic component on  $\mathcal{S}_p$ 
  is a Jordan disk. 
     \end{abstract}

   % AMS subject classifications (used in AMS journals)

   % AMS keywords (used in AMS journals)

    \date{\today}

   \maketitle

\section{Main Theorem} \label{intro}

Polynomial maps $f: \mathbb C\rightarrow \mathbb C$, viewed as dynamical systems,
 yield complicated behaviors under iterations.
 Their bifurcations, both in the dynamical plane and in the parameter space,
 are the major attractions in the field of complex dynamics in recent thirty years.

   To start, let $\mathbf P(d)$ be the space of monic and centered polynomials of degree $d\geq 2$.
The connectedness locus $\mathbf C(d)$ consists of  $f\in \mathbf P(d)$ whose Julia set is connected, while the 
shift locus $\mathbf S(d)$ consists of  maps in $\mathbf P(d)$  for which all critical orbits escape to infinity under iterations.
The topology of $\mathbf C(d)$ and $\mathbf S(d)$ attracts lots of people. 
A well-known theorem states that $\mathbf C(d)$ is compact and cellular, this is  proven by Douady-Hubbard \cite{DH1} for $d=2$,  
 Branner-Hubbard \cite{BH1} for $d=3$, and  Lavaurs \cite{L},  DeMarco-Pilgrim \cite{DP1} for the general cases.  
 On the other hand, the shift locus $\mathbf{S}(d)$  is studied from a different viewpoint. 
 It's fundamental group is studied by  Blanchard, Devaney and Keen \cite{BDK},
 while its simplicial structure and bifurcations are  studied  comprehensively  by
 DeMarco and Pilgrim \cite{DP1,DP2,DP3}.
 
 When $d=3$, more amazing structure of the parameter space are revealed.
Pioneering work of Branner and Hubbard \cite{BH1} exhibited two dynamically meaningful solid tori with linking number 3 in the 
parameter space.  Meanwhile, it is observed by Milnor,  and proven by
 Lavaurs \cite{L} 
that $\mathbf C(3)$ is not locally connected. Later, Epstein and Yampolski \cite{EY}  proved the existence of products of the Mandelbrot set in $\mathbf C(3)$. 
 These striking results indicate the complexity of the cubic polynomial space.

In their papers \cite{BH1,BH2}, Branner and Hubbard used the following form
$$f_{c,a}(z)=z^3-3c^2z+a,$$
where $(c,a)\in \mathbb C^2$ is a pair of parameters. 
This critically marked form are widely adopted by the followers. 
Note that $f_{c,a}$ has two critical points $\pm c$.

To study the parameter space of $f_{c,a}$,  
Milnor suggested to study the one dimensional slices, and defined a kind of notable slices  called the {\it critically-periodic curves} $\mathcal S_p$, $p\geq 1$. The curve $\mathcal S_p$ consists of $(c,a)\in \mathbb C^2$ for which  the critical point $c$ has period exactly $p$ under iterations of $f_{c,a}$:
$$\mathcal{S}_p=\{(c, a)\in \mathbb{C}^2; 
f_{c,a}^p(c)=c \text{ and } f^k_{c,a}(c)\neq c,  \forall \ 1 \leq k<p \}.$$
Milnor showed that $\mathcal S_p$ is a smooth affine algebraic curve, and asked whether it is irreducible.
A proof of irreducibility is recently announced  by Arfeux and Kiwi \cite{AK}. 
 
The curve $\mathcal S_p$ has a remarkable topology. 
 It is known that $\mathcal S_1$ is biholomorphic to the $\mathbb C$,  
$\mathcal S_2$ is biholomorphic to the punctured plane $\mathbb C^*$, $\mathcal S_3$ has genus one with $8$ punctures (see Figure 1),  $\mathcal S_4$ has genus 15 with 20 punctures. 
Both the genus $g_p$ and the number $N_p$ of punctures of $\mathcal S_p$ grow exponentially with $p$.
 Bonifant, Kiwi and Milnor \cite{BKM} proved that the Euler characteristic of $\mathcal S_p$ (without assuming its irreducibility) is given by 
$$\chi (\mathcal S_p)=(2-p) d(p),$$
where $d(p)$ is the degree of $\mathcal S_p$, satisfying the formula:
$\sum_{n|p}d(n)=3^{p-1}.$
 The genus $g_p$ and the number  $N_p$ of punctures of $\mathcal S_p$ have no explicit formulas.
 An algorithm to compute $N_p$ (hence also $g_p$) is designed by DeMarco and Schiff \cite{DS}, building on previous work of DeMarco and Pilgrim \cite{DP3}. 
 By pluripotential methods,  Dujardin \cite{Du} showed that 
 $$\lim_{p\rightarrow+\infty}\frac{\chi (\mathcal S_p)+N_p}{3^{p}}\rightarrow -\infty.$$
 Assuming the irreducibility \cite{AK}, the above behavior implies that  the genus $g_p$ actually grows faster than $3^p$.

  \begin{figure}[h]
\begin{center}
\includegraphics[height=5cm]{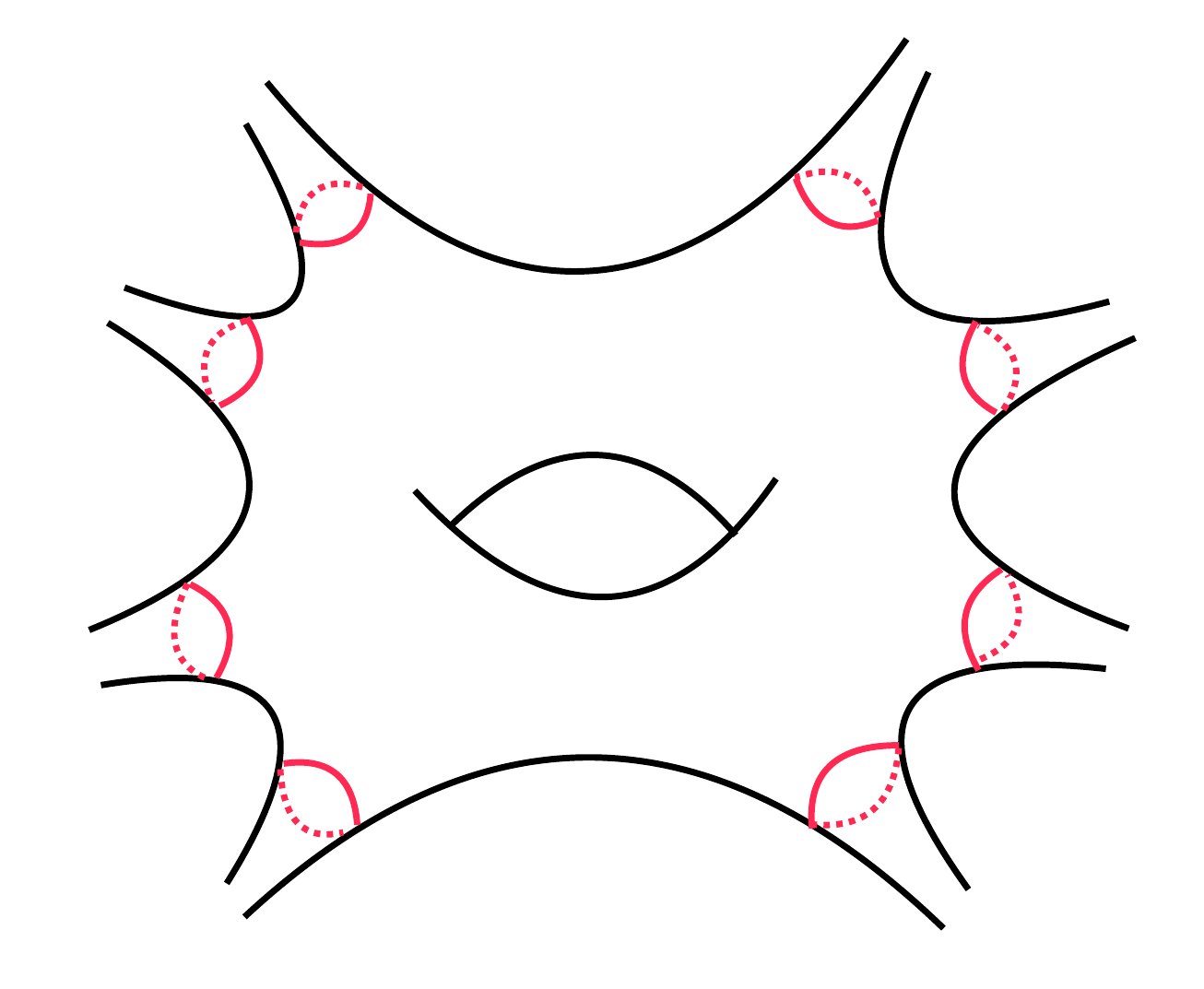} 
 \caption{$\mathcal S_3$ has a non-trivial topology. It is a complex torus with 8 punctures.
 }
\end{center}\label{f5}
\end{figure}

In this article, we study the {\it bifurcations of dynamical systems} on the algebraic curve $\mathcal S_p$, see Figure 2. 
 Precisely,  the bifurcations on the boundary of stable regions are the main focus. Here, `stable' refers to {\it hyperbolic}.
 Recall that, a rational map $f$ is {\it hyperbolic} if all  the critical points are attracted by the attracting cycles.
 In a holomorphic family of rational maps, hyperbolic maps form an open (and conjecturally dense) subset, each component  is called a {\it hyperbolic component}.

 \begin{figure}[!htpb]
  \setlength{\unitlength}{1mm}
  {\centering
  \includegraphics[width=42mm]{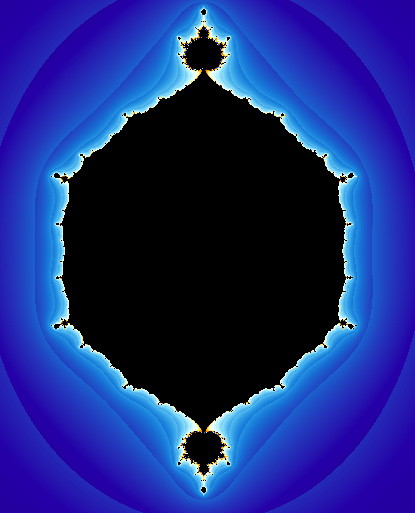}
  \includegraphics[width=72mm]{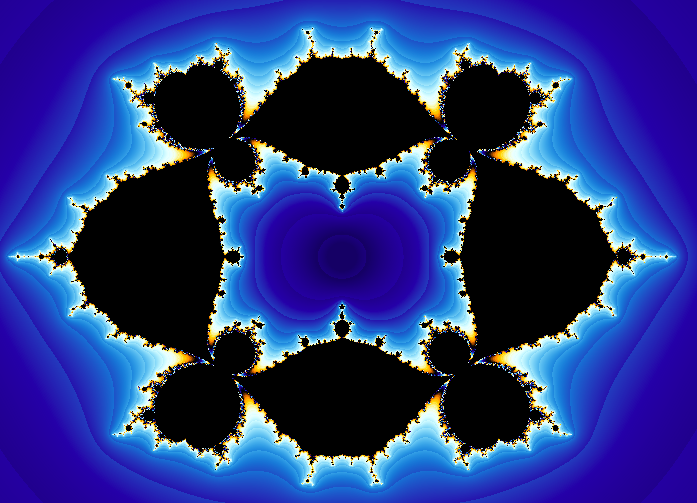}   }
  \caption{Bifurcations on $\mathcal S_1$  (left) and $\mathcal S_2$ (right).} 
\end{figure}

In this paper, we establish the following   
  \begin{thm}   \label{Sp_hyp}  
 For any integer $p\geq1$,   any bounded hyperbolic component on   $\mathcal{S}_p$  
 is a Jordan disk.
\end{thm} 

 Theorem \ref{Sp_hyp} is among one of several conjectural pictures of $\mathcal{S}_p$, proposed by Milnor \cite[p.13]{M2},\cite{M3}.  
 The case $p=1$ is proven independently by Faught \cite{F}  and Roesch \cite{R}.
 The main analytical tool in their proof is the so-called {\it para-puzzle}  technique, whose philosophy,  as interpreted by 
 Douady, is: sowing in the dynamical plane and harvesting in the parameter space. 
  However, the 
  para-puzzle technique  loses its power when dealing with the parameter space with a complicated topology. 
   Since the topology of $\mathcal S_p$ is far beyond understanding 
      when $p$ is large, this makes a tough enemy of the para-puzzles.

 Instead of using  para-puzzle technique, our approach makes the most of  
the dynamical puzzles, combinatorial rigidity, and holomorphic motion theory.
 Our arguments  are {\it local}, this makes our techniques being powerful and 
 serve as a model to study bifurcations on more general algebraic curves: those 
 defined by 
 critical relations, in any critically marked polynomial space.
  
The strategy and organization of the proof is as follows:

A classification and dynamical parameterization of hyperbolic components, due to Milnor, is given in Section \ref{hyp-comp}.
Then some basics of dynamical rays are recalled in Section \ref{dyn-ray}.

Section \ref{BHY-puzzle} solves one main technical difficulty in the Branner-Hubbard-Yoccoz puzzle theory: finding a puzzle with
 a non-degenerate critical annulus.
 The idea is to discuss the relative position of the critical orbits with respect to the candidate graphs, 
 the proof logics has an independent interest.

In Section \ref{rigidity-puzzle}, we prove the combinatorial rigidity for maps on $\mathcal S_p$.
The ideas and methods for quadratic polynomials \cite{H,L, M5} can not work here. 
We take advantage of recent development  \cite{AKLS,KL1,KL2, KSS,KS} in deriving rigidity phenomenon to treat our situations.

We then prove Theorem \ref{Sp_hyp} for two types of hyperbolic components in Section \ref{br},  using rigidity and characterization of boundary maps. 
Finally, we deal with the capture type hyperbolic components in Section \ref{cc}. Instead of using rigidity there, we
make the best use of  holomorphic motion theory.

Theorem \ref{Sp_hyp} then  follows from Theorems \ref{parameterization}(4), \ref{regularity} and \ref{boundary-C}.
Note that its statement is complete because  unbounded hyperbolic components on $\mathcal{S}_p$ are  not Jordan disks, see \cite{BKM}.
To the author's knowledge, Theorem \ref{Sp_hyp} is the first complete description of boundary regularity of stable regions whose parameter space has a non-trivial topology.

\vspace{4 pt} 
% \subsection
\noindent \textbf{Acknowledgement} 
The author  thanks Yongcheng Yin for helpful discussions on polynomials rays and  puzzles. 
The research is supported by NSFC.
% National Natural Science Foundation of China (NSFC 11622108)
% and the Fundamental Research Funds for the Central Universities. 

\section{Hyperbolic components} \label{hyp-comp}

 In this paper, for any $(c,a)\in \mathcal{S}_p$,  any point $z$ in the Fatou set of $f_{c,a}$, let $U_{c,a}(z)$ be the Fatou component containing $z$.  
 Let's use the notations 
 \bess
 &\mathcal{B}_{c,a}=\{U_{c,a}(f^k_{c,a}(c)); 0\leq k<p\},& \\
 &\mathcal{A}_{c,a}=U_{{{c},a}}(c)\cup U_{c,a}(f_{c,a}(c))\cup\cdots\cup  U_{c,a}(f^{p-1}_{c,a}(c)).&
 \eess
The boundary of each $V\in \mathcal{B}_{c,a}$, according to the work of Roesch and Yin \cite{RY}, is a Jordan curve.
The connectedness locus  of $\mathcal S_p$ is denoted by $\mathcal C(\mathcal S_p)$.
For any $z$, let ${\rm orb}(z)=\{f^{k}_{c,a}(z);k\in \mathbb N\}$ be the set of forward orbit of $z$.

 According to Milnor\cite{M3}, there are four types of bounded hyperbolic components on $\mathcal S_p$ (see Figures 3 and 4):
  
  \vspace{5 pt}
  
  \textbf{Type A} ({\it Adjacent critical points}), with both critical points in the same periodic
Fatou component.

\textbf{Type B} ({\it Bitransitive}), with two critical points in different Fatou components
belonging to the same periodic cycle.
 
\textbf{Type C} ({\it Capture}), with just one critical point in the cycle of periodic Fatou
components. The orbit of the other critical point must eventually land in (or
be “captured by”) this cycle.
 
\textbf{Type D} ({\it Disjoint attracting orbits}), with two distinct attracting periodic orbits,
each of which necessarily attracts just one critical orbit.
 
 \vspace{5 pt}

 \begin{figure}[!htpb]
  \setlength{\unitlength}{1mm}
  {\centering
  \includegraphics[width=75mm]{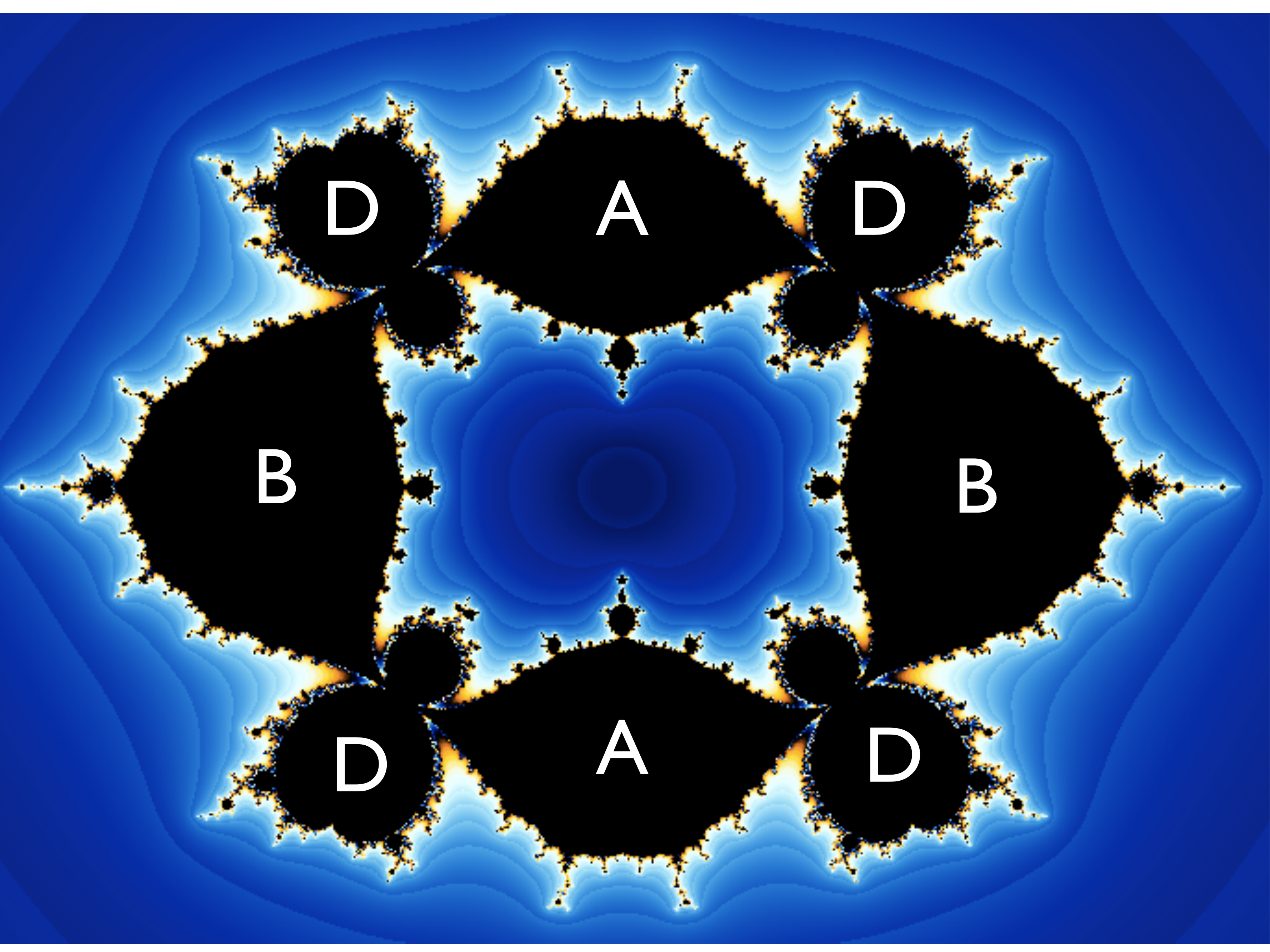}  
  }
  \caption{Hyperbolic components of Types-$A,B,D$  on $\mathcal S_2$.} 
\end{figure}

 \begin{figure}[!htpb]
  \setlength{\unitlength}{1mm}
  {\centering
  \includegraphics[width=75mm]{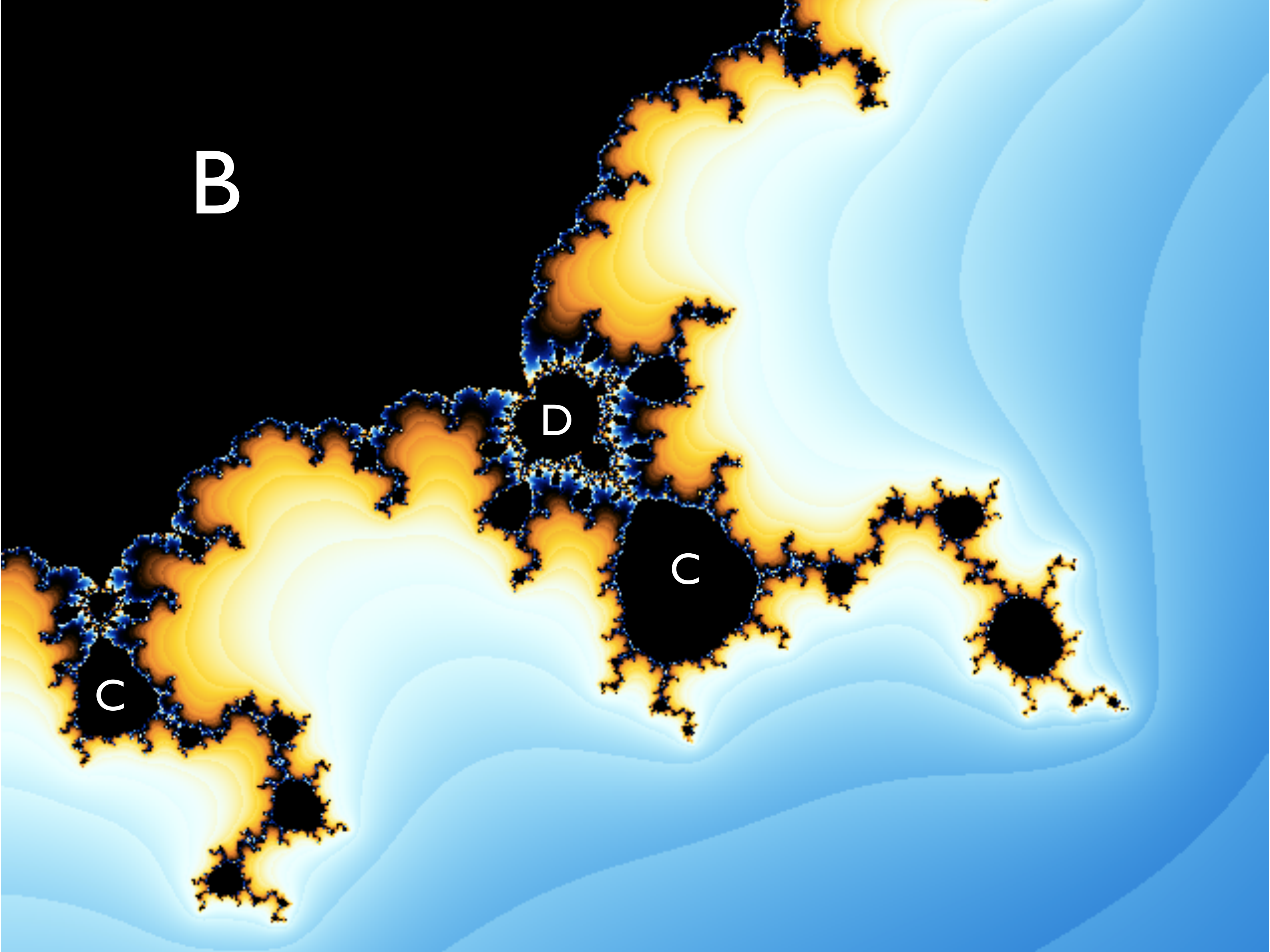}  
  }
  \caption{Some Type-$C$ components  on $\mathcal S_2$.} 
\end{figure}

All these four types of hyperbolic components admit the following natural dynamical parameterizations, due to Milnor \cite[Lemma 6.8]{M3}. This serves as the first step to study the boundaries of hyperbolic components.

   \begin{thm} \label{parameterization}
   Let  $\mathcal{H}$ be a hyperbolic component in $\mathcal C(\mathcal{S}_p)$ of Type-$\omega$.
   
   1. If $\omega=A$, i.e. $-c\in U_{c,a}(c)$, then the map
 $$\Phi: \mathcal{H}\rightarrow \mathbb{D}, \  ({c},a)\mapsto B_{c,a}(-{c})$$
 is a double cover ramified at a single point, where $B_{c,a}$ is the B\"ottcher map of $f^p_{{c},a}$ defined in a neighborhood of $c$.
 
  2. If $\omega=B$, i.e. $-c\in U_{c,a}(f^l_{c,a}(c))$ for some $1\leq l<p$, then the map
  $$\Phi: \mathcal{H}\rightarrow \mathbb{D}, \  ({c},a)\mapsto B_{c,a}(-{c})$$
 is a triple cover ramified at a single point, where $B_{c,a}$ is the B\"ottcher map of $f^p_{{c},a}$ defined in a neighborhood of $f_{c,a}^l(c)$.
 
 3.  If $\omega=C$, i.e. $f_{c,a}^l(-c)\in \mathcal{A}_{c,a}$ for some smallest integer $l>0$, then 
 $$\Phi: \mathcal{H}\rightarrow \mathbb{D}, \  ({c},a)\mapsto B_{c,a}(f_{c,a}^l(-c))$$
 is a conformal isomorphism, where $B_{c,a}$ is the B\"ottcher map of $f^p_{{c},a}$ defined in $U_{c,a}(f_{c,a}^l(-c))$.
 
 4. If $\omega=D$, let $z_{{c},a}\in U_{{{c},a}}(-c)$ be the attracting point with period say $q$, then the multiplier map
 $$\rho: \mathcal{H}\rightarrow \mathbb{D}, \  ({c},a)\mapsto (f^q_{{c},a})'(z_{{c},a})$$
 is  a conformal isomorphism. In this case, $\rho$ can be extended to a homeomorphism
  $\rho: \overline{\mathcal{H}}\rightarrow \overline{\mathbb{D}}$ (implying that $\partial \mathcal{H}$ is a Jordan curve).
 \end{thm}
 
 We remark that for each $p\geq 1$, there are only finitely many Type-$A,B$ components on $\mathcal S_p$, but there are infinitely many Type-$C$ or $D$ components.
 
 By Theorem \ref{parameterization}, for a Type-$\omega\in\{A,B\}$ component $\mathcal H$, one has 
 $${\rm deg}(\Phi)={\rm deg}(f_{c,a}^p|_{U_{c,a}(c)})-1,$$
where $f_{c,a}$ is a representative map in  $\mathcal H$.  The above number depends only on the type $\omega\in\{A,B\}$, not the specific component $\mathcal H$. For this, we write $d_{\omega}={\rm deg}(\Phi)$.
Clearly $d_A=3, d_B=2$. The notation $d_\omega$ will be used later.%in the other part of the paper.

%We introduce a notation $d_{\omega}={\rm deg}(\Phi)$.
%In the rest part of this article,  
%
%
%
%we write $d_\omega={\rm deg}(\Phi: \mathcal{H}\rightarrow \mathbb{D})$ for the 
% Type-$\omega\in\{A,B\}$ component $\mathcal H$. By Theorem \ref{parameterization}, 
% $$d_A=3, d_B=2$$
 % $d_\omega$ to denote the degree of the map 
 
% We remark that in the Type-A or -B case, the ramification point $(c^*,a^*)$ for $\Phi$ is a removable singularity.
%One can make $\Phi$  holomorphic on $\mathcal{H}$ by setting $\Phi(c^*,a^*)=0$.
% 
 
 \section{Dynamical rays} \label{dyn-ray}
 
 We introduce the {\it dynamical rays}  in this section, as a preparation for further discussions.
  These materials are standard in polynomial dynamics.

 \subsection{Dynamical internal rays}   There are finitely many maps  on $\mathcal{S}_p$, for which $-c$ meets the orbit of $c$ \cite[Lemma 5.8]{M3}. 
  Let 
 $$\mathcal{S}_p^*=\mathcal{S}_p-\{(c,a); f_{c,a}^{k}(c)=-c \text{ for some } 0< k\leq p\}.$$
 
 Let $(c,a)\in \mathcal{S}_p^*$. For any $V\in\mathcal{B}_{c,a}$,
%  and  set $W=$.Assume $V=U_{c,a}(f^{k}_{c,a}(c))$, then
   its Green function
  $G_{c,a}^V: V\rightarrow [-\infty, 0)$ is defined by 
  $$G^V_{c,a}(z)=\lim_{n\rightarrow +\infty}2^{-n}\log|f_{c,a}^{pn}(z)-w|,$$
  where $w\in {\rm orb}(c)\cap V$.
 One may verify that 
  \begin{equation*}
G^{f_{c,a}(V)}_{c,a} \circ f_{c,a}=
\begin{cases}
 2 G^{V}_{c,a},\ \  & \text{ if }V= U_{c,a}(c),\\
G^{V}_{c,a},\ \ &\text{ if }V=U_{c,a}(f^{k}_{c,a}(c)), 1\leq k<p.
\end{cases}
\end{equation*}
  
 The locus $(G^{V}_{c,a})^{-1}(\ell)=\{z\in V; G^V_{c,a}(z)=\ell\}$ with $\ell<0$ is called an {\it equipotential curve} in $V$.
 The internal rays are defined as follows.
  
  If $f_{c,a}$ is hyperbolic and $-c\in U:=U_{c,a}(f^l_{c,a}(c))$ for some $0\leq l<p$, then the B\"ottcher map
  $B^U_{c,a}$ of $f^p_{{c},a}$ is defined in a neighborhood of $f_{c,a}^{l}(c)$.
   For any $t\in\mathbb{S}$,
the set $R_{c,a}^U(t)$ in $U$ %the Fatou component $V=U_{c,a}(f_{c,a}^{l}(c))$
 is defined as
the orthogonal trajectory (possibly bifurcates) of the equipotential curves, starting from $f^l_{c,a}(c)$ and
 containing $(B^U_{c,a})^{-1}((0,\epsilon)e^{2\pi it})$
 for some $\epsilon\in (0,1)$. By conformal pushing forward or pulling back via some iterations of  $f_{c,a}$, 
  we can define $R_{c,a}^W(t)$ for any $W\in \mathcal{B}_{c,a}-\{U\}$.
   The set $R_{c,a}^U(t)$ or $R_{c,a}^W(t)$ is called an {\it internal ray} if it does not bifurcate,
 namely $2^nt \neq \arg B^U_{c,a}(-c)$ for any $n\in\mathbb N$. 
  The pulling back procedure allows one to define internal rays in any Fatou component
 whose orbit meets $U$. 
  
   In all other situations,  set $V=U_{c,a}(c)$. The B\"ottcher map $B^{V}_{c,a}$ of $f^p_{{c},a}$ can be defined in $V$, and the internal ray $R_{c,a}^{V}(t)=(B^{V}_{c,a})^{-1}((0,1)e^{2\pi it})$, $\forall \ t\in \mathbb{S}$.
 By conformal pulling back $R_{c,a}^{V}(t)$ via iterations of $f_{c,a}$, one can define the internal ray $R_{c,a}^U(t)$
 in any Fatou component $U  (\neq V)$
 whose orbit meets $V$.
 
One may verify that 
  \begin{equation*}
f_{c,a}(R_{c,a}^V(t))=
\begin{cases}
R_{c,a}^{f_{c,a}(V)}(2t),\ \  & \text{ if }V= U_{c,a}(c),\\
R_{c,a}^{f_{c,a}(V)}(t),\ \ &\text{ if }V=U_{c,a}(f^{k}_{c,a}(c)), 1\leq k<p.
\end{cases}
\end{equation*}

 \subsection{Dynamical external rays} For any $(c,a)\in \mathcal{S}_p$, let $A_{c,a}^{\infty}$ be the  basin of $\infty$ for $f_{c,a}$.
 Near $\infty$, the B\"ottcher map $B_{c,a}^{\infty}$ is defined as 
 $$B_{c,a}^{\infty}(z)=\lim_{n\rightarrow +\infty}\sqrt[3^n]{f_{c,a}^{n}(z)}.$$
 The B\"ottcher map  $B_{c,a}^{\infty}$ is unique if we require that it is asymptotic to the identity map at $\infty$.
 It satisfies $B_{c,a}^{\infty}(f_{c,a}(z))=B_{c,a}^{\infty}(z)^3$ when $|z|$ is large.
 
The Green function $G^\infty_{c,a}: A_{c,a}^\infty\rightarrow (0,+\infty)$ is defined by
  $$G^\infty_{c,a}(z)=\lim_{n\rightarrow +\infty}3^{-n}\log|f_{c,a}^n(z)|.$$
Each locus $(G^\infty_{c,a})^{-1}(\ell)=\{z\in A_{c,a}^\infty; G^\infty_{c,a}(z)=\ell\}$ with $\ell> 0$ is called an {\it equipotential curve}. 
For $t\in \mathbb{R}/\mathbb{Z}$, the set  $R_{c,a}^{\infty}(t)$ is
the orthogonal trajectory (possibly bifurcates) of the equipotential curves, 
starting from $\infty$ and containing $(B_{c,a}^{\infty})^{-1}((R,\infty)e^{2\pi it})$
 for some $R>0$.  
 It is called an {\it external ray}  if it does not bifurcate. % (i.e. its orbit meets the critical point $-c$). 
 Clearly $f_{c,a}(R_{c,a}^{\infty}(t))=R_{c,a}^{\infty}(3t)$.

 \subsection{Continuity of dynamical rays}
 If an internal (or external) ray lands at a repelling point, then they satisfy the following {\it local stability property}:
 
 \begin{lem}\label{cont-ray} Let $(c_0,a_0)\in \mathcal{S}^*_p$ so that the dynamical ray $R^\varepsilon_{c_0,a_0}(\theta)$  
 with $\varepsilon\in \mathcal{B}_{c,a}\cup \{\infty\}$ lands at a repelling periodic point $p_{c_0,a_0}$. Then there is a neighborhood $\mathcal{U}\subset \mathcal{S}^*_p$ of $(c_0,a_0)$ such that for all $({c}, {a})\in \mathcal{U}$, 
 
 1.  the set $R^\varepsilon_{{c}, {a}}(\theta)$ is a ray landing at a repelling periodic point, and
 
 2. the closure $\overline{R^\varepsilon_{{c}, {a}}}(\theta)$ moves continuously in Hausdorff topology with respect to $({c}, {a})\in \mathcal{U}$.
 \end{lem}

 \begin{proof} We only prove the result for external rays, the argument is similar for internal rays. The idea is  
 to cut the external ray $R^\infty_{c_0,a_0}(\theta)$ into two parts: one near $\infty$ and the other near the repelling point $p_{c_0,a_0}$.  Each part moves continuously w.r.t parameters.
This implies that,  after gluing them  together, the external ray itself moves continuously.  Here is the detail:

There exist a neighborhood $\mathcal{U}$ of $(c_0,a_0)$ and a large number $R>1$ such that for all
 $({c}, {a})\in \mathcal{U}$,  the B\"ottcher map $B_{c,a}^{\infty}$  is defined in
 $U_{c,a}^R=\{z\in A_{c,a}^\infty; G^\infty_{c,a}(z)>\log R\}$, and $B_{c,a}^{\infty}: U_{c,a}^R\rightarrow \{w\in \mathbb C;|w|>R\}$ is conformal.
 
  For all $t\in \mathbb{R}/\mathbb{Z}$, $({c}, {a})\in \mathcal{U}$ and $k\in\mathbb N$, let
  $$L^0_{c,a}(t)=(B_{c,a}^{\infty})^{-1}((R,+\infty)e^{2\pi it})$$
  and $L^k_{c,a}(t)$ be the component of $f_{c,a}^{-k}(L^0_{c,a}(3^kt))$ containing $L^0_{c,a}(t)$.

   We may shrink $\mathcal{U}$ if necessary so that

    \begin{itemize}

\item   after perturbation in  $\mathcal{U}$,   the $f_{c_0,a_0}$-repelling periodic  point $p_{c_0,a_0}$  becomes an  $f_{c,a}$-repelling periodic point $p_{c,a}$, and 

\item  there is a large integer $s$ (independent of $(c,a)\in \mathcal{U}$) so that
 $E_{c,a}:={L^{s+1}_{c,a}(\theta)\setminus L^{s}_{c,a}(\theta)}$ is included in a linearized neighborhood $Y_{c,a}$ of  $p_{c,a}$. 

\item  $-c\notin \bigcup_{k\geq 0}L_{c,a}^{s+1}(3^k\theta)$ for all   $(c,a)\in  \mathcal{U}$.

\end{itemize}

%We may further assume that $ \mathcal{U}$ is small enough so that  for all   $(c,a)\in  \mathcal{U}$,  the set $\bigcup_kL_{c,a}^{m+1}(3^k\theta)$ avoids the  free critical point $-c$. 

Note that $\theta$ is periodic under the angle tripling map $t\mapsto 3t  \ ({\rm mod} \  \mathbb{Z})$. Let $l$ be its period.
Since the inverse $h=(f_{c,a}^l|_{Y_{c,a}})^{-1}$ is contracting,  the arc 
%This implies that the closure of  
$$T_{c,a}=\bigcup_{k\geq0}  h^k(\overline{E_{c,a}})$$
moves continuously with respect to $(c,a)\in \mathcal{U}$. 

Note that neither $E_{c,a}$ nor $T_{c,a}$ meets the backward orbit of $-c$.
 Hence the set $R_{c,a}^\infty(\theta)$  defines an external ray.
 
Finally, the continuity of  $(c,a)\mapsto\overline{R_{c,a}^{\infty}}(\theta)$ follows  from the fact that $R_{c,a}^{\infty}(\theta)={L^{s+1}_{c,a}(\theta)\cup T_{c,a}}$ and the continuity of $L^{s+1}_{c,a}(\theta)$ and $T_{c,a}$.
 \end{proof}

\subsection{Intersection of attracting components} 

\begin{pro} \label{common-point} 
 Let $(c,a)\in\mathcal{S}_p$, and $V_1, V_2 \in\mathcal{B}_{c,a}$ with $V_1\neq V_2$.
 % be two different components of $f_{c,a}$. 
  If $\partial V_1\cap \partial V_2\neq \emptyset$, then   $\partial V_1\cap \partial V_2$
 is a singleton $\{q\}$, satisfying that $f_{c,a}^p(q)=q$. 
\end{pro}

\begin{proof} Let $U$ be the unbounded component of $\mathbb{C}-\overline{V_1}\cup\overline{V_2}$, and $B=\mathbb{C}-\overline{U}$.
Clearly, $V_1\cup V_2\subset B$.
If $\partial V_1\cap \partial V_2$ contains at least two points, then $B\cap J(f_{c,a})\neq \emptyset$.
Let's consider the iterations $\{f_{c,a}^n|_{B}\}_{n\geq 1}$. Note that the iterations $\{f_{c,a}^n|_{\partial B}\}_{n\geq 1}$ is uniformly bounded. By the maximum principle, the 
iterations $\{f_{c,a}^n|_{B}\}_{n\geq 1}$ is uniformly bounded too. By Montel's theorem,  $\{f_{c,a}^n|_{B}\}_{n\geq 1}$ is a normal family, implying that 
$B$ is contained in the Fatou set. This contradicts  $B\cap J(f_{c,a})\neq \emptyset$. 

If $q\in \partial V_1\cap \partial V_2$, then obviously $f_{c,a}^p(q)\in \partial V_1\cap \partial V_2$. Above argument shows that $\partial V_1\cap \partial V_2$ consists of a singleton, implying that  $f_{c,a}^p(q)=q$.
\end{proof}
 
  \begin{rmk} \label{common-l} Assume $-c\notin \mathcal{A}_{c,a}$, and $\partial V_1\cap \partial V_2=\{q\}$
  for $V_1, V_2 \in\mathcal{B}_{c,a}$. Then $q$ is the common landing point of the internal rays $R_{c,a}^{V_1}(0)$ and $R_{c,a}^{V_2}(0)$.
    \end{rmk}

  \section{Branner-Hubbard-Yoccoz Puzzle}\label{BHY-puzzle}

In this section, we first introduce some basic definitions for the Branner-Hubbard-Yoccoz Puzzle theory. Among these definitions, the most important one (for this section) is the {\it admissible puzzle}. The main result here is to show the existence of admissible puzzles for
 maps on $\mathcal{S}_p$ (Theorem \ref{adm-puzzle}).

\subsection{Definitions} \label{definition}

 Let $X,X'$ be open subsets of $\mathbb{\widehat{C}}$, each 
is bounded by finitely many Jordan curves, such that $X'\Subset X\neq \mathbb{\widehat{C}}$. A proper holomorphic
map $f : X'\rightarrow X$ is called a {\it rational-like map}.
We denote by ${\rm deg}(f)$ the topological degree
of $f$ and by $K(f)=\bigcap_{n\geq 0}
f^{-n}(X)$ the
 {\it filled Julia set}, by $J(f)=\partial K(f)$ the {\it Julia set}. The set of critical points on $K(f)$ is denoted by $C(f)$.
  A rational-like map $f : X'\rightarrow X$  is called {\it polynomial-like} if $X,X'$ are Jordan disks and $K(f)$ is connected.
 
A finite graph $\Gamma\subset \overline{X}$ is called a {\it puzzle} of $f$ if it
satisfies the conditions:  $\partial X\subset \Gamma$, $f(\Gamma\cap X')\subset \Gamma$,  and the orbit of each critical point  of $f$ avoids $\Gamma$.

The {\it puzzle pieces} $P_n$ of depth $n\geq0$ are the connected components of $f^{-n}(X\setminus \Gamma)$,
and the one containing the point $z$ is denoted by $P_n(z)$. 
  Let  $$\Gamma_\infty=\bigcup_{k\geq 0}f^{-k}(\Gamma).$$

For any  $z\in J(f)-\Gamma_\infty$, the puzzle piece  $P_n(z)$ is well defined for all $n\geq 0$. In this case, let $P_n^*(z)=\overline{P_n(z)}$. For  $z\in J(f)\cap\Gamma_\infty$, let $P_n^*(z)=\bigcup \overline{P_n}$, where the union is taken for those $P_n$'s satisfying that $z\in\partial P_n$.
The  {\it impression}  ${\rm Imp}(z)$  of $z$   is defined by
\begin{equation*}
{\rm Imp}(z)=
 \bigcap_{n\geq0} {P^*_n(z)}.
\end{equation*}

For any $z\in J(f)-\Gamma_\infty$, the {\it tableau} $T_f(z)$ is the two-dimensional array $(P_{n,l}(z))_{n,l\geq 0}$ with $P_{n,l}(z)=P_{n}(f^l(z))$.
  The tableau $T_f(z)$ is called
{\it periodic} if there is an integer $k\geq 1$ such that $P_n(z)=f^k(P_{n+k}(z))$ for all
$n\geq 0$. Otherwise, $T_f(z)$ is said to be {\it aperiodic}.
For $n,l\geq 0$, we say the position $(n,l)$ of $T_f(z)$ is critical if $P_{n,l}(z)$ contains some critical point $c\in C(f)$ (in this case, we say $(n,l)$ is $c$-critical). We say the tableau $T_f(z)$ is {\it non-critical} if there exists
an integer $n_0\geq0$ such that $(n_0, j)$ is not critical for all $j>0$.  Otherwise $T_f(z)$ is called {\it critical}.

\vspace{5 pt}

All the tableaus satisfy the following two rules \cite{BH2,M5}:

\vspace{4 pt}

\textbf{R1.}  If $P_{n,l}(z)=P_n(z')$, then $P_{i,j}(z)=P_{i,j}(z')$ for all $0\leq i+j\leq n$.

\textbf{R2.}  Let $c\in C(f)$. Assume $T_f(c)$ and $T_f(z)$ satisfy 

(a) $P_{n+1-l,l}(c)=P_{n+1-l}(c')$ for some $c'\in C(f)$ and $n>l>0$,
and $P_{n-i,i}(c)$ contains no critical points for $0<i<l$.

(b) $P_{n,m}(z)=P_n(c)$ and $P_{n+1,m}(z)\neq P_{n+1}(c)$ for some $m>0$.

Then $P_{n+1-l,m+l}(z)\neq P_{n+1-l}(c')$.

\vspace{5 pt}

We say the forward orbit of $x$ {\it combinatorially accumulates} to $y$,
written as $x\xrightarrow{f} y$, if for any $n>0$, there exists $j>0$ such that $y\in P_{n,j}(x)$,
i.e. $f^j(P_{n+j}(x))=P_n(y)$. It is clear that if $x\xrightarrow{f} y$ and $y\xrightarrow{f} z$, then $x\xrightarrow{f} z$.

An aperiodic tableau $T_f(c)$ with  $c\in C(f)$ is said to be  {\it recurrent} if $c\xrightarrow{f} c$.  Otherwise $T_f(c)$ is called
{\it non-recurrent}. 

For two critical puzzle pieces, we say that $P_{n+k}(c')$ is a {\it child} of $P_n(c)$ if
$f^{k-1}: P_{n+k-1}(f(c')) \rightarrow P_n(c)$ is a conformal map.

 Assume that $T_f(c)$ is recurrent, let's define 
 $$[c]_f=\{c'\in C(f); c \xrightarrow{f} c' \text{ and } c' \xrightarrow{f} c \}.$$ 
 We say that $T_f(c)$ is {\it persistently recurrent}
if for any $c_1\in [c]_f$ and any $n\geq 0$, the piece $P_n(c_1)$ has only finitely many children.
Otherwise, $T_f(c)$ is said to be {\it reluctantly recurrent}.

\begin{defi}[Admissible puzzle]
Let $\ell\geq 1$ be an integer,  a puzzle $\Gamma$ is said $\ell$-{\it admissible} for $f$ if it satisfies the conditions: 

(1). for each  $c\in C(f)$, there is an integer $d_c\geq0$ with $\overline{ P_{d_c+\ell}(c)}\subset P_{d_c}(c)$.

(2). all periodic points on $\Gamma\cap J(f)$ are repelling, and

(3). each puzzle piece is a Jordan disk. 
\end{defi}
 
 By definition, an $\ell$-admissible puzzle is always $\ell'$-admissible, where $\ell'\geq \ell$.
 The existence of an admissible puzzle, when combining with analytic techniques,  leads to significant properties of the map $f$ (e.g. local connectivity of Julia set, rigidity, see Section \ref{rigidity-puzzle}). Our task in next subsection is to show the existence of admissible puzzles for most maps on $\mathcal{S}_p$.

  \subsection{Cubic polynomials} \label{cubic-graph}
 Define $\mathcal{C}_0(\mathcal S_p)\subset\mathcal C(\mathcal S_p)$ by
  $$\mathcal{C}_0(\mathcal S_p)=\{(c,a)\in \mathcal C(\mathcal S_p); f^k_{c,a}(-c)\notin  \mathcal A_{c,a} \text{ for any } 
  k\in\mathbb N\}.$$

   Let $f_{c,a}\in \mathcal C_0(\mathcal{S}_p)$ and
  $$X_{c,a}=\mathbb{C}\setminus\Big((G^{\infty}_{c,a})^{-1}([1,\infty))\cup\bigcup_{V\in \mathcal{B}_{c,a}} (G_{c,a}^{V})^{-1}((-\infty, -1])\Big).$$
 Obviously, the set $ X'_{c,a}:=f_{c,a}^{-p}(X_{c,a})$ satisfies $\overline{X'_{c,a}}
\subset X_{c,a}$.
 
 Let $V=U_{c,a}(q)$ for $q\in\{f_{c,a}^k(c); 0\leq k<p\}$, and $\tau$ be the angle doubling map.
   Given a $\tau$-(pre-)periodic angle $\theta$,
  let $\zeta({V,\theta})$ be the landing point of the internal ray $R_{c,a}^V(\theta)$. The point  $\zeta({V,\theta})$  is either (pre-)repelling or (pre-)parabolic and hence it is also the landing point of finitely many external rays (See \cite[Theorems 18.10 and 18.11]{M1}), say
  $R_{c,a}^\infty(\alpha_1), \cdots, R_{c,a}^\infty(\alpha_m)$. If $\zeta({V,\theta})$ is periodic and repelling, then these external rays are all periodic with the same period.

  We define:
  $$\mathbf{R}^q_{c,a}(\theta)=\overline{R_{c,a}^V}(\theta)\cup\overline{R_{c,a}^\infty}(\alpha_1)\cup\cdots \cup\overline{R_{c,a}^\infty}(\alpha_m),$$
  $$\gamma_{c,a}(\theta)=\bigcup_{k\geq 0}f_{c,a}^k(\mathbf{R}^q_{c,a}(\theta)).$$

  Clearly, 
   when $\theta$ is $\tau$-periodic, the graph $\gamma_{c,a}(\theta)$ satisfies $f_{c,a}(\gamma_{c,a}(\theta))=\gamma_{c,a}(\theta)$. 
 Given two rational angles $\theta_1\neq \theta_2$, let $S_q(\theta_1,\theta_2)$ be the component of 
 $\mathbb{C}\setminus (\mathbf{R}^q_{c,a}(\theta_1)\cup\mathbf{R}^q_{c,a}(\theta_2))$ containing the internal rays  $R_{c,a}^V(t)$
 with $\theta_1\leq t\leq \theta_2$.  Let $S^*_q(\theta_1,\theta_2)=\mathbb{C}\setminus S_q(\theta_2,\theta_1)$.
 Clearly, $S^*_q(\theta_1,\theta_2)$ is a closed set containing $S_q(\theta_1,\theta_2)$. See Figure 5.

 \begin{figure}[h]
\begin{center}
\includegraphics[height=6cm]{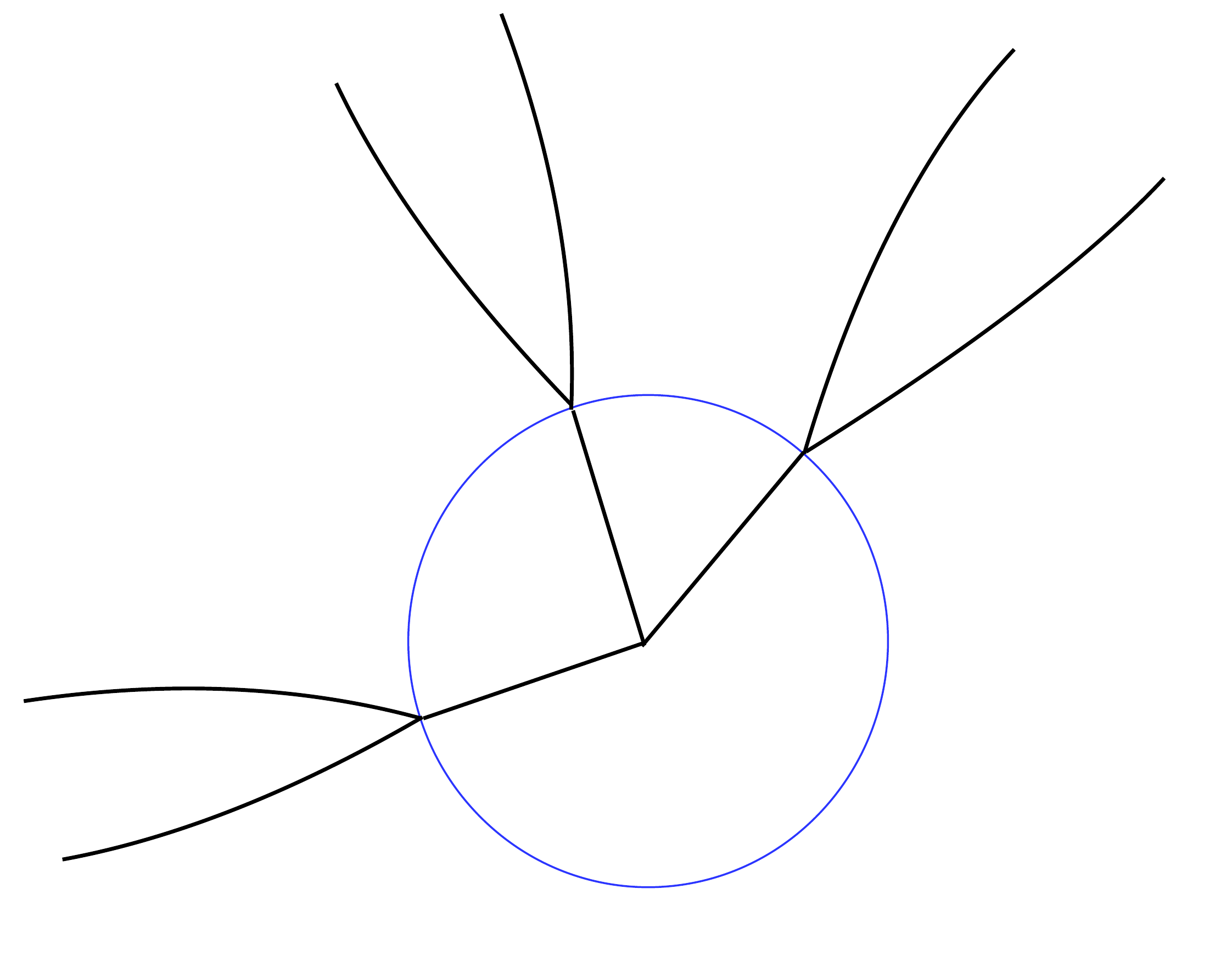}\put(-137,55){$\frac{4}{7}$}
 \put(-80,75){$\frac{1}{7}$} \put(-122,76){$\frac{2}{7}$}  \put(-135,123){$Q_1$}\put(-55,123){$Q_2$}  \put(-100,123){$Q$} 
% \put(-180,80){$S_q(\frac{2}{7},\frac{4}{7})$}  \put(-50,50){$S_q(\frac{4}{7},\frac{1}{7})$}
  \put(-107,28){$V$}\put(-105,49){$q=c$}
 \caption{The graph $\gamma_{c,a}(\frac{1}{7})$ in the case $p=1$. The regions $Q,Q_1,Q_2$ are three components of $\mathbb C-\gamma_{c,a}(\frac{1}{7})$.
 Note that $S_q(\frac{1}{7},\frac{2}{7})=Q$, $S^*_q(\frac{1}{7},\frac{2}{7})=\overline{Q}\cup\overline{Q}_1\cup \overline{Q}_2$.
 %The sectors $S_q(\frac{1}{7},\frac{2}{7}), S_q(\frac{2}{7},\frac{4}{7}),S_q(\frac{4}{7},\frac{1}{7})$ are labelled. 
 Here $V=U_{c,a}(q)$. }
\end{center}\label{f5}
\end{figure}

   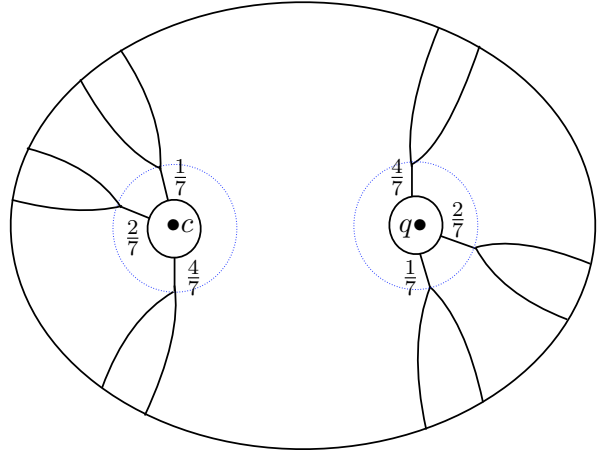
\begin{figure}[h]
\begin{center}
\includegraphics[height=6.5cm]{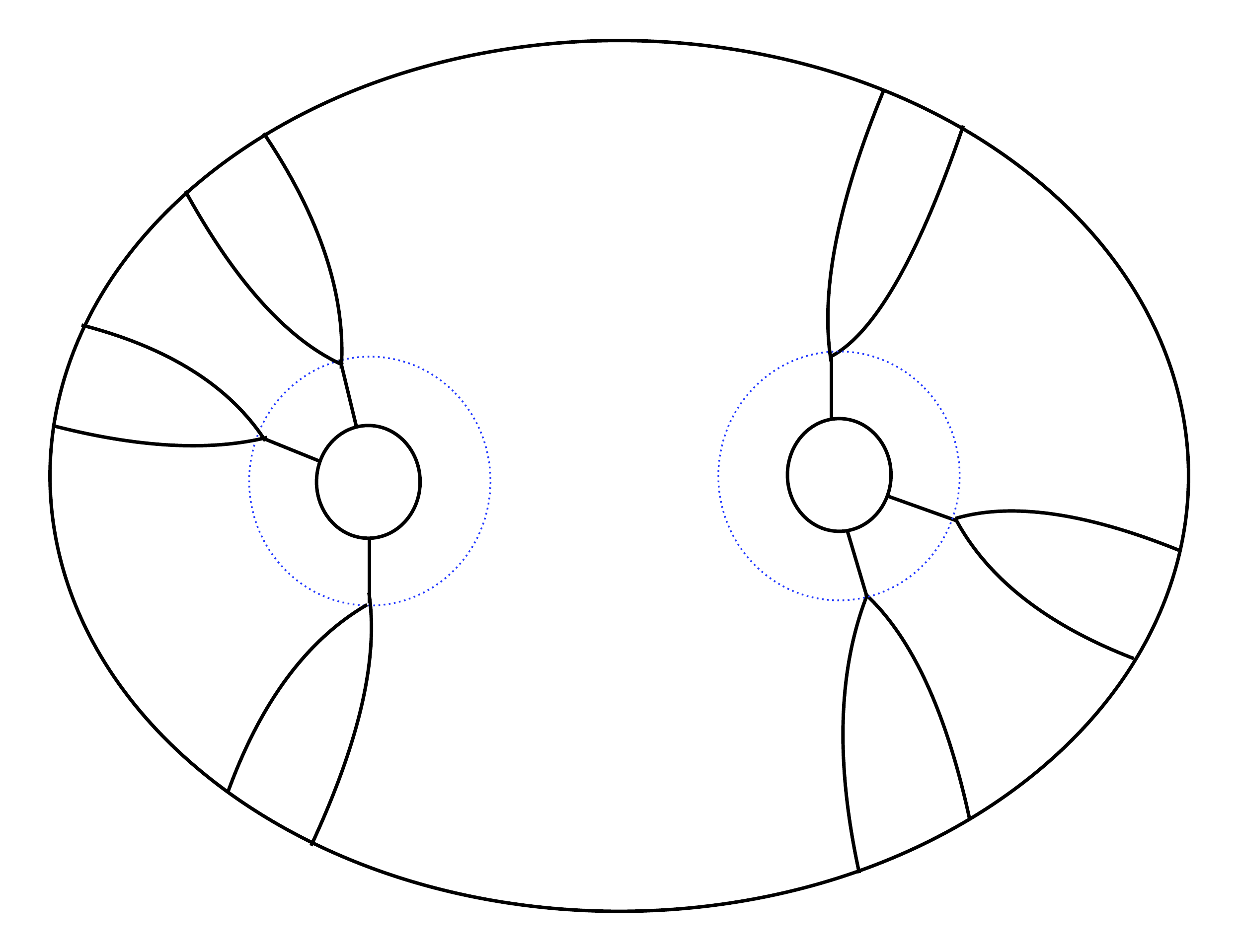} \put(-165,70){$\frac{4}{7}$}  \put(-188,85){$\frac{2}{7}$}  \put(-170,108){$\frac{1}{7}$} 
 \put(-83,70){$\frac{1}{7}$} \put(-88,108){$\frac{4}{7}$} \put(-65,92){$\frac{2}{7}$}  \put(-172,90){$\bullet c$}\put(-84,90){$q\bullet$}  
 \caption{A possible structure of  $\Gamma_{c,a}(\frac{1}{7})$ in the case $p=2$, here $q=f_{c,a}(c)$. }
\end{center}\label{f5}
\end{figure}

   The graph $\Gamma_{c,a}(\theta)$ induced by $\theta$ is defined as follows (see Figure 6):
  $$\Gamma_{c,a}(\theta)=\partial X_{c,a}\cup(X_{c,a}\cap \gamma_{c,a}(\theta)).$$

   The following is the main result of this section.
  
    \begin{thm}\label{adm-puzzle}
Any map $f_{c,a}\in \mathcal C_0(\mathcal{S}_p)$ admits a $p$-admissible puzzle $\Gamma$.
 In fact, at least one of the graphs
$$\Gamma_{c,a}\Big(\frac{1}{7}\Big), \ \Gamma_{c,a}\Big(\frac{3}{7}\Big), \ \Gamma_{c,a}\Big(\frac{1}{7}\Big)\cup\Gamma_{c,a}\Big(\frac{3}{7}\Big)$$
 is a $p$-admissible puzzle.
 \end{thm}

 Before the proof, we explain our strategy.

 \vspace{4pt}

{\it Equivalent statement and strategy.}   
 Note that for any $k\in\mathbb N$ and any graph $\Gamma\in\{\Gamma_{c,a}(\frac{1}{7}), \Gamma_{c,a}(\frac{3}{7}), \Gamma_{c,a}(\frac{1}{7})\cup\Gamma_{c,a}(\frac{3}{7})\}$, the set $f_{c,a}^{-k}(\Gamma)$ is connected (because the pre-images of external rays are external rays,  connecting the rest parts of $f_{c,a}^{-k}(\Gamma)$), so each component of
 $f_{c,a}^{-k}(X_{c,a}\setminus\Gamma)$ is a Jordan disk.
 
 Our goal is to
 show that {\it there is a puzzle $\Gamma$ among the three candidate graphs, with the property that there is a component $Q$ of $f_{c,a}^{-p}(X_{c,a}\setminus\Gamma)$, a component $P$ of $X_{c,a}\setminus\Gamma$, satisfying that}
 $${\rm orb}(-c)\cap Q\neq \emptyset, \ \overline{Q}\subset P.$$
 
 Then it's not hard to see that $\Gamma$ is a $p$-admissible puzzle.
  In fact, assume $f_{c,a}^n(-c)\in Q$ for some $n\geq 0$. The puzzle pieces induced by $\Gamma$ satisfy that 
 $\overline{P_{p}(f_{c,a}^n(-c))}\subset P_{0}(f_{c,a}^n(-c))$.
By taking $f_{c,a}^n$-preimages, we see that $\overline{P_{p+n}}(-c)\subset P_{n}(-c)$,
  implying that $\Gamma$ is $p$-admissible.
 
 \vspace{3pt}
 
 The main idea of the proof is to discuss the relative position of the critical orbit with respect to the candidate graphs.
 The argument has some independent interest.
 We first treat the case $p=1$ to illustrate the idea, then deal with the more delicate case $p>1$.

 \vspace{8pt}

{\it Proof of Theorem \ref{adm-puzzle} when $p=1$.} 
 First note that 
 $$\tau^{-1}\Big(\Big\{\frac{1}{7},\frac{2}{7},\frac{4}{7}\Big\}\Big)=\Big\{\frac{1}{7},\frac{2}{7},\frac{4}{7}\Big\}\cup\Big\{\frac{1}{14},\frac{9}{14},\frac{11}{14}\Big\},$$
  $$\tau^{-1}\Big(\Big\{\frac{3}{7},\frac{5}{7},\frac{6}{7}\Big\}\Big)=\Big\{\frac{3}{7},\frac{5}{7},\frac{6}{7}\Big\}\cup\Big\{\frac{3}{14},\frac{5}{14},\frac{13}{14}\Big\}.$$

 We first  assume that the graph $\Gamma=\Gamma_{c,a}(\frac{1}{7})\cup\Gamma_{c,a}(\frac{3}{7})$ avoids the orbit of $-c$, and   $\Gamma\cap J(f_{c,a})$ contains no parabolic point. Figure 7 will be helpful to understand the proof. We first assert that either we are in
  {\it Case 1:}
 $${\rm orb}(-c)\cap S^*_c\Big(\frac{9}{14},\frac{5}{14}\Big)\neq \emptyset,$$
 or in {\it Case 2:} we can find a component $Q$ of $f_{c,a}^{-1}(X_{c,a}\setminus \Gamma)$,  a component $P$ of $X_{c,a}\setminus\Gamma$, satisfying that
 $${\rm orb}(-c)\cap Q\neq \emptyset, \ \overline{Q}\subset P.$$

  To see this, note that $\mathbb{C}=S_c(\frac{5}{14},\frac{9}{14})\cup S^*_c(\frac{9}{14},\frac{5}{14})$. If ${\rm orb}(-c)\cap S^*_c(\frac{9}{14},\frac{5}{14})=\emptyset$, then we have ${\rm orb}(-c)\subset S_c(\frac{5}{14},\frac{9}{14})$. Let $Q$ be a component of 
  $f_{c,a}^{-1}(X_{c,a}\setminus \Gamma)$ so that $Q\cap {\rm orb}(-c)\neq \emptyset$, and let $P$ be the component of $X_{c,a}\setminus \Gamma$ that contains $Q$. 
Clearly $Q\subset  S_c(\frac{5}{14},\frac{9}{14})$.  We claim $\overline{Q}\subset P$. 
In fact, if it is not true, then
 $\partial Q\cap\overline{U_{c,a}(c)}\neq\emptyset$. This would imply that $f_{c,a}(Q)\subset  S^*_c(\frac{9}{14},\frac{5}{14})$, and further ${\rm orb}(-c)\cap S^*_c(\frac{9}{14},\frac{5}{14})\neq\emptyset$, leading to a contradiction.  This shows that we are in Case 2, in which case $\Gamma$ is $1$-admissible and the proof is done.

  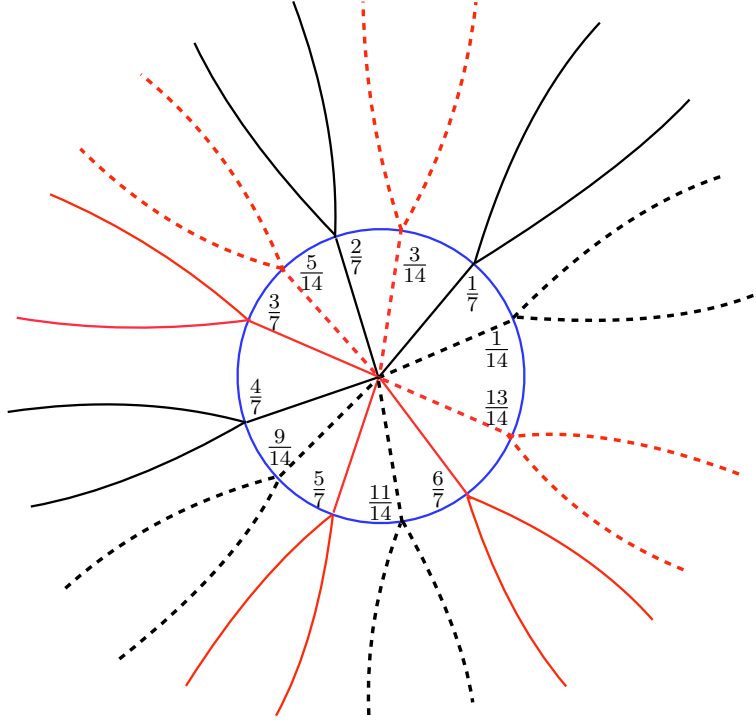
\begin{figure}[h]
\begin{center}
\includegraphics[height=10cm]{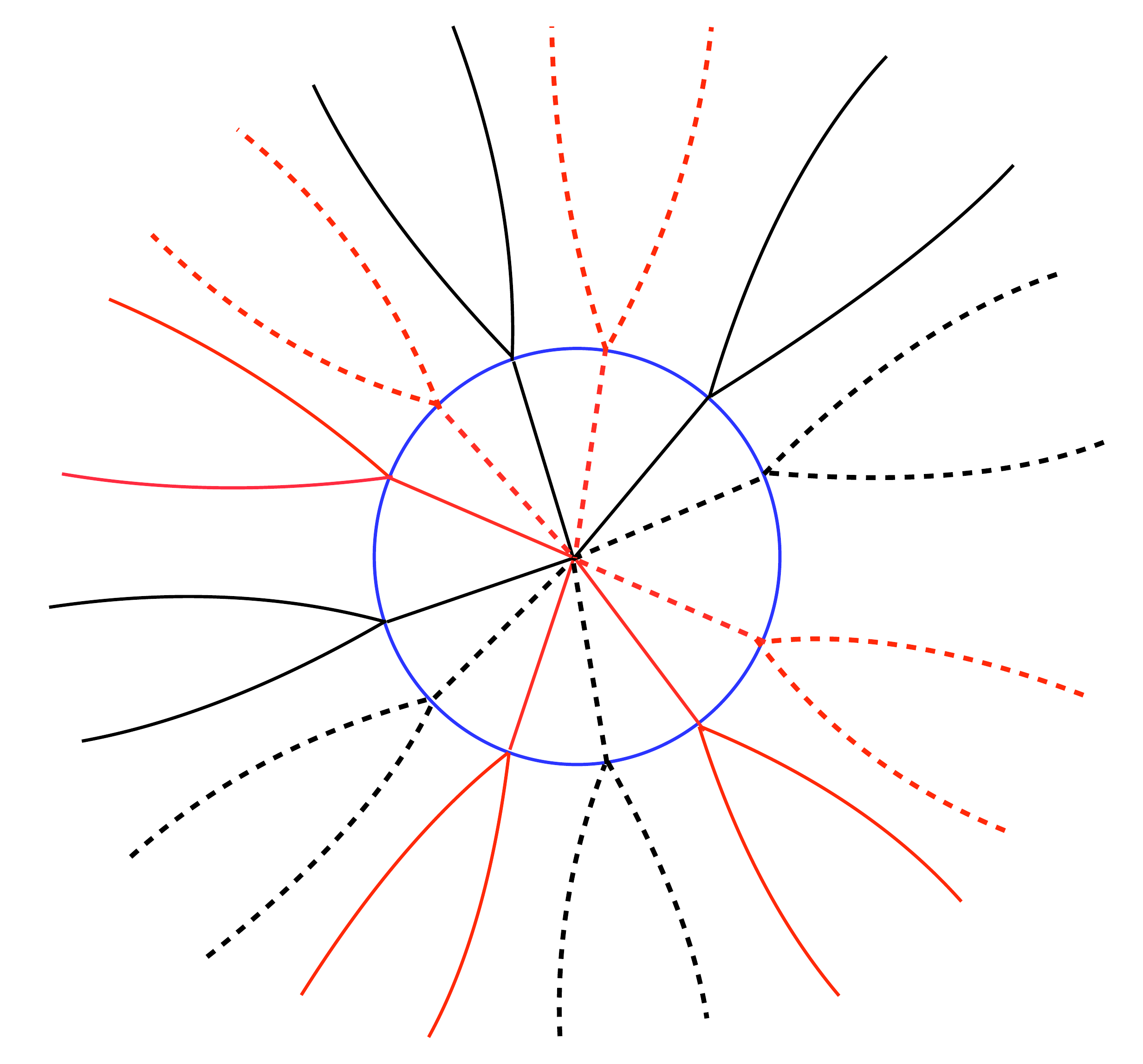}
\put(-113,121){$\frac{13}{14}$}  \put(-113,143){$\frac{1}{14}$} \put(-120,164){$\frac{1}{7}$} \put(-144,175){$\frac{3}{14}$}
 \put(-164,178){$\frac{2}{7}$} \put(-183,172){$\frac{5}{14}$} \put(-195,157){$\frac{3}{7}$} \put(-202,125){$\frac{4}{7}$}
 \put(-195,106){$\frac{9}{14}$} \put(-178,90){$\frac{5}{7}$}   \put(-157,85){$\frac{11}{14}$} \put(-133,90){$\frac{6}{7}$}
 \caption{The graphs $\gamma_{c,a}(\frac{1}{7})$ (black), $\gamma_{c,a}(\frac{3}{7})$ (red) and their first $f^p_{c,a}$-preimages (dashed black, dashed red, respectively) near $c$. Here $U_{c,a}(c)$ is bounded by blue curve.}
\end{center}\label{f5}
\end{figure}

  We need further discuss  
  Case 1. In this case, we have 
  $$f_{c,a}^k(-c)\in S^*_c\Big(\frac{9}{14},\frac{1}{14}\Big)\cup S^*_c\Big(\frac{1}{14},\frac{5}{14}\Big)$$
  for some integer $k\geq 0$.
  If $f_{c,a}^k(-c)\in S^*_c(\frac{9}{14},\frac{1}{14})$, then let's consider the graph $\Gamma_{c,a}(\frac{1}{7})$. There is a component $Q$ of   $f_{c,a}^{-1}(X_{c,a}\setminus \Gamma_{c,a}(\frac{1}{7}))$ and a component $P$ of $X_{c,a}\setminus \Gamma_{c,a}(\frac{1}{7})$ satisfying that 
  $$f_{c,a}^k(-c)\in Q\subset P.$$
  In this case, the boundaries of $Q$ and $P$ will not touch (see Figure), therefore $\overline{Q}\subset P$,
  implying that  $\Gamma_{c,a}(\frac{1}{7})$ is 1-admissible.
  If $f_{c,a}^k(-c)\in S^*_c(\frac{1}{14},\frac{5}{14})$, with the similar argument, we  see that the graph 
  $\Gamma_{c,a}(\frac{3}{7})$ is 1-admissible.

 Finally, we treat the rest cases. Write $V=U_{c,a}(c)$.
    
If $\Gamma_{c,a}(\frac{1}{7})\cap {\rm orb}(-c)\neq \emptyset$ or $\zeta(V,{\frac{1}{7}})$ (recall that it is the landing point of $R_{c,a}^V(\frac{1}{7})$) is parabolic,
 then $\Gamma_{c,a}(\frac{3}{7})$ is a puzzle.
 In the former case, one has $\zeta(V,{\frac{1}{7}})=f_{c,a}^{k}(-c)$ for some $k\geq 0$. In the latter case, 
 let $W$ be a parabolic basin so that $\zeta(V,{\frac{1}{7}})\in\partial W\cap \partial V$, then  $f_{c,a}^{k}(-c)\in W$ for some $k\geq 0$.  
  In either case,
  % $\Gamma_{c,a}(\frac{3}{7})\cap {\rm orb}(-c)=\emptyset$.
%  the critical point $-c$ is in a parabolic basin, whose boundary touches 
% $\partial V$ exactly at one of the points $\zeta(V,{\frac{1}{7}}), \zeta(V,{\frac{2}{7}}), \zeta(V,{\frac{4}{7}})$.
%One may see from Figure ?? that 
there is a component of $f_{c,a}^{-1}(X_{c,a}\setminus \Gamma_{c,a}(\frac{3}{7}))$, say 
 $Q$, containing $f_{c,a}^{k}(-c)$, and a component $P$ of $X_{c,a}\setminus \Gamma_{c,a}(\frac{3}{7})$ containing $\overline{Q}$. 
  This implies that $\Gamma_{c,a}(\frac{3}{7})$ is 1-admissible.

 If $\Gamma_{c,a}(\frac{3}{7})\cap {\rm orb}(-c)\neq \emptyset$ or $\zeta(V,{\frac{3}{7}})$  is parabolic, by similar argument as above, we see that $\Gamma_{c,a}(\frac{1}{7})$ is 1-admissible.

  This completes the proof in the case $p=1$.   
  \hfill\fbox
  
  \vspace{6pt}
  
  To deal with the case $p>1$, we first prove the following  fact:
  
  \begin{lem} \label{compact-in} Assume $p>1$.
   Let $q=f_{c,a}^k(c)$ for some $0\leq k<p$ and let $\Gamma$ be one of the graphs
  $$\Gamma_{c,a}\Big(\frac{1}{7}\Big),\ \Gamma_{c,a}\Big(\frac{3}{7}\Big), \ 
  \Gamma_{c,a}\Big(\frac{1}{7}\Big)\cup\Gamma_{c,a}\Big(\frac{3}{7}\Big).$$ 
  Let $Q$ be a component of $f_{c,a}^{-p}(X_{c,a}\setminus\Gamma)$.  
    Suppose that ${Q}\subset f^p_{c,a}(Q)$ and $\overline{Q} \cap \overline{U_{c,a}(q)}\neq \emptyset$.
  Then we have 
  $$\overline{Q}\subset f^p_{c,a}(Q).$$
  \end{lem}
  
  Before the proof, we need a notation. 
  For any integer $k\geq 0$, any component $P$ of $f_{c,a}^{-k}(X_{c,a}\setminus\Gamma)$, 
  let $\nu(P)$ be the number of attracting basins $V\in \mathcal B_{c,a}$ whose boundary touches $\overline{P}$:
  $$\nu(P)=|\{V\in \mathcal B_{c,a}; \ \overline{P}\cap \overline{V}\neq\emptyset\}|.$$
  
   \begin{figure}[h]
\begin{center}
\includegraphics[height=6.5cm]{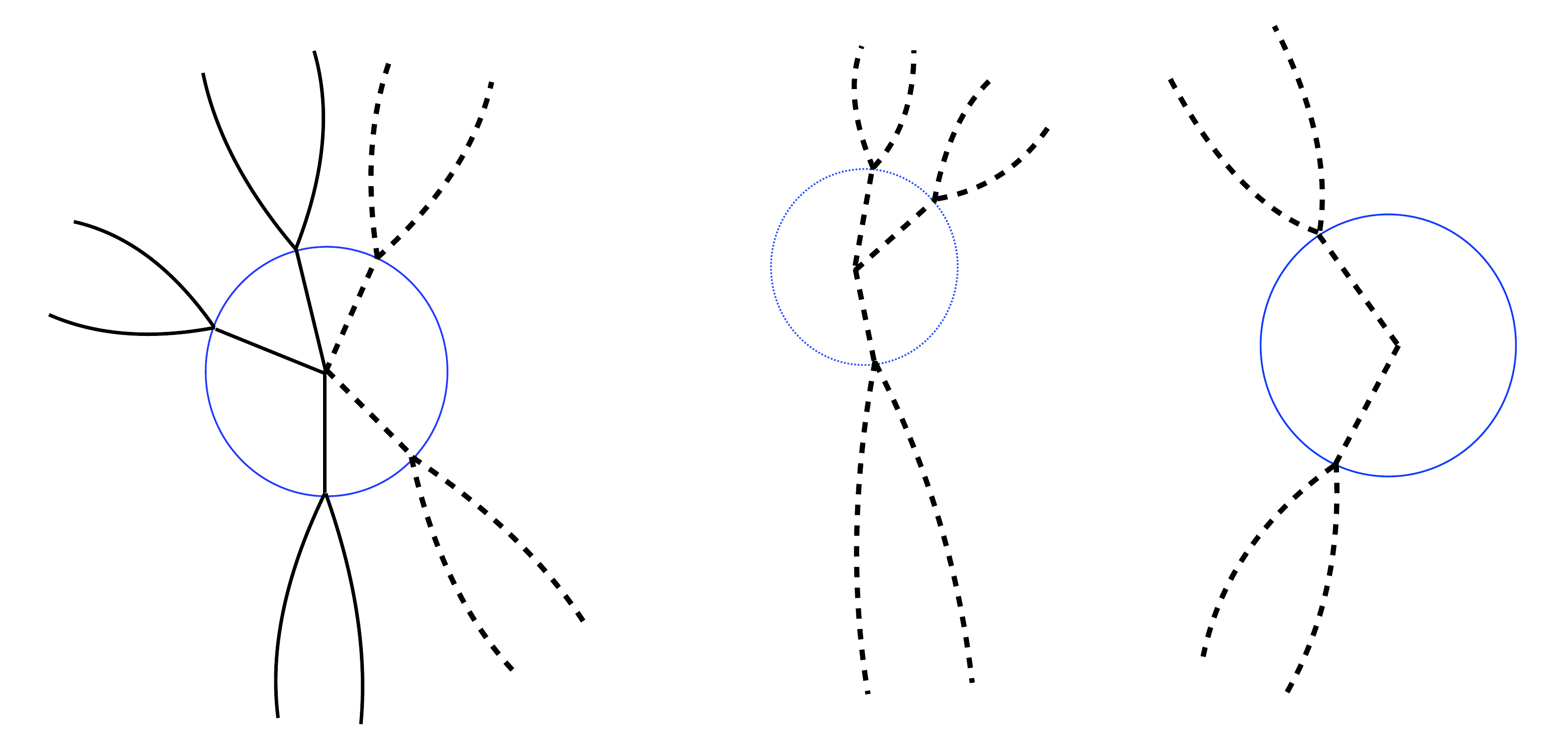}
\put(-315,88){$q$} \put(-310,115){$\frac{1}{7}$} \put(-295,111){$\frac{1}{14}$}
  \put(-305,72){$\frac{4}{7}$} \put(-290,82){$\frac{11}{14}$}  \put(-232,100){$Q$}
   \put(-270,42){$Q_1$}  \put(-170,42){$Q_2$}

 \caption{A possible graph structure, $q=f^k_{c,a}(c)$.  Here $\nu(Q)=\nu(Q_1)=1, \ \nu(Q_2)=0$.
 Dashed rays are first pre-image of the dynamical rays defining the graph. Equipotential curves are not included.}
\end{center}\label{f5}
\end{figure}

  \begin{figure}[h]
\begin{center}
\includegraphics[height=7cm]{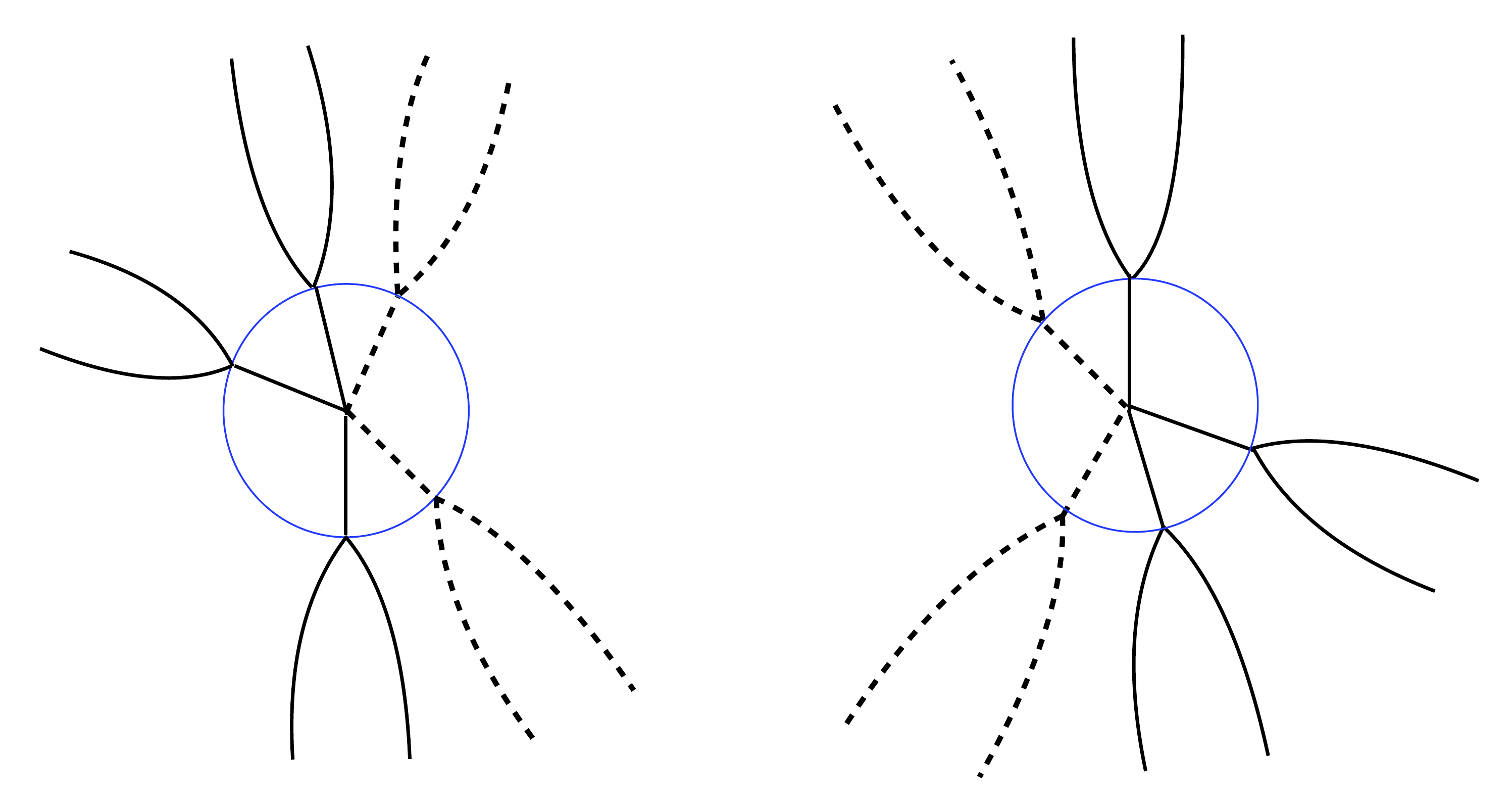}
\put(-93,119){$\frac{4}{7}$} \put(-92,100){$q'$} \put(-111,117){$\frac{11}{14}$} \put(-105,74){$\frac{1}{14}$} 
\put(-85,74){$\frac{1}{7}$} \put(-295,93){$q$} \put(-290,117){$\frac{1}{7}$} \put(-277,114){$\frac{1}{14}$}
  \put(-295,75){$\frac{4}{7}$} \put(-282,75){$\frac{11}{14}$}  \put(-192,100){$Q$}
 \caption{A possible graph structure, $q\neq q'$ and $q=f^k_{c,a}(c)$, $q'=f^l_{c,a}(c)$.
 Here $\nu(Q)=2$. Dashed rays are first pre-image of the dynamical rays defining the graph. Equipotential curves are not included. 
 }
\end{center}\label{f5}
\end{figure}

  \vspace{6pt}
 
 {\it Proof of Lemma \ref{compact-in}.}  
 By the graph structure, the assumption $\overline{Q} \cap \overline{U_{c,a}(q)}\neq \emptyset$ implies that either
  $\overline{Q} \cap \overline{U_{c,a}(q)}$ is a singleton or $Q\cap U_{c,a}(q)\neq \emptyset$. 
  Another assumption ${Q}\subset f^p_{c,a}(Q)$ implies that the former case is impossible. To see this, suppose $\overline{Q} \cap \overline{U_{c,a}(q)}=\{\zeta\}$, then 
  $\overline{f^p_{c,a}(Q)}\cap \overline{U_{c,a}(q)}$ is also a singleton, say $\{\zeta'\}$. 
  From ${Q}\subset f^p_{c,a}(Q)$, we see that $\zeta=\zeta'$, equivalently $f_{c,a}^p(\zeta)=\zeta$. But this is a contradiction, because $\zeta$ is of period three under the map $f_{c,a}^p$.
  
  Now we know that $Q\cap U_{c,a}(q)\neq \emptyset$, which implies that $ f^p_{c,a}(Q)\cap U_{c,a}(q)\neq \emptyset$.
  Note that $f^p_{c,a}(Q)\subset S_q(\alpha_q, \beta_q)$, where $(\alpha_q, \beta_q)$ has three choices 
  $(\frac{1}{7},\frac{2}{7})$, $(\frac{2}{7},\frac{4}{7})$, $(\frac{4}{7},\frac{1}{7})$ if $\Gamma=\Gamma_{c,a}(\frac{1}{7})$; three choices $(\frac{3}{7},\frac{5}{7}), (\frac{5}{7},\frac{6}{7}), (\frac{6}{7},\frac{3}{7})$ if $\Gamma=\Gamma_{c,a}(\frac{3}{7})$; six choices $(\frac{1}{7},\frac{2}{7}), (\frac{2}{7},\frac{3}{7}), (\frac{3}{7},\frac{4}{7}), (\frac{4}{7},\frac{5}{7}), (\frac{5}{7},\frac{6}{7}), (\frac{6}{7},\frac{1}{7})$
  if $\Gamma=\Gamma_{c,a}(\frac{1}{7})\cup \Gamma_{c,a}(\frac{3}{7})$.

 The fact $Q\subset f^p_{c,a}(Q)$  implies that there is a component $C$ of $\tau^{-1}((\alpha_q, \beta_q))$ that is contained in $(\alpha_q, \beta_q)$. This situation can happen only if 
   $$(\alpha_q, \beta_q)=\begin{cases} (\frac{4}{7}, \frac{1}{7}), &\text{ if } \Gamma=\Gamma_{c,a}(\frac{1}{7}), \\
 (\frac{6}{7}, \frac{3}{7}), &\text{ if } \Gamma=\Gamma_{c,a}(\frac{3}{7}), \\
  (\frac{6}{7}, \frac{1}{7}), &\text{ if } \Gamma=\Gamma_{c,a}(\frac{1}{7})\cup\Gamma_{c,a}(\frac{3}{7}).
\end{cases}$$
  In either case, one may verify that $\overline{C}\subset(\alpha_q, \beta_q)$. This implies that if $\nu(Q)=1$, then $\overline{Q}\subset f^p_{c,a}(Q)$ and the proof is done (see Figure 8).
  
  If $\nu(Q)\geq 2$ (see Figure 9), then for any $q'=f_{c,a}^l(c)$ with $q'\neq q$ and $\overline{Q} \cap \overline{U_{c,a}(q')}\neq \emptyset$,  
by the same argument as above, we see that $Q\cap U_{c,a}(q')\neq\emptyset$ and $\overline{Q\cap U_{c,a}(q')}\subset f_{c,a}^p(Q)$. Therefore in this case,
 $\overline{Q}\subset f^p_{c,a}(Q).$
 % The proof is completed.
\hfill\fbox

\vspace{8pt}

{\it Proof of Theorem \ref{adm-puzzle} when $p>1$.} Let $\Gamma=\Gamma_{c,a}(\frac{1}{7})\cup \Gamma_{c,a}(\frac{3}{7})$.
We first assume that $\Gamma\cap {\rm orb}(-c)=\emptyset$  and $\Gamma\cap J(f_{c,a})$ contains no parabolic point.

Let $Q$ be a component of $f_{c,a}^{-p}(X_{c,a}\setminus\Gamma)$ so that ${\rm orb}(-c)\cap Q\neq \emptyset$, and let
$P$ be the component of $X_{c,a}\setminus\Gamma$ containing $Q$.

\textbf{Case 1. $\nu(Q)=0$.} This is equivalent to say that $\partial Q\cap \partial \mathcal{A}_{c,a}=\emptyset$.
% has no intersection with any attracting basin of the form $U_{c,a}(f_{c,a}^k(c))$, 
Therefore $\partial Q \subset f_{c,a}^{-p}(\Gamma)\setminus \Gamma$ and $\overline{Q}\subset P$.

\textbf{Case 2. $\nu(Q)\geq 2$.}  Note that $\overline{V}\cap\overline{Q}\neq \emptyset
\Longrightarrow  \overline{V}\cap\overline{f_{c,a}^p(Q)}\neq \emptyset$ for any $V\in \mathcal B_{c,a}$.
Therefore $\nu(f_{c,a}^p(Q))\geq \nu(Q)\geq 2$. 
Clearly, 
%for any $V\in\mathcal B_{c,a}$ with $\overline{V}\cap \overline{P}\neq \emptyset$, one has 
% $\overline{V}\cap \overline{f_{c,a}^p(Q)}\neq \emptyset$.
both $\overline{f_{c,a}^p(Q)}$ and $\overline{P}$   intersect those $\overline{V}$'s with 
$V\in\mathcal B_{c,a}$ and $\overline{V}\cap \overline{Q}\neq\emptyset$.
%closure of the basins $\overline{U_{c,a}(f_{c,a}^k(c))}$ that 
%intersects $\overline{Q}$. 
Since there is only one component of $X_{c,a}\setminus\Gamma$ satisfying this property, we have $f_{c,a}^p(Q)=P$. Then 
by Lemma \ref{compact-in}, we get $\overline{Q}\subset P$.
 
\textbf{Case 3. $\nu(Q)=1$.} Let $q\in \{f_{c,a}^k(c); 0\leq k<p\}$ be the unique point with $\overline{Q}\cap\overline{U_{c,a}(q)}\neq\emptyset$. Note that
$$S^*_q\Big(\frac{9}{14},\frac{1}{14}\Big)\cup S^*_q\Big(\frac{1}{14},\frac{5}{14}\Big)\cup S_q\Big(\frac{5}{14},\frac{9}{14}\Big)=\mathbb{C}.$$

\textbf{Case 3.1. $Q\subset S^*_q(\frac{9}{14},\frac{1}{14})$.} Let $Q'$ be the component of
$f_{c,a}^{-p}(X_{c,a}\setminus \Gamma_{c,a}(\frac{1}{7}))$ containing $Q$, and $P'$ the component of
$X_{c,a}\setminus \Gamma_{c,a}(\frac{1}{7})$ containing $Q'$. Clearly, 
$$\nu(P')\geq \nu(Q')\geq \nu(Q)=1.$$
If $\nu(Q')\geq 2$, then by the same argument as  Case 2, we have  $\overline{Q'}\subset P'$.
If   $\nu(Q')=1$, the fact $\overline{Q'\cap U_{c,a}(q)}\subset P'$ (see Figure 8) implies that $\overline{Q'}\subset P'$.
In either case, the graph $\Gamma_{c,a}(\frac{1}{7})$ is a $p$-admissible puzzle.

\textbf{Case 3.2. $Q\subset S^*_q(\frac{1}{14},\frac{5}{14})$.} Let $Q'$ be the component of
$f_{c,a}^{-p}(X_{c,a}\setminus \Gamma_{c,a}(\frac{3}{7}))$ containing $Q$, and $P'$ the component of
$X_{c,a}\setminus \Gamma_{c,a}(\frac{3}{7})$ containing $Q'$. 
Similarly % With the same argument as in
as Case 3.1, we have   $\overline{Q'}\subset P'$, implying that $\Gamma_{c,a}(\frac{3}{7})$ is a $p$-admissible puzzle.

\textbf{Case 3.3. $Q\subset S_q(\frac{5}{14},\frac{9}{14})$.} In this case, $f_{c,a}^p(Q)$ is a component of $X_{c,a}\setminus\Gamma$ 
satisfying that
$$f_{c,a}^p(Q)\subset S^*_q\Big(\frac{9}{14},\frac{1}{14}\Big)\cup S^*_q\Big(\frac{1}{14},\frac{5}{14}\Big) \text{ and } \nu(f_{c,a}^p(Q))\geq 1.$$
Note that ${\rm orb}(-c)\cap f_{c,a}^p(Q)\neq \emptyset$. There is a component $Q''$ of  $f_{c,a}^{-p}(X_{c,a}\setminus \Gamma)$
satisfying that $Q''\subset f_{c,a}^p(Q)$ and ${\rm orb}(-c)\cap Q''\neq \emptyset$. 
Clearly either $Q''\subset S^*_q(\frac{9}{14},\frac{1}{14})$ or $Q''\subset S^*_q(\frac{1}{14},\frac{5}{14})$,
%of  $f_{c,a}^{-p}(X_{c,a}\setminus \Gamma)$ with
%${\rm orb}(-c)\cap Q''\neq \emptyset$ and $Q''\subset S^*_p(\frac{9}{14},\frac{1}{14})\cup S^*_p(\frac{1}{14},\frac{5}{14})$. 
meaning that we are again in Case 3.1 or Case 3.2. 
 With the same argument, we see that $\Gamma_{c,a}(\frac{1}{7})$ or $\Gamma_{c,a}(\frac{3}{7})$ is $p$-admissible.
  % as there gives rise to a $p$-admissible puzzle. 
%Therefore either  $\Gamma_{c,a}(\frac{1}{7})$ or  $\Gamma_{c,a}(\frac{3}{7})$ is $p$-admissible.
 
\textbf{Rest Cases.} % (NOT COMPLETED) 
Finally, we handle the rest cases:  $\Gamma\cap {\rm orb}(-c)\neq\emptyset$
 or  $\Gamma\cap J(f_{c,a})$ contains a parabolic point.

Suppose that $\Gamma_{c,a}(\frac{1}{7})\cap {\rm orb}(-c)\neq\emptyset$  or  $\Gamma_{c,a}(\frac{1}{7})\cap J(f_{c,a})$ contains a parabolic point. Let $V=U_{c,a}(c)$, then the landing point $\zeta(V,\frac{1}{7})$ of $R_{c,a}^V(\frac{1}{7})$ is either $f_{c,a}^k(-c)$ for some $k\geq 1$, or parabolic.
 In the latter case,  $\zeta(V,\frac{1}{7})$ is on the boundary of some parabolic basin $W$, which contains $f_{c,a}^k(-c)$ 
 for some $k$ (here we use the same $k$ because the two cases can not happen simultaneously).  In either case, let $Q$ be the component of $f_{c,a}^{-p}(X_{c,a}\setminus \Gamma_{c,a}(\frac{3}{7}))$ containing $f_{c,a}^k(-c)$, and let  $P$ be the component of $X_{c,a}\setminus \Gamma_{c,a}(\frac{3}{7})$ containing ${Q}$.
 The fact ${Q}\cap V\neq \emptyset$ implies that  $\nu(Q)\geq 1$.
 % (note that the case $\nu(Q)=0$  can not happen, because  .)
   If $\nu(Q)\geq 2$, then by the same argument as in Case 2, we have  $\overline{Q}\subset P$.
 If  $\nu(Q)=1$, note that $\overline{Q\cap V}\subset P$, we also have $\overline{Q}\subset P$.
  Therefore $\Gamma_{c,a}(\frac{3}{7})$ is $p$-admissible.

The last cases are $\Gamma_{c,a}(\frac{3}{7})\cap {\rm orb}(-c)\neq\emptyset$  or  $\Gamma_{c,a}(\frac{3}{7})\cap J(f_{c,a})$ contains a parabolic point.
%: $\Gamma_{c,a}(\frac{3}{7})$ is touchable or the intersection $\Gamma_{c,a}(\frac{3}{7})\cap J(f_{c,a})$ contains a parabolic point.
% By the same argument
 Similarly  as above,  we have that $\Gamma_{c,a}(\frac{1}{7})$ is $p$-admissible. 
 
The proof of the theorem is completed.  \hfill\fbox

 \section{Rigidity via puzzles}  \label{rigidity-puzzle}

  This section is devoted to proving the {combinatorial rigidity} for maps on $\mathcal S_p$. 
 Rigidity is one of the most remarkable phenomena in holomorphic dynamics. One of its applications is
  to study the boundaries of hyperbolic components in the next section. 
  To simplify notations, write
  $$f=f_{c,a}, \  \tilde{f}=f_{\tilde{c},\tilde{a}},\ X=X_{c,a}, \ c^*=-c. $$
Here, $X_{c,a}$ is defined in Section \ref{cubic-graph}. The corresponding objects (graphs, puzzles, tableau, etc) for $\tilde{f}$ are marked with tilde.

In this section, we assume $f,\tilde f\in \mathcal{C}_0(\mathcal S_p)$. Let's take a $p$-admissible puzzle $\Gamma$ for $f$, given by Theorem \ref{adm-puzzle}. The puzzle pieces and tableau (in particular $T_{f}(c^*)$) are induced by $\Gamma$. 
Write $\Gamma_k=f^{-k}(\Gamma)$ for $k\geq 0$, the collection $\mathcal P_k$ of puzzle pieces of depth $k$ consists of the connected components of $\mathbf P_{k}:=f^{-k}(X-\Gamma)$. % let $\mathbf P_{k}=\cup_{P\in \mathcal P_k}P$.
 Let $\widetilde \Gamma$ be the graph of $\tilde f$ with the same structure as $\Gamma$.

We first define the {\it combinatorial equivalence} between $f$ and $\tilde f$.
 Roughly speaking, it means that the two maps have the same puzzle structures at any depth. Rigorous definition goes as follows.
Let $\phi: \Gamma\rightarrow\widetilde \Gamma$ be a homeomorphism, written as the identity map in the B\"ottcher coordinates. For an integer $k\geq 1$,
we say that  $f$ and $\tilde f$ have the {\it same combinatorics up to depth $k$}, if there is a homeomorphism $\phi_k: \Gamma_k\rightarrow \widetilde \Gamma_k$  so that $\tilde f\circ \phi_k=\phi\circ f$ on $\Gamma_k$ and $\phi_{k}|_{\Gamma_k\cap \Gamma}=\phi|_{\Gamma_k\cap \Gamma}$.  We say that $f$ and $\tilde f$ are {\it combinatorially equivalent} if they have the same combinatorics up to any depth. If furthermore $\phi$ can be extended to a quasi-conformal map $\Phi:\mathbb C\rightarrow \mathbb C$, we say that $f$ and $\tilde f$ are {\it qc-combinatorially equivalent}. Combinatorial equivalence allows one to extend $\phi$ as a homeomorphism 
$$\phi: \bigcup_{k\geq 0}\Gamma_k\rightarrow \bigcup_{k\geq 0}\widetilde\Gamma_k$$
by setting $\phi|_{\Gamma_k}=\phi_k|_{\Gamma_k}$ for any $k$. Further, $\phi$ induces a bijection $\phi_*$ between puzzle pieces:
$$\phi_*: \bigcup_{k\geq 0}\mathcal P_k\rightarrow \bigcup_{k\geq 0}\widetilde{\mathcal P}_k,$$
here $\phi_*(P_k)$ is defined to be the puzzle piece of $\tilde f$ bounded by $\phi(\partial P_k)$.

  \begin{thm}\label{c-rigidity}
If $f, \tilde f\in \mathcal C_0(\mathcal{S}_p)$ are qc-combinatorially equivalent and $T_{f}(c^*)$  is aperiodic,
 then $f=\tilde f$.
   \end{thm}
The assumption that $T_{f}(c^*)$  is aperiodic implies that $f$ is not renormalizable\footnote{We say that $f$ is {\it renormalizable},  there exist an integer $k\geq 0$, two open disks $U,V$ with $\overline{U}\subset V$, such that $f^k: U\rightarrow V$ is a polynomial-like map of degree $\geq 2$, with connected Julia set, which is not equal to $J(f)$.} and $J(f)=K(f)$\footnote{To see this, note that $K(f)\neq J(f)\Longrightarrow \text{ there is a periodic Fatou component }U\subset K(f)$
$\Longrightarrow \text{ there exists puzzles pieces $P_n\supset P_{n+k}\supset U$}$
$\Longrightarrow {\rm deg}(f^k: P_{n+k}\rightarrow P_{n})\geq 2$ (by Schwarz Lemma)
$\Longrightarrow c^*$ is in the cycle of $U$ $\Longrightarrow T_f(c^*)$ is periodic.}. 
The proof of Theorem \ref{c-rigidity} actually gives more:

  \begin{thm}\label{cubic-property}
Let $f\in \mathcal C_0(\mathcal{S}_p)$,  suppose that $T_{f}(c^*)$  is aperiodic. Then  

(1). The Julia set $J(f)$ is locally connected.

(2). $f$ carries no invariant line fields on $J(f)$.
   \end{thm}
   
Here, a {\it line field} $\mu$ supported on $E$ is a Beltrami differential $\mu=\mu(z)\frac{d\bar z}{dz}$ supported on $E$ with $|\mu|=1$. A line field $\mu$ is called {\it measurable} if $\mu(z)$ is a measurable function. We say that $f$  carries an {\it invariant line field} if there is a measurable line field $\mu=\mu(z)\frac{d\bar{z}}{dz}$ supported on a positive measurable subset of $J(f)$  such that $f^*\mu=\mu$ almost everywhere.

Theorem \ref{c-rigidity} and Theorem \ref{cubic-property} (1) generalize Yoccoz's famous theorem to cubic maps $f\in\mathcal{C}_0(\mathcal S_p)$.
In fact, in the case $p=1$,  Yoccoz's proof of local connectivity \cite{H,M5}, Lyubich's proof of zero measure \cite{L} for quadratic Julia sets both work here.  However, their arguments will break down for cubic $f\in\mathcal{C}_0(\mathcal S_p)$  in the persistently recurrent case when $p\geq 2$.  This is because the existence of a $p$-admissible puzzle with $p\geq 2$ makes the situation essentially as complicated as the multicritical case. The {\it principle nest} of critical puzzle pieces  (see Theorem \ref{KSS}) will be involved to deal with this case.

%
%This ge
% rigidity for quadratics, 
% even if $f$ has a simple critical point $c^*$ and a puzzle structure (similar to the quadratic case).   
%
%
%
%
%Theorem \ref{cubic-property} (1) implies that if $f$ is not renormalizable, then
%its Julia set  is locally connected. This is analogous to Yoccoz's famous theorem for quadratic polynomials \cite{H, M5}. In fact,
%Yoccoz's  proof also works in cubic case when $p=1$, but fails when $p\geq 2$. %  where Yoccoz's argument breaks down. 
%We will see that  

%
% \begin{thm}\label{Y-P}
%Let $(f,X', X)_{\Gamma}$ be a rational-like map, with a unique critical point $c\in K(f)$. Suppose that $\Gamma$ is not periodic at $c$. The we have 
%
%(1). {\textbf (Local connectivity) } $K(f)=J(f)$ and for any $x\in J(f)$,
%$${\rm Imp}(x)=\{x\}.$$ %\ \forall \ x\in J(f).$$
%In particular, the connectivity of $J(f)$ implies local connectivity.
%
%
%
%(2).  {\textbf (Quasiconformal rigidity) } If  $(f,X', X)_{\Gamma}$ and $(g,Y', Y)_{\Sigma}$ are topologically conjugate, then they are  quasi-conformally conjugate.
%
%
%(3). {\textbf (NILF property) } $(f,X', X)_{\Gamma}$ carries no invariant line fields on $J(f)$.
% \end{thm}

This section is organized as follows. We first recall some analytic lemmas to be used in our approach (Section \ref{analytic-tool}).
For further discussions, we distinguish $T_{f}(c^*)$ into the persistently recurrent case and the other (non-recurrent, reluctantly recurrent) cases.
We will recall the principal nest in Section \ref{principal-nest} and use it to deal with the persistently recurrent case  in Section \ref{p-r-case}.
Finally, we treat the rest cases in Section \ref{other-case}.
The methods for these cases are slightly different.
%This case involves the principle nest for critical puzzle pieces. 

\subsection{Analytic tools} \label{analytic-tool}
To prove Theorems \ref{c-rigidity} and \ref{cubic-property}, we 
 need some analytic tools, including a qc-extension lemma (Lemma \ref{qc-extension}); a criterion of no invariant field (Lemma \ref{nilf}); a qc-criterion (Lemma \ref{qc-c}); an analytic fact on Lebesgue density and geometry (Lemma \ref{regular}). The first two will be used in the persistently recurrent case, while the last two take effect in other cases.  
 
%  In particular, 
 
 \begin{lem}[see {\cite[Lemma 3.2]{AKLS}}] \label{qc-extension} For every number $\rho\in(0,1)$ and integer $d\geq 2$, there exist numbers $r=r(\rho, d)\in (\rho, 1)$ and $K_0=K_0(\rho,d)$ with the following property.
 Let $G, \widetilde{G}: \mathbb{D}\rightarrow \mathbb{D} $ be proper holomorphic maps of degree $d$.
 Let $h_1,h_2: \partial\mathbb{D}\rightarrow \partial\mathbb{D}$ be such
that $\widetilde{G}\circ h_2= h_1\circ G$. Assume that

(1).  $|G(0)|, |\widetilde G(0)| \leq\rho$;
 
(2). The critical values of $G, \widetilde{G}$ are contained in $\overline{\mathbb{D}}_\rho$;
 
(3). $h_1$ has a $K_1$-qc extension $H_1: \mathbb{D}\rightarrow\mathbb{D}$ which is the identity on $\mathbb{D}_r$.
 
Then $h_2$ admits a $K_2$-qc extension $H_2: \mathbb{D}\rightarrow\mathbb{D}$ which is the identity on $\mathbb{D}_r$, where
 $K_2=\max\{K_1,K_0\}$.
 \end{lem}

 Lemma \ref{qc-extension} is a variant of \cite[Lemma 3.2]{AKLS} which require that $G(0)=\widetilde G(0)=0$.
 The rewritten condition (1) here allows more flexible applications. 
 Their proofs are essentially same. %with essentially same proof.
 % Here we relax this condition as 
 %(1), for better application. Their proofs are essentially same.
 
 For a topological disk $U\subset \mathbb C$ and a point 
$z\in U$, the {\it shape} of $U$ with respect to $z$ is a quantity to measure the geometry of $U$, defined by
$$\mathbf S(U, z)={\sup_{w\in \partial U} |w-z|}/{\min_{w\in \partial U} |w-z|}.$$

%Given a rational map $R$, let $H(R)$ be the set of holomorphic maps $h: U\rightarrow V$, where $U,V$ are open
%sets so that $R^i\circ h=R^j$ on $U$ for some $i,j\in \mathbb{N}$. 
 %The following criterion is proven in \cite[Proposition 3.2]{S}.

     \begin{lem}[see {\cite[Prop. 3.2]{S}}] \label{nilf}
 Let $R$ be a rational map of degree $\geq 2$ with $\infty\notin J(R)$. Let $z\in J(R)$.
If there exist a constant $C\geq 1$, positive integers $N\geq 2$, $n_k$'s,  and proper maps $h_{k}=R^{n_k}|_{U_k}: U_k\rightarrow V_k, \ k\geq 1$  with the following properties:
 
(1).  $U_k,V_k$ are topological disks in $\mathbb C$ and as $k\rightarrow \infty$
$${\rm diam}(U_k)\rightarrow 0, \ {\rm diam}(V_k)\rightarrow0.$$
 
(2). $2\leq {\rm deg}(h_k)\leq N$, for all $k\geq 1$. 
 
(3). For some $u\in U_k$ with $h_k'(u)=0$ and for $v=h_k(u)$, we have 
 $${\mathbf S}(U_k, u), \  {\mathbf S}(V_k, v)\leq C.$$
 
 (4). $d(U_k, z)\leq C{\rm diam}(U_k), d(V_k, z)\leq C{\rm diam}(V_k)$.
 
Here ${\rm diam}$ and $d$ denote the Euclidean diameter and distance.

 Then for any line field $\mu$ with $R^*\mu=\mu$, either $z\notin {\rm supp}(\mu)$ or $\mu$ is not almost continuous at $z$.
 \end{lem}

The following qc-criterion is a simplified version of \cite[Lemma 12.1]{KSS}, with a slightly difference in the second assumption (that is, we replace {\it a sequence of
 round disks} in \cite{KSS} by
{\it a sequence of disks with uniformly bounded shape}),  and the original proof goes through without any problem.
       
%      To prove the Theorem \ref{c-rigidity} in the other cases, the following QC-criterion due to Kozlovski, Shen and van Strien \cite{KSS} is useful.
%       
       
        \begin{lem}[see {\cite[Lemma 12.1]{KSS}}]\label{qc-c} Let $\phi:\Omega\rightarrow \widetilde{\Omega}$  be a  homeomorphism
between two Jordan domains, $k\in(0,1)$ be a constant.  Let $X$ be a subset of $\Omega$ such that both $X$ and $\phi(X)$ have zero Lebesgue measures. Assume:

1. $|\bar{\partial}\phi|\leq k|{\partial}\phi|$ a.e. on  $\Omega\backslash X$.

2.  There is a constant $M>0$ such that for all  $x \in X$, there is a sequence of open topological disks $D_1 \Supset  D_2 \Supset  \cdots$ containing $x$, satisfying that

(a). $\bigcap_j \overline{D_j}=\{x\}$, and 

(b).  $\sup _j \mathbf S(D_j , x) \leq M$, \   $\sup_j \mathbf S(\phi(D_j ), \phi(x))< \infty.$ 

Then $\phi$ is a $K$-quasi-conformal map, where $K$ depends on $k$ and $M$.
\end{lem}

Lastly, the following fact is useful when dealing with the  non persistently recurrent cases, see \cite[Prop. 6.1]{QWY} and \cite[Lemma 9.4]{QRWY}. 

\begin{lem}\label{regular} Let $R$ be a rational (or rational-like) map with $\infty\notin J(R)$. Let $z\in J(R)$. Suppose there exist integers $D_z>0$ and $0\leq n_1< n_2<\cdots$, a sequence of disk neighborhoods  $U'_j\Subset U_j$ of $z$, two disks $V_z'\Subset V_z$ so that
%$R^{n_k}:(U_k, U'_k) \rightarrow (V, V'), k\geq 1$  are proper maps with degree ${\rm deg}(R^{n_k}|_{U_k}) \leq D$.
%so that 
$R^{n_j}: U'_j\rightarrow V_z'$ and $R^{n_j}: U'_j\rightarrow V_z$  are proper maps of degree $\leq D_z$. Then

(1).  ${\rm diam}(U'_j)\rightarrow 0$ as $j\rightarrow \infty$.

(2). $z$ is not a Lebesgue density point of $J(R)$.

(3).  Assume further $\{R^{n_j}(z)\}_{j\geq1}\subset V''_z$ for some disk $V''_z\Subset V'_z$. Then  
%If there exists an additionalsuch that $ R^{n_j}(z)\in V''_z$, then
  $$\mathbf S(U'_j, z)\leq C(D_z, m_z), \ \forall j\geq 1,$$
% $\mathbf S(U'_k, z)\leq C(D_z,m)\mathbf S(V_z', R^{n_k}(z))$ for all  $k\geq 1$, where $C(D_z,m)$ is a constant depends on
%$D_z$ and 
 where $C(D_z, m_z)$ depends on $D_z$ and $m_z=\{{\rm mod}(V_z-\overline{V'_z}), {\rm mod}(V'_z-\overline{V''_z})\}$.
\end{lem}

\subsection{Principal nest}\label{principal-nest} We assume that $T_f(c^*)$ is persistently recurrent.
%This is our main focus. We will following the idea of Avila-Kahn-Lyubich-Shen \cite{AKLS} to derive the combinatorial rigidity.
Since $\Gamma$ is a $p$-admissible puzzle for $f$, we see that
$P_{d_0+p}(c^*)\Subset P_{d_0}(c^*)$ for some $d_0\geq 0$.  The recurrence of $T_f(c^*)$ allows us to find infinitely many integers $L\geq d_0$
so that $P_{L+p}(c^*)\Subset P_{L}(c^*)$.

Sometimes, we work with $g=f^p$.
%   For integer $L_0\geq d_0$, define a graph for $g$:
%$$\Gamma_g=f^{-L_0}(\Gamma).$$
%Puzzle pieces for $g$ generated by $\Gamma_g$ are denoted by
%$Y_n$ and the one contains $z$ is $Y_n(z)$. The puzzle pieces of $f$ and $g$ satisfy the relation
%$$Y_n(z)=P_{L_0+pn}(z), \ \forall z\in K(f)-\Gamma_\infty, n\geq 0.$$
%The tableau for $g$, denote by $T_g(z)$ (here we add a subscript $g$ to distinguish it from the tableau $T(z)$ for $f$) IN FACT, THIS IS ALREADY INCLUDED IN DEFINITION. 
%Clearly, the $g$-tableau $T_g(z)$ can be obtained by extracting the $(L_0+p\mathbb{N})$-rows and then $p\mathbb{N}$-columns from the  $f$-tableau $T_f(z)$.
It's critical set  
$C(g)=\bigcup_{0\leq k<p} f^{-k}(c^*)$.
View $\Gamma$ as a graph of $g$, one can define the puzzle pieces of $g$ induced by $\Gamma$. The tableau $T_g(z)$ consists of the 
$p\mathbb N\times p\mathbb N$-positions of the tableau $T_f(z)$.

 We may decompose $C(g)=C_0(g)\sqcup C_1(g)$, where 
 $$C_0(g)=\{\zeta\in C(g); \zeta\xrightarrow{g} \zeta\}, \  C_1(g)=C(g)-C_0(g).$$
 
\begin{lem} \label{pr-critical}
 Assume that $T_f(c^*)$ is persistently recurrent, then $c^*\in C_0(g)$.
%there is a $g$-critical point $\zeta^*\in C(g)$ so that $\zeta^*\xrightarrow{g} \zeta^*$. If we further assume that $T_f(c^*)$ is persistently recurrent, then $c^*\xrightarrow{g} c^*$.
\end{lem}
\begin{proof} The recurrence of  $T_f(c^*)$ implies that 
%(1). For any $\zeta\in C(g)$, one has $\zeta\xrightarrow{f} c^*$, and 
there is an integer $0\leq l \leq p$, so that the $(l+p\mathbb{N})$-columns  of the tableau $T_f(c^*)$ contain 
$c^*$-positions of arbitrarily large depth.  
It follows that one can find $\zeta\in f^{-l}(c^*)$, so that the $p\mathbb{N}$-columns  of the tableau $T_f(c^*)$ contain 
$\zeta$-positions of arbitrarily large depth.  
   This means  that $\zeta\xrightarrow{g} \zeta$. %\Longrightarrow 
By tableau rules,  $f^l(\zeta)\xrightarrow{g} f^l(\zeta)$.
%If we further assume that $T_f(c^*)$ is persistently recurrent, then in the columns $0, p, 2p, \cdots$ of the tableau $T_f(c^*)$, there are $c^*$-positions of arbitrarily large depth, so one may take above $\zeta^*=c^*$. 
\end{proof}
 
% To further decompose $C_0(g)$, we pose an equivalence relation on $C_0(g)$: $\zeta\sim \zeta' \Longleftrightarrow   \zeta\xrightarrow{g} \zeta'$ (equivalently,
% $\zeta'\xrightarrow{g} \zeta$). An equivalent class is denoted by $[\zeta]$.
% In this way, one may write $C_0(g)$ as
% $$C_0(g)=[\zeta_1]\sqcup\cdots\sqcup[\zeta_m].$$ 
 
% $=[c^*]\sqcup \langle c^*\rangle$, where
%$$[\zeta^*]=\{\zeta\in C(g); \zeta^*\xrightarrow{g} \zeta \text{ and }\zeta\xrightarrow{g} \zeta^*\}, \  \langle \zeta^*\rangle=
%C(g)-[\zeta^*].$$
%
% 
% 
%
%
%The persistent recurrence of tableau $T_f(c^*)$ implies that $c^*\in \omega_g(\zeta)$ (i.e. $\zeta\xrightarrow{g} c^*$) for all $\zeta\in C(g)$, but the reverse $\zeta\in \omega_{g}(c^*)$ might not be true. This suggests a decomposition $C(g)=[c^*]\sqcup \langle c^*\rangle$, where
%$$[\zeta^*]=\{\zeta\in C(g); \zeta^*\xrightarrow{g} \zeta \text{ and }\zeta\xrightarrow{g} \zeta^*\}, \  \langle \zeta^*\rangle=
%C(g)-[\zeta^*].$$
%
%Note that  $[\zeta^*]=\{\zeta\in C(g); \zeta\xrightarrow{g}\zeta\}$.
 
%To further decompose $C_0(g)$, we pose

For any $\zeta\in C_0(g)$, clearly $\zeta\xrightarrow{g}c^*$. On the other hand, by Lemma \ref{pr-critical} and tableau rules,  
 we see that $c^*\xrightarrow{g}\zeta$.  So we have $[c^*]_g=\{\zeta\in C_0(g); c^*\xrightarrow{g}\zeta\}$. 
 Let ${\rm orb}_g([c^*]_g)=\bigcup_{k\in\mathbb N}g^k([c^*]_g)$. Clearly ${\rm orb}_g([c^*]_g)\subset {\rm orb}(c^*)\cup C(g)$.
%There is a natural equivalence relation on $C_0(g)$: $\zeta\sim \zeta' \Longleftrightarrow   \zeta\xrightarrow{g} \zeta'$ ($\Longleftrightarrow\zeta'\xrightarrow{g} \zeta$). Let $[\zeta]$ be an equivalent class. 
% is denoted by $[\zeta]$.
%Clearly $C_0(g)$ is a disjoint union of finite equivalent classes.
% Let $=\bigcup_{k\geq 0}g^k([c^*]_g)$.

\vspace{5 pt}

We may assume the graph $\Gamma$ (by choosing $L_0\geq d_0$ suitably) satisfy that

\vspace{5 pt}

\textbf{A1}.  $P_{L_0+p}(c^*)\Subset P_{L_0}(c^*)$. 
% ($\Longleftrightarrow P_{L_0+p}(c^*)\Subset P_{L_0}(c^*)$)$.

\textbf{A2}.  For any $\zeta_1, \zeta_2\in C(g)$ with $\zeta_1\neq \zeta_2$, one has $P_{L_0}(\zeta_1)\neq P_{L_0}(\zeta_2)$.

\textbf{A3}. For any $\zeta_1, \zeta_2\in C(g)$ (not necessarily distinct), if they do not satisfy $\zeta_2\xrightarrow{g}\zeta_1$,
 then the $\{L_0\}\times p\mathbb{N}^+$ positions of 
 $T_f(\zeta_2)$ are not $\zeta_1$-positions.

\vspace{5 pt}

%
%
%
%
%We may assume the graph $\Gamma_g$ (by choosing $L_0$ suitably) satisfy that
%
%\vspace{5 pt}
%
%\textbf{A1}. $\Gamma_g$ is 1-admissible and $Y_1(c^*)\Subset Y_0(c^*)$ ($\Longleftrightarrow P_{L_0+p}(c^*)\Subset P_{L_0}(c^*)$).
%
%\textbf{A2}.  For any $\zeta_1, \zeta_2\in C(g)$ with $\zeta_1\neq \zeta_2$, one has $Y_0(\zeta_1)\neq Y_0(\zeta_2)$.
%
%\textbf{A3}. For any $\zeta_1, \zeta_2\in C(g)$ with $\zeta_1\notin\omega_g(\zeta_2)$, the first row of the tableau $T_g(\zeta_2)$ has no
%$\zeta_1$-position.
%
%\vspace{5 pt}
%

The assumption A1 implies that the puzzle piece $Y_0(c^*)=P_{L_0+p}(c^*)$
%any $\zeta \in C(g)$,  one has $P_{L_0+p+l}(\zeta)\Subset P_{L+l}(\zeta)$, where the integer 
%$l=l(\zeta)\in[0, p)$ satisfies $f^{l}(\zeta)=c^*$. In particular, the piece $Y(\zeta)=P_{L+p+l}(\zeta)$ 
satisfies $g(\partial Y_0(c^*))\cap \partial Y_0(c^*)=\emptyset$. In literature, a puzzle piece $Y$ satisfying $g(\partial Y)\cap \partial Y=\emptyset$ is called {\it nice}.
Nice puzzle piece allows one to construct the {\it principle nest}, whose significant properties are summarized  as follows

%
%Write $Y=f^{-L}(X_{c,a}),  Y'=f^{-L-p}(X_{c,a})$.
%
%
% This gives rise to a rational like map $g:Y'\rightarrow Y$.
%

%
%To help our discussion, write 
%$$(g,Y',Y)=(f_{c,a}^p, f_{c,a}^{-d_0-p}(X_{c,a}), f_{c,a}^{-d_0}(X_{c,a})).$$
%It's easy to verify that $\Gamma_g=f_{c,a}^{-d_0}(\Gamma)$ is a 1-admissible puzzle for $g$. 
%The puzzle pieces of $g$ satisfy $Y_{1}(c^*)\Subset Y_0(c^*)$. 
%
%The assumption A1 implies that $g(\partial Y_{1}(c^*))\cap \partial Y_{1}(c^*)=\emptyset$. In literature, puzzle piece $Y$ satisfying $g(\partial Y)\cap \partial Y=\emptyset$ is called {\it nice}.
%Nice puzzle piece allows one to construct the {\it principle nest}, whose significant properties can read as follows
%

\begin{thm} \label{KSS} Assume $T_f(c^*)$ is persistently recurrent and the puzzle $\Gamma$ satisfies  A1, A2, A3. Then 
there exist a nest of $c^*$-puzzle pieces 
$$Y_{0}(c^*)\Supset Y_{1}(c^*)\Supset Y'_{1}(c^*)\Supset Y_{2}(c^*)\Supset Y'_{2}(c^*)\Supset\cdots,$$
each is a suitable pull back of $Y_{0}(c^*)$,
satisfying the following properties:

(1). There exist integers $D_0>0$, $n_j>m_j\geq 1$ for all $j\geq 1$, so that
$$g^{m_j}: Y'_{j}(c^*)\rightarrow Y_{j}(c^*),  \ g^{n_j}: Y_{j+1}(c^*)\rightarrow Y_{j}(c^*)$$
are proper maps of degree $\leq D_0$, and $g^{n_j}(Y'_{j+1}(c^*))\subset Y'_{j}(c^*)$.
%$$g^{m_j-n_j}(Y_{m_{j}}(c^*))=Y_{n_{j}}(c^*), \ g^{n_{j+1}-n_j}(Y_{m_{j+1}}(c^*))\subset Y_{m_{j}}(c^*),$$
%$$ g^{n_{j+1}-n_j}(Y_{n_{j+1}}(c^*))= Y_{n_{j}}(c^*), \ {\rm deg}(g^{n_{j+1}-n_j}|_{Y_{n_{j+1}}(c^*)})\leq D.$$
%

%
%x there is integer $1<s_j<n_{j+1}-n_j$ so that the maps $g^{s_j}: Y_{m_{j}}(c^*)\rightarrow Y_{n_{j}}(c^*)$ and $g^{n_{j+1}-n_j}: Y_{n_{j+1}}(c^*)\rightarrow Y_{n_{j}}(c^*)$ are proper maps of degree $\leq D$, and
%$g^{t_i}(Y_{m_{j+1}}(c^*))\subset Y_{m_{j}}(c^*)$

%(2). The annulus $Y_{n_j}(c^*)- \overline{Y_{m_j}}(c^*)$ has no intersection with ${\rm orb}([c^*])=\bigcup_{k\geq 0}g^k([c^*])$, 

(2). For  all $j\geq 1$,
$$(Y_{j}(c^*)- \overline{Y'_{j}}(c^*))\cap {\rm orb}_g([c^*]_g)=\emptyset.$$

(3). There is a constant $\nu>0$ so that for all $j\geq 1$,
 $${\rm mod}(Y_{j}(c^*)- \overline{Y'_{j}}(c^*))\geq \nu.$$

(4). There is a constant $C_0>0$ so that for all $j\geq 1$,
$$\mathbf S(Y' _j(c^*), c^*)\leq C_0.$$

Moreover, for $j\geq1$, there is another  $c^*$-piece $Y_j''(c^*)\Subset Y_j'(c^*)$ with 
$$(Y'_{j}(c^*)-\overline{Y''_{j}}(c^*))\cap {\rm orb}_g([c^*]_g)=\emptyset \text{ and }  {\rm mod}(Y'_{j}(c^*)- \overline{Y''_{j}}(c^*))\geq \nu.$$

\end{thm}

%MAYBE I should say $ Y_{m_j}(c^*)$ is the first seccessor of $Y_{n_j}(c^*)$ for all $j\geq 1$

The construction of the principal nest is attributed to Kahn-Lyubich \cite{KL1} 
in the unicritical case, Kozlovski-Shen-van Strien \cite{KSS} in the multicritical case.  The complex bounds are proven by Kahn-Lyubich \cite{KL1, KL2} (unicritical case), Kozlovski-van Strien \cite{KS} and Qiu-Yin \cite{QY} independently (the multicritical case). 
The bounded geometry property (4) is derived by  Yin-Zhai \cite[Lemma 6 and Prop.1]{YZ}\footnote{The notations 
$Y_j(c^*), Y'_j(c^*),  Y''_j(c^*)$ here correspond to $K'_n, K_n, \widetilde K_n$ in \cite{YZ}.}.
See these references for  a detail construction of the nest and the proof of its properties.
 
%Besides, the pieces in the principle nest are shown to have bounded geometry, which is measured by {\it shape}.
% For a topological disk $U\subset \mathbb C$ and a point 
%$z\in U$, the {\it shape} of $U$ with respect to $z$ is defined by
%$$\mathbf S(U, z)=\frac{\sup_{w\in \partial U} |w-z|}{\min_{w\in \partial U} |w-z|}.$$

%\begin{pro}\label{bg} Let $Y'_j(c^*), j\geq 1$, $\nu$ be defined in Theorem \ref{KSS}, then 
%
%(1).  $\mathbf S(Y'_j(c^*), c^*), j\geq 1$ is uniformly bounded above by a constant $C_0$.
%
%(2). \end{pro}

%Proposition \ref{bg} is proven by Yin-Zhai \cite[Lemma 6 and Prop.1]{YZ}.

%The bounded shape geometry of the puzzle pieces is proven by Yin-Zhai \cite{YZ}\footnote{The notations 
%$Y_j(c^*), Y'_j(c^*)$ here correspond to $K_n, \widetilde K_n$ in \cite{YZ}.}. 
%The powerful analytic tool to prove the bounds is a covering lemma developed by Kahn-Lyubich \cite{KL2}.
 %See these references for a detail construction of the nest and the proof of complex bounds.

 \begin{rmk}
 Theorem \ref{KSS} with the assumptions $A2,A3$, implies that
  $$(Y_{j}(c^*)- \overline{Y''_{j}}(c^*))\cap {\rm orb}(c^*)=\emptyset.$$
%   has no intersection with
% ${\rm orb}_f(c^*)=\bigcup_{k\geq 0}f^k(\{c^*\})$. 
 \end{rmk}

  \subsection{Proof of Theorems \ref{c-rigidity} and \ref{cubic-property}:  persistently recurrent case. } \label{p-r-case}
  
  Assume $T_f(c^*)$ is persistently recurrent.
  %, since other cases are easier to deal with (see Remark??).  %Recall that for $k\geq0$, $\Gamma_k=f^{-k}(\Gamma)$.
  % Let $\mathcal P_k$ be the collection of all puzzle pieces of depth $k$, and
  Recall that  $\mathbf{P}_k=f^{-k}(X-\Gamma)$  and 
  $\phi: \bigcup_{k\geq 0}\Gamma_k\rightarrow \bigcup_{k\geq 0}\widetilde\Gamma_k$ is a 
  homeomorphism induced by the combinatorial equivalence. Our proof follows the strategy of \cite{AKLS} 
and has six steps.
%  The proof follows the same line as \cite{AKLS},.
  
%  We mainly treat . 
%  The proof also implies that e The other cases are easy and will be remarked afterwards. 
%  

  \vspace{4pt}
     
 %\noindent
 \textbf{Step 1: Construction of qc maps at any depth.}   {\it For any $n\geq 0$, there is a qc-map $\phi_n:\mathbb C\rightarrow \mathbb C$, so that $\tilde f\circ \phi_n=\phi_{n}\circ f$ on $\mathbb C-\mathbf{P}_{n}$.}
   
   %   Need to remark that the dynamical rays are quasi-arcs and the boundary of puzzle pieces are quasicircles

      Note that $\phi|_{\mathbb C-\mathbf P_0}$ is the identity map in B\"ottcher coordinates. % outside the puzzle pieces of depth $0$.
%        the  matches the Bottcher maps near $\infty$ and the attracting cycles in $\mathbb{C}-X_{c,a}$, in other words 
%      $$\phi_0= \begin{cases} 
%(B_{\tilde c, \tilde a}^\infty)^{-1}\circ B_{c, a}^\infty, &\text{ in  } V_{\infty},  \\
%(B_{\tilde c, \tilde a}^k)^{-1}\circ B_{c, a}^k,  &\text{ in } V_k,  0\leq k <p.
%\end{cases}$$
The restriction $\phi|_{\mathbb{C}-\mathbf P_0}$ can be extended to a qc map $\phi_0: (\mathbb C, f(c^*))\rightarrow  (\mathbb C, \tilde f(\tilde c^*))$. Then there is a qc map $\phi_1: \mathbb C\rightarrow  \mathbb C$
so that $\tilde f\circ \phi_1=\phi_0\circ f$ and $\phi_1|_{\mathbb C-\mathbf{P}_0}=\phi_0|_{\mathbb C-\mathbf{P}_0}$. We may adjust $\phi_1$ so that $\phi_1(f(c^*))=\tilde f(\tilde c^*)$. This allows us to get a lift $\phi_2$ of $\phi_1$, in the way that
 $\tilde f\circ \phi_2=\phi_1\circ f$ and $\phi_2|_{\mathbb C-\mathbf{P}_1}=\phi_1|_{\mathbb C-\mathbf{P}_1}$.
By induction, for any $n$, there is a qc-map $\phi_{n+1}$, such that
 $\tilde f\circ \phi_{n+1}=\phi_{n}\circ f$ and $\phi_{n+1}|_{\mathbb C-\mathbf{P}_{n}}=\phi_{n}|_{\mathbb C-\mathbf{P}_{n}}$.
%
% which is a pseudo conjugacy between $f$ and $\tilde f$ at depth $n$.
   We remark that the dilatations of $\phi_n$'s might not be uniformly bounded, to overcome this, we prove:
   
   \vspace{4pt}
   
  \noindent
  
  \textbf{Step 2: Bounding dilatation by critical piece.} {\it For any $j\geq 1$, % let $d_j$ be the depth of $Y_{j}(c^*)$. 
  if $\phi|_{\partial Y_j(c^*)}$ has a $K$-qc-extension 
  $\Phi_j: \overline{Y_j}(c^*)\rightarrow \overline{\widetilde Y_j}(\tilde c^*)$, then it has a further  $K$-qc-extension 
  $H_j: \mathbb C\rightarrow \mathbb C$ such that
$\tilde f \circ H_j=H_j\circ f$ on $\mathbb C-Y_{j}(c^*)$.}
  
%  
%  we have that $$\boxed{\begin{array}{cc}\phi|_{\partial Y_j(c^*)} \text{ has a }  K\text{-qc-extension }   \\
%\Phi_j: Y_j(c^*)\rightarrow \widetilde Y_j(\tilde c^*)  \end{array} }\Longrightarrow \boxed{ \begin{array}{cc}
%\phi|_{\Gamma_{d_j}}  \text{ has a }  K\text{-qc-extension }   \\
%H_j: \mathbb C\rightarrow \mathbb C \text{ satisfying }\\
%H_j\circ f=\tilde f \circ H_j  \text{ on } \mathbb C-Y_{j}(c^*)
%\end{array} }$$ }
%      

%      Let $n=n(j)$ be the depth of $Y_{j}(c^*)$, suppose that 
%      Let  $\phi$ which is a pseudo conjugacy between $f$ and $\tilde f$ at depth $n$.
%        If $\phi|_{\partial Y_j(c^*)}$ has a 
%       $K_j$-q.c. extension $\Phi_j: Y_j(c^*)\rightarrow \widetilde Y_j(\tilde c^*)$, then there is a $K_j$-q.c. pseudo-conjugacy $\Psi_j$ between  $f$ and $\tilde f$ at depth $n$, satisfying that $\Psi_j|_{Y_j(c^*)}= \Phi_j|_{Y_j(c^*)}$. 
       
  To prove the implication, let's define 
      $$Z=\bigcup_{k\geq 1}f^{-k}(Y_j(c^*))-Y_j(c^*).$$ 
      For each component $U$ of $Z$, there is an integer $l=l(U)\geq 1$ (called {\it return time}) so that $f^l: U\rightarrow Y_j(c^*)$ is conformal (clearly its counterpart $\tilde f^l: \widetilde U\rightarrow \widetilde Y_j(\tilde c^*)$ is also conformal). 
       We define $H_j|_U: U\rightarrow \widetilde U$ by
       $$H_j|_U= (\tilde f^l|_{\widetilde U})^{-1}\circ \Phi_j\circ f^l|_{U}.$$
%       
%       
%       
%       
%        so that the following diagram commutes 
%           $$
%\xymatrix{ & U \ar[r]^{f^l}
%\ar[d]_{H|_U}
%&Y_{j}(c^*) \ar[d]^{\Phi_j}\\
%&\widetilde U \ar[r]_{\tilde f^l} & \widetilde Y_{j}(\tilde c^*) }
% $$ 
On $F(f)-Z$ (where $F(f)$ is the Fatou set of $f$), we may define $H_j$ to be identity map in the B\"ottcher coordinates, hence conformal. These maps match on the common boundary of the pieces $U$. Since the residual set 
$$J_{{\rm res}}=\bigcap_{k\geq 0}f^{-k}(J(f)-Y_{j}(c^*))$$
is hyperbolic hence has zero Lebesgue measure,  we see that $H_j$ admits a qc-extension to the plane, with the same dilatation as that of $\Phi_j$.  

 \vspace{4pt}   
      
  \noindent
  
  \textbf{Step 3. An induction procedure. }  {\it  For any $j\geq 1$, we have that
       $$\boxed{\begin{array}{cc}\phi|_{\partial Y_j(c^*)} \text{ has a qc-extension }   \\
\Phi_j: Y_j(c^*)\rightarrow \widetilde Y_j(\tilde c^*)  \end{array} }\Longrightarrow \boxed{ \begin{array}{cc}
 \phi|_{\partial Y_{j+1}(c^*)} \text{ has a qc-extension }   \\
\Phi_{j+1}: Y_{j+1}(c^*)\rightarrow \widetilde Y_{j+1}(\tilde c^*) 
\end{array} }$$
  and the dilatations $K_j, K_{j+1}$ of $\Phi_j, \Phi_{j+1}$ satisfy 
     $$K_{j+1}\leq \max\{K_{j}, K(\nu, \tilde\nu, D_0)\}\leq \max\{K_{1}, K(\nu, \tilde\nu, D_0)\},$$
     where $\nu, \tilde\nu, D_0$ are given by Theorem \ref{KSS}, and $K(\nu, \tilde\nu, D_0)$ depends on them.
      }
%       $K_j$-q.c. extension 
%       $\Phi_j: Y_j(c^*)\rightarrow \widetilde Y_j(\tilde c^*)$. 
%               Then $\phi|_{\partial Y_{j+1}(c^*)}$ has a $K_{j+1}$ q.c. extension 
%       $\Phi_{j+1}: Y_{j+1}(c^*)\rightarrow \widetilde Y_{j+1}(\tilde c^*)$
%             with dilatation $K_{j+1}\leq \max\{K_j, K(\nu, \tilde\nu, D)\}$. Here $K(\nu, \tilde\nu, D)$ is independent of $j$. 

       For each $j\geq 1$, let $\psi_j: (Y_{j}(c^*), c^*)\rightarrow (\mathbb{D}, 0)$ be a conformal map, then 
       $G_j=\psi_j\circ g^{n_j}\circ \psi_{j+1}^{-1}: \mathbb{D}\rightarrow \mathbb{D}$ is proper holomorphic, and $2\leq {\rm deg}(G_j)\leq D_0$.
%       $$
%\xymatrix{ & Y_{j+1}(c^*) \ar[r]^{g^{n_j}}
%\ar[d]_{\psi_{j+1}}
%&Y_{j}(c^*) \ar[d]^{\psi_j}\\
%&\mathbb{D} \ar[r]_{G_j} &  \mathbb{D} }
% $$ 

Let $\Omega_j=\psi(Y'_{j}(c^*))$ for $j\geq 1$. By Theorem \ref{KSS},
we have $G_j(\Omega_{j+1})\subset \Omega_j$, the critical values of $G_j$ are in $\Omega_j$ (in particular $G_j(0)\in \Omega_j$),
and
% we have $g^{n_j}(Y'_{j+1}(c^*))\subset Y'_{j}(c^*)$ and 
%$Y_j(c^*)-\overline{Y'_{j}}(c^*)$ is disjoint from ${\rm orb}_g([c^*]_g)$. So the critical values of $G_j$ are in $\Omega_j$ (in particular $G_j(0)\in \Omega_j$). The fact
$${\rm mod}(\mathbb D-\overline{\Omega}_j)={\rm mod}(Y_{j}(c^*)- \overline{Y'_{j}}(c^*))\geq \nu.$$
So there is a constant $\rho(\nu)\in(0,1)$ with ${\Omega}_j\subset\overline{\mathbb D}_{\rho(\nu)}$.
%Let $\rho=\max\{\rho(\nu), \rho(\tilde\nu)\}$.
% and $r_0=\max\{r(\rho, d); 2\leq d\leq D\}$, where $d(\rho, d)$'s are defined as in Lemma \ref{qc-extension}. 

Let's define $h_j, h_{j+1}: \partial \mathbb D\rightarrow  \partial \mathbb D$ by
$$h_j=\widetilde \psi_j\circ \phi|_{\partial Y_j(c^*)}\circ \psi^{-1}_{j}, \ 
h_{j+1}=\widetilde \psi_{j+1}\circ \phi|_{\partial Y_{j+1}(c^*)}\circ \psi^{-1}_{j+1}.$$
Clearly $h_j\circ G_j= \widetilde G_j\circ h_{j+1}$.
By induction hypothesis, $h_j$ has an extension $L_j=\widetilde \psi_j\circ \Phi_j \circ \psi^{-1}_{j}$. 
Note that $G_j, \tilde G_j, h_j, h_{j+1}$ satisfy the assumptions in Lemma \ref{qc-extension}. Let 
 $$\rho=\max\{\rho(\nu), \rho(\tilde\nu)\}, \ r_0=\min\{r(\rho, d); 2\leq d\leq D_0\} \in (\rho,1),$$
 where $r(\rho, d)$'s are given by Lemma \ref{qc-extension}.
 Assume that $L_j$ is identity\footnote{This assumption is satisfied by making $L_1$ satisfy $L_1|_{\mathbb D_{r_0}}=id$ and by induction.}  on $\mathbb D_{r_0}$,
% If we further assume that $H_j$ (this is can be done at the in) is identity on $\mathbb D_{r_0}$, t
 then by Lemma \ref{qc-extension}, $h_{j+1}$ has a  qc extension $L_{j+1}$, which is identity on  $\mathbb D_{r_0}$, with dilatation $K_{j+1}\leq \max\{K_j, K(\nu, \tilde\nu, D_0)\}$.
 Finally, we take $\Phi_{j+1}=\widetilde \psi^{-1}_{j+1}\circ L_{j+1}\circ \psi_{j+1}$
% . Then $\Phi_{j+1}: Y_{j+1}(c^*)\rightarrow \widetilde Y_{j+1}(\tilde c^*)$ is 
 and get an extension of $\phi|_{\partial Y_{j+1}(c^*)}$.

%       
%       Write $Y_{j}(c^*)=P_{m}(c^*), Y_{j+1}(c^*)=P_{l+m}(c^*)$, $c^*_j=f^i(c^*)$ for $0\leq i\leq l$. For $0\leq i\leq l$, take a conformal map $\phi_i: (P_{l+m-i}(c_i^*), c_i^*)\rightarrow (\mathbb{D}, 0)$.
%          $$   
%       \xymatrix{&
%P_{l+m}(c_0^*)\ar[r]^{f}
%\ar[d]_{\phi_0}
%&P_{l+m-1}(c_1^*)\ar[r]^{f}\ar[d]_{\phi_1}
% & \cdots \ar[r]^{f}
%&P_{l+1}(c_{m-1}^*)\ar[r]^{f}\ar[d]_{\phi_{m-1}}
% &P_{l}(c_{m}^*)\ar[d]^{\phi_m}\\
%& \mathbb D\ar[r]_{g_1}
% &\mathbb D\ar[r]_{g_2}
%& \cdots \ar[r]_{g_{m-1}}&
%\mathbb D\ar[r]_{g_m}   &
%\mathbb D
% }
%$$
%Each $g_i=\phi_i\circ f\circ \phi_{i-1}: (\mathbb D, 0)\rightarrow (\mathbb{D}, 0)$ is either conformal or a double cover.
%For $i\geq 1$, let $E_i= \phi_i(f^i(Y'_{j+1}(c^*)))$. By Theorem \ref{KSS}, $g^{n_j}(Y'_{j+1}(c^*))\subset Y'_{j}(c^*)$ and 
%$P_{l+m-i}(c_j^*)-f^i(Y'_{j+1}(c^*))$ is disjoint from ${\rm orb}_f(\{c^*\})$. So the critical value of $g_i$ is in $E_i$ and
%$${\rm mod}(\mathbb D-\overline{E_i})\geq \nu/D.$$

 \vspace{4pt}  
 
 \noindent
 
 \textbf{Step 4: Conjugacy via taking a limit.}

 \vspace{4pt}  
 
 By Step 3,   the map $\phi|_{\partial Y_j(c^*)}$ has a 
       qc extension $\Phi_j: Y_j(c^*)\rightarrow \widetilde Y_j(\tilde c^*)$, with dilatation $K_*=\max\{K_1, K(\nu, \nu', D_0)\}$. By Step 2, there is an extension $H_j$ of $\Phi_j$, conjugate $f$ to $\tilde f$ on $\mathbb C-Y_{j}(c^*)$, without increasing the dilatation of $\Phi_j$. Then $\{H_j; j\geq 1\}$ is a normal family, whose limit is a $K_*$-qc map $H$, satisfying $\tilde f\circ H=H\circ f$
     on the Fatou set of $f$. Since $J(f)$ has no interior, $H$ is a conjugacy on $\mathbb C$ by continuity. 
     
\vspace{4pt} 

\noindent

\textbf{Step 5: $f$ carries no invariant line fields on  $J(f)$.} 
%We show that $f$ carries no invariant line fields on its Julia set.

\vspace{4pt}

First note that the set
$$ X_\infty:=\bigcup_{j\geq 1}\bigcap_{k\geq 1}f^{-k}(J(f)-Y_{j}(c^*))$$
has Lebesgue measure zero and $\Gamma_\infty\cap J(f)\subset X_\infty$. Suppose that $f$ carries an invariant line field $\mu$.
Let $z\in {\rm supp}(\mu)\cap  (J(f)-X_\infty)$.
Clearly, $T_f(z)$ is critical.  Let $Y_{j}(c^*), Y'_{j}(c^*), Y''_{j}(c^*)$ be given by Theorem \ref{KSS}, and write 
$$Y_{j}(c^*)=P_{d_j}(c^*), Y'_{j}(c^*)=P_{d'_j}(c^*), Y''_{j}(c^*)=P_{d''_j}(c^*).$$ Let $s_j\geq 0$ be the first integer with $f^{s_j}(P_{d''_j+s_j}(z))=P_{d''_j}(c^*)$. 
By Theorem \ref{KSS} (2)(4), one has 
$f^{s_j}(P_{d_j+s_j}(z))=P_{d_j}(c^*)$, $f^{s_j}(P_{d'_j+s_j}(z))=P_{d'_j}(c^*)$, and $f^{s_j}|_{P_{d_j+s_j}(z)}$ is conformal (in particular, if $z\in P_{d''_j}(c^*)$, then $s_j=0$ and $f^{s_j}=id$),
we have
$${\rm mod}(P_{d_j+s_j}(z)-\overline{P_{d'_j+s_j}}(z))\geq \nu, \ {\rm mod}(P_{d'_j+s_j}(z)-\overline{P_{d''_j+s_j}}(z)) \geq \nu.$$
It follows that 
\bess
\mathbf S(P_{d'_j+s_j}(z),z)&\leq& C_1(\nu) \mathbf S(Y_{d'_j}(c^*), f^{s_j}(z)) \ (\text{by Koebe distortion})\\ 
&\leq& C_1(\nu)C_2(\nu) \mathbf S(Y_{d'_j}(c^*), c^*) \ (\text{by \cite[Thm 2.5]{Mc}})\\
&\leq& C_1(\nu)C_2(\nu)C_0,
\eess
where  $C_1(\nu), C_2(\nu)$ are constants depending on $\nu$, and $C_0$ is given by Theorem \ref{KSS}.
To apply Lemma \ref{nilf}, we take $V_j=P_{d'_j+s_j}(z)$. It remains to find $U_j$. Let $t_j>0$ be the first integer such that  $f^{t_j}(P_{d''_j+t_j}(c^*))=P_{d''_j}(c^*)$ and  $r_j\geq 0$
be the first integer such that $f^{r_j}(P_{d''_j+t_j+r_j}(z))=P_{d''_j+t_j}(c^*)$.
Again Theorem \ref{KSS} (2)(4) assert that 
$f^{r_j}(P_{d_j+t_j+r_j}(z))=P_{d_j+t_j}(c^*)$ and $f^{r_i}(P_{d'_j+t_j+r_j}(z))=P_{d'_j+t_j}(c^*)$. 
Similarly as above, one has  
\bess
\mathbf S(P_{d'_j+t_j+r_j}(z),z)\leq C_1(\nu)C_2(\nu)C_0.
\eess
We take $U_j=P_{d'_j+t_j+r_j}(z)$, and $h_j=f^{t_j+r_j-s_j}|_{U_j}$ (one may verify that $s_j\leq r_j$). Then 
 $h_j: U_j\rightarrow V_j$ is of degree two, and 
satisfies the assumptions of Lemma \ref{nilf}. Hence $\mu$ is not almost continuous at $z$, which gives a contradiction.  The proof of Step 5 is completed.
 
 It follows that the qc conjugacy $H$ obtained in Step 4 is conformal, and $H(z)=z+O(1)$ near $\infty$, therefore 
 $H(z)=z$ and $f=\tilde f$. The proof of Theorem \ref{c-rigidity} in the persistently recurrent case is finished.
     
\vspace{4pt} 

\noindent

\textbf{Step 6: Local connectivity of $J(f)$.} Note that for any $z\in J(f)$ and any $n\geq 0$, the intersection ${P^*_{n}(z)}\cap J(f)$ is connected (because each connected component of ${P^*_{n}(z)}-J(f)$ is simply connected). It suffices to show that
${\rm Imp}(z)=\{z\}$. Let $d_j,d_j',d_j''$ be given in Step 5.

For  $z\in J(f)-\Gamma_{\infty}$ with $T_f(z)$ critical, let $l_j\geq 0$ be an integer so that
  $f^{l_j}: P_{d'_j+l_j}(z)\rightarrow P_{d'_j}(c^*)$ is conformal. Clearly $f^{l_j}: P_{d_j+l_j}(z)\rightarrow P_{d_j}(c^*)$ is also conformal, by Theorem \ref{KSS} (2). Therefore ${\rm mod}(P_{d_j+l_j}(z)-\overline{P_{d'_j+l_j}}(z))\geq \nu$, and hence 
  ${\rm Imp}(z)=\bigcap  \overline{P_{d_j+l_j}}(z)=\{z\}$.  In particular, ${\rm Imp}(c^*)=\{c^*\}.$

  For $z\in J(f)-\Gamma_{\infty}$ with $T_f(z)$ non critical,  or $z\in \Gamma_{\infty}\cap J(f)$,
  the proof of the fact ${\rm Imp}(z)=\{z\}$ is the same as the quadratic case \cite{M4}.
  This case involves the so called {\it thickened puzzle piece} technique, see \cite{M4} for its construction and 
   \cite[Lemmas 1.6 and 1.8, Theorem 1.9]{M4} for its applications.  For this, we skip the details.

%  the proof use the {\it thicken puzzle piece} as 
%  
%  
%  
%  
%   
%  
%  
%   one can find integers $n_j$'s, a puzzle piece  $P_d$ of depth $d$ so that $f^{n_j}(z)\in P_d$ and $f^{n_j}: P_{n_j+d}(z)\rightarrow  P_d$ is conformal (replace $z$ by $f(z)$ if necessary). Let $h_j: P_d\rightarrow  P_{n_j+d}(z)$ be the inverse of $f^{n_j}|_{P_{n_j+d}(z)}$.
% Since each point of $\partial P_d\cap J(f)$ is pre-repelling,  by choosing subsequences, we may assume 
% $\{f^{n_j}(z);j\geq 1\}\subset U\Subset P_d$. Then Lemma \ref{regular} (1)
% implies that ${\rm diam}(h_j(U))\rightarrow 0$ as $j\rightarrow \infty$.
%Equivalently, the maps $h_j$'s converge uniformly on $\overline{U}$ to a point $z$. By Hurwitz's theorem and Caratheodory's theorem, we have $\bigcap  \overline{P_{n_j+d}}(z)=\bigcap \overline{h_j(P_d)}=\{z\}$. 
%   
%  
%   
%   For $z\in \Gamma_{\infty}\cap J(f)$, recall that
%$P^*_n(z)$ is a finite union of the puzzle pieces of depth $n$ touching at $z$. Since all periodic points on $\Gamma_{\infty}\cap J(f)$ are repelling, similarly as the non-critical case, one has ${\rm diam}(P^*_n(z))\rightarrow 0$ as $n\rightarrow \infty$.
The proof of Theorem \ref{cubic-property} in the persistently recurrent case is finished.
 \hfill\fbox

\vspace{4pt}

 \subsection{Proof of Theorems \ref{c-rigidity} and \ref{cubic-property}: other cases.}\label{other-case} In this part, we assume $T_f(c^*)$ is either  reluctantly recurrent or non recurrent. For  Theorem \ref{cubic-property}(2), a stronger fact that $J(f)$ has zero Lebesgue measure is proven.
 
Lemmas \ref{qc-c} and \ref{regular} will take effect in the proof. % of Theorems \ref{c-rigidity} and \ref{cubic-property}.
To verify the assumptions of these lemmas, we first show:

\begin{lem}\label{red-assump} For any $z\in J(f)$, there exist integers 
$D, m>0$(both independent of $z$)  and $n_j$'s, Jordan disks $U_j(z)\Supset U'_j(z)\Supset U''_j(z)$'s and $V_z\Supset V_z'\Supset V''_z$ such that 

1. $\{f^{n_j}(z);j\geq 1\}\subset V''_z$.

2. ${\rm deg}(f^{n_j}: U_j(z)\rightarrow V_z) \leq D$ for all $j\geq1$.

3. ${\rm mod}(V_z-\overline{V'_z})\geq m, \ {\rm mod}(V'_z-\overline{V''_z}) \geq m$.
\end{lem}

\begin{proof}   Recall that $P_{L_0+p}(c^*)\Subset P_{L_0}(c^*)$.  Let $q>0$ be an integer (to be determined later).
We first treat the points in $J(f)-\Gamma_{\infty}$ whose tableau is critical,
then deal with points in $J(f)\cap\Gamma_{\infty}$ or in $J(f)-\Gamma_{\infty}$ whose tableau is non critical.
The choices of $D$, $n_j$'s, and the Jordan disks can be seen in the proof. 
The numbers $m$ and $q$ will be determined in the final step.

 \textbf{(1). $z\in J(f)-\Gamma_{\infty}$ and $T_f(z)$ is critical.}
   
{\it Case 1. $T_f(c^*)$ is not recurrent}. In this case,
$$D_{c^*}:=\sup_{k} {\rm deg}(f^k|_{P_k(c^*)})<+\infty.$$
  Let $(L_0+p+q, n_j), j\geq 1$ be all the $c^*$-positions in the tableau  $T_f(z)$.
By the tableau rules, we have that for all $j\geq 1$,
$${\rm deg}(f^{n_j}: P_{n_j+L_0}(z)\rightarrow P_{L_0}(c))\leq D_{c^*}.$$ 
It suffices to take
\bess
&(U_j(z),U'_j(z), U''_j(z))=(P_{n_j+L_0}(z), P_{n_j+L_0+p}(z), P_{n_j+L_0+p+q}(z)),& \\
&(V_{z},V_{z}', V''_z)=(P_{L_0}(c^*), P_{L_0+p}(c^*), P_{L_0+p+q}(c^*)).&
\eess

{\it Case 2. $T_f(c^*)$ is reluctantly recurrent.}
The recurrence of $T_f(c^*)$ implies that there is an integer $L\geq L_0$ so that 
$P_{L+p}(c^*)\Subset P_{L}(c^*)$ and $P_{L}(c^*)$ has infinitely many children, say $P_{L+n_j}(c^*),j\geq1$.
Let $\mathcal J$ be the collection of indices $j\in \mathbb N$ so that  $P_{L+p}(c^*)$ has a child  $P_{L+l+p}(c^*)$ with $l\in[n_j,n_{j+1})\cap \mathbb N$. For each $j\in\mathcal J$, let $m_j\in[n_j,n_{j+1})\cap \mathbb N$ be the first integer so that $P_{L+m_j+p}(c^*)$ is a child of $P_{L+p}(c^*)$. 
Define $\mathcal J'\subset \mathcal J$ by
$$\mathcal J'=\{j\in \mathcal J; P_{L+p+q}(c^*) \text{ has a child } P_{L+l+p+q}(c^*) \text{ with } l\in [m_j,n_j)\cap \mathbb N \}.$$
The recurrence of $T_f(c^*)$ implies that  $\mathcal J'$ is an infinite set.

For each $j\in\mathcal J'$, let $l_j\in[m_j,n_{j+1})\cap \mathbb N$ be the first integer so that $P_{L+l_j+p+q}(c^*)$ is a child of $P_{L+p+q}(c^*)$. The choices of $m_j, l_j$ imply that
$${\rm deg}(f^{l_j}: P_{L+l_j}(c^*)\rightarrow P_{L}(c^*))\leq 2\cdot 2^{p+q}, \ \forall j\in\mathcal J'.$$

For $c^*$, we take $(V_{c^*},V_{c^*}', V''_{c^*})=(P_{L}(c^*), P_{L+p}(c^*), P_{L+p+q}(c^*))$ and
$$(U_j(c^*),U'_j(c^*), U''_j(c^*))=(P_{L+l_j}(c^*), P_{L+p+l_j}(c^*), P_{L+p+q+l_j}(c^*)).$$

Let $z\in J(f)-\Gamma_{\infty}$ with $z\neq c^*$ and $T_f(z)$ critical. For each $j\in \mathcal J'$, let $k_j\geq 0$ be the first integer so that 
$f^{k_j}: P_{L+l_j+k_j}(z)\rightarrow  P_{L+l_j}(c^*)$ is conformal. Fix $k_j$, let $s_j\geq k_j$ be the first integer so that 
$f^{s_j}: P_{L+p+l_j+s_j}(z)\rightarrow  P_{L+p+l_j}(c^*)$ is conformal. Fix $s_j$, let $t_j\geq s_j$ be the first integer so that 
$f^{t_j}: P_{L+p+q+l_j+t_j}(z)\rightarrow  P_{L+p+q+l_j}(c^*)$ is conformal.
Then the degree of $f^{t_j}: P_{L+l_j+t_j}(z)\rightarrow  P_{L+l_j}(c^*)$ is bounded by $2^{p+q}$.
It follows that the degree of $f^{l_j+t_j}: P_{L+l_j+t_j}(z)\rightarrow  P_{L}(c^*)$ is bounded by $2\cdot 2^{p+q}\cdot 2^{p+q}=2\cdot 4^{p+q}$.

 We may take
 \bess
&(U_j(z),U'_j(z), U''_j(z))=(P_{L+l_j+t_j}(z), P_{L+p+l_j+t_j}(z), P_{L+p+q+l_j+t_j}(z)),&\\ 
&(V_{z},V_{z}', V''_z)=(P_{L}(c^*), P_{L+p}(c^*), P_{L+p+q}(c^*)).&
\eess

\textbf{(2).} $z\in J(f)-\Gamma_{\infty}$ and $T_f(z)$ is non critical, or $z\in J(f)\cap\Gamma_{\infty}$.

Suppose $z\in J(f)-\Gamma_{\infty}$ and $T_f(z)$ is non critical.  
By the same argument as step 6 in Section \ref{p-r-case}, one can
show that ${\rm Imp}(z)=\{z\}$. Based on this fact, we can find integers 
$d''>d'>d$, $n_j$'s, three puzzle pieces $P_d, P_{d'}, P_{d''}$ of depths $d,d',d''$ respectively, with the following properties: 

\begin{itemize}{

\item ${\rm mod}(P_d-\overline{P_{d'}})\geq 1$ and ${\rm mod}(P_{d'}-\overline{P_{d''}})\geq 1$.

\item For $j\geq 1$ and $l\in\{d,d',d''\}$, the map $f^{n_j}: P_{l+n_j}(z)\rightarrow P_l$ is proper. 

\item $\{f^{n_j}(z);j\geq 1\}\subset P_{d''}$ and ${\rm deg}(f^{n_j}: P_{d+n_j}(z)\rightarrow P_d)\leq 2.$
 }
 \end{itemize}
 
 Set $V_z=P_{d}, V'_z=P_{d'}, V''_z=P_{d''}$ and
$$U_j(z)=P_{d+n_j}(z), U'_j(z)=P_{d'+n_j}(z), U''_j(z)=P_{d''+n_j}(z).$$

For $z\in J(f)\cap\Gamma_{\infty}$,  we replace $P_l$ by $P^*_l$ and  the argument is similar.

\textbf{(3). The choice of $q$ and $m$}.  In fact, for all $z\in J(f)$,  we have already chosen puzzle pieces $U_j(z)\Supset U'_j(z)$ so that ${\rm mod}(U_j(z)-\overline{U'_j}(z))$ has a lower bound independent of $j$. This implies that ${\rm Imp}(z)=\bigcap \overline{U_j}(z)=\{z\}$. In particular, we have ${\rm Imp}(c^*)=\{c^*\}$ implying that
${\rm diam}(P_j(c^*))\rightarrow 0$ as $j\rightarrow \infty$.
It suffices to take $q$
with $P_{L_0+p+q}(c^*)\Subset P_{L_0+p}(c^*)$. It follows that 
$$m:=\inf_{z\in J(f)}\min\{{\rm mod}(V_z-\overline{V'_{z}}), \ {\rm mod}(V'_z-\overline{V''_{z}}) \}>0.$$
The proof is completed.
\end{proof}

  {\it Proof of Theorems \ref{cubic-property} and \ref{c-rigidity} (other cases).}
  For any $z\in J(f)$,  the sets $U_j(z), U'_j(z)$'s given by Lemma \ref{red-assump}
    are puzzle pieces satisfying that 
    $\bigcap \overline{U_j}(z)=\{z\}$ and $\overline{U_j}\cap J(f)$'s  are  connected. The local connectivity of $J(f)$ at $z$, hence  
    at all points, follows immediately.

By the proof Theorem \ref{cubic-property} (1), one has ${\rm Imp}(z)=\{z\}$  for any $z\in J(f)$. 
Recall that $\phi_*$ is a bijection  between puzzle pieces.  Combining these facts, we 
get a natural extension 
 $\Phi:\mathbb C\rightarrow\mathbb C$ of $\phi: \bigcup\Gamma_k\rightarrow \bigcup\widetilde\Gamma_k$
as follows: on the Julia set,
we define $\Phi(z)$  as the intersection point of $\bigcap \phi_*(P_k(z))$; in the Fatou
components, we define $\Phi(z)$ inductively by  $\tilde f\circ \Phi=\Phi\circ f$.
One may verify that  $\Phi:\mathbb C\rightarrow \mathbb C$ is a homeomorphism, 
satisfying $\tilde f\circ \Phi=\Phi\circ f$ in $\mathbb C$.

Lemmas \ref{red-assump} and \ref{qc-c} imply that $J(f)$ has zero measure.
To show that $\Phi$ is quasi-conformal, by Lemma \ref{qc-c}, it suffices to verify the assumption 2(b). 
By Lemma \ref{regular}, it reduces to show that
$D_z, \mathbf{S}(V_z', f^{n_k}(z))$ are bounded by constants independent of $z,k$.
This follows from Lemma \ref{red-assump}.   \hfill \fbox

 \section{Boundary regularity}\label{br}

In this section, we show
 
 \begin{thm} \label{regularity}
  Every Type-A or B hyperbolic component is  a Jordan disk.
% 
%  Let $\mathcal{H}\subset\mathcal S_p$ be a hyperbolic component of Type-$\omega\in\{A,B\}$,  then its boundary  $\partial\mathcal{H}$ is a Jordan curve. 
 \end{thm}
 
 Let $\mathcal{H}$ be a hyperbolic component of Type-$\omega\in\{A,B\}$.  Recall that $\ell(\mathcal H)$ is the integer $k\in[0,p)\cap\mathbb N$ so that $-c\in U_{c,a}(f_{c,a}^k(c))$.

 \subsection{Maps on the boundary of $\mathcal{H}$} We first show
%  Recall that
% $$\mathcal{A}_{c,a}=U_{{{c},a}}(c)\cup U_{c,a}(f_{c,a}(c))\cup\cdots\cup  U_{c,a}(f^{p-1}_{c,a}(c)).$$
%
% 
 \begin{pro} \label{bm1} The boudary $\partial\mathcal{H}$ consists of parameters $(c,a)$ for which the map $f_{c,a}$ has
either a parabolic point or the critical point $-c$ on $\partial\mathcal{A}_{c,a}$.
% 
%  Let $\mathcal{H}$ be a hyperbolic component of Type-A or -B on $\mathcal{S}_p$. Then for any $(c,a)\in \partial\mathcal{H}$, the boundary $\partial U_{c,a}(f^k_{c,a}(c))$ of some attracting basin of $f_{c,a}$ contains either a parabolic  point or $-c$.
  \end{pro}
  
  \begin{proof} Let $(c,a)\in \partial\mathcal{H}$,  assume  the restriction  $f_{c,a}^p: \partial U_{c,a}(c)\rightarrow \partial U_{c,a}(c)$ has neither critical point nor parabolic point. % on $\partial U_{c,a}(c)$.  % We will get a contradiction by Mane's theorem.
%  \vspace{5pt} 
% \textit{ Claim:  There exist an integer $k_0\geq 1$ and two topological disks $X_{c,a}, Y_{c,a}$ with $U_{c,a}(c)\subset X_{c,a}\Subset Y_{c,a}$, such that $f^{k_0p}_{c,a} : X_{c,a}\rightarrow Y_{c,a}$ is a polynomial-like map of degree $2^{k_0}$ with only one critical point $c$.}
Write $g=f_{c,a}^p$ and $V={U_{c,a}(c)}$. 

By assumption, one has $\overline{V} \cap \overline{(g^{-1}(V)-V)}=\emptyset$.
%. and other components of $g^{-1}(V)$ have disjoint closures.
Let $\phi:\mathbb{\widehat{C}}\setminus \overline{V}\rightarrow \mathbb{\widehat{C}}\setminus{\overline{\mathbb{D}}}$ be the Riemann mapping fixing $\infty$, then $G=\phi\circ g\circ \phi^{-1}$ is  defined in $\phi(\mathbb{\widehat{C}}\setminus {g^{-1}(\overline{V})})$ and  $G(\partial \mathbb{D})=\partial \mathbb{D}$. By Schwartz reflection, this $G$ can be defined in an annular neighborhood $U$ of $\partial \mathbb{D}$. %, and is still denoted by $G$.

By assumption, $G$ has no critical point on $\partial \mathbb{D}$ and ${\rm deg}(G|_{\partial \mathbb{D}})=2$. If $G$ has a non-repelling periodic point, say $q$ with period $k$, on $\partial\mathbb{D}$. The multiplier $\lambda=(G^k)'(q)$ is real because  $G(\partial \mathbb{D})=\partial \mathbb{D}$.
It turns out that $q$ is either attracting or parabolic.  Let $\mathbf{B}$ be its immediate basin, and $\mathbf{A}=\phi^{-1}(\mathbf{B})=\phi^{-1}(\mathbf{B}\cap (\mathbb{\widehat{C}}\setminus{\overline{\mathbb{D}}}))$. Then  $\mathbf{A}$ is bounded in 
$\mathbb{C}$ and is stable by $g^k$.  

Note that every point of $\mathbf{B}$ is attracted to $q$ under iterations of $G^k$, meaning that every point of $\mathbf{A}$ is attracted to a boundary point of $\partial V$ by $g^k$, this means that $g$ has a parabolic point on $\partial V$. Contradiction. 

Hence the analytic map $G$ has neither critical point nor non-repelling point on $\partial\mathbb{D}$. By Ma\~n\'e's theorem \cite{Ma}, 
$\partial \mathbb{D}$ is a hyperbolic set of $G$: there are constants $C>0$ and $\nu>1$ such that for all $n\geq 1$
$$|(G^{n})'(z)|\geq C\nu^n.$$
%One may chose $\epsilon>0$ and $m>0$ so that $\gamma=G^{-m}(\{|w|=1+\epsilon\})$ is compactly contained in $\{|w|<1+\epsilon\}$.

Then one can find an integer $\ell\geq 1$ and   two annular neighborhoods $X, Y$ of $\partial\mathbb{D}$  with $X\Subset Y\subset U$, such that $G^{\ell}: X\rightarrow Y$ is a proper map of degree $2^{\ell}$.
By pulling back $X\setminus{\overline{\mathbb{D}}}, \ Y\setminus{\overline{\mathbb{D}}}$ via $\phi$,  we  get a polynomial-like map $f_{c,a}^{\ell p}: Z_{c,a}\rightarrow Y_{c,a}$,
  where $$Z_{c,a}=\phi^{-1}(X\setminus{\overline{\mathbb{D}}})\cup  \overline{V}, \ Y_{c,a}=\phi^{-1}(Y\setminus{\overline{\mathbb{D}}})\cup  \overline{V}.$$

It's clear that there is a neighborhood  $\mathcal{U}$ of $(c,a)$, such
that for all $(\tilde{c},\tilde{a})\in \mathcal{U}$, the map  $f_{\tilde{c},\tilde{a}}^{\ell p}$ has exactly one  critical value  in $\overline{Y_{c,a}}$.
 This critical value is nothing but $\tilde{c}$.
Thus the
component $Z_{\tilde{c},\tilde{a}}$ of  $f_{\tilde{c},\tilde{a}}^{-\ell p}(Y_{c,a})$ that contains $\tilde{c}$ is
a disk. Since $Z_{\tilde{c},\tilde{a}}$ moves holomorphically with respect to
$(\tilde{c},\tilde{a})\in \mathcal{U}$, we may shrink $\mathcal{U}$ if necessary
 so that $Z_{\tilde{c},\tilde{a}}\Subset Y_{{c},{a}}$ for all  $(\tilde{c},\tilde{a})\in \mathcal{U}$.   In this way, we get a polynomial-like map $f_{\tilde{c},\tilde{a}}^{\ell p}: Z_{\tilde{c},\tilde{a}}\rightarrow
 Y_{c,a}$  of degree $2^{\ell}$ for all $(\tilde{c},\tilde{a})\in \mathcal{U}$.

 However when $(\tilde{c},\tilde{a})\in \mathcal{U}\cap\mathcal{H}$,
 its clear that $\pm \tilde{c}\in \mathcal{A}_{\tilde c,\tilde a}$, 
 % the cycle of $U_{\tilde c,\tilde a}(\tilde c)$ contains the two critical points $\pm \tilde{c}$ of $f_{\tilde{c},\tilde{a}}$ 
  and the degree of
 %{\rm deg}(N^K_u|_{B_u^\varepsilon})=
 $f_{\tilde{c},\tilde{a}}^{\ell p}: Z_{\tilde{c},\tilde{a}}\rightarrow
 Y_{c,a}$ is  $(d_\omega+1)^{\ell}>2^{\ell}$. It's a contradiction.
 \end{proof}

 \begin{pro} \label{bm1-1} Let $(c,a)\in \partial\mathcal{H}, f=f_{c,a}$ and let $\Gamma$ be the $p$-admissible puzzle of $f$ (by Theorem \ref{adm-puzzle}). 
 Then $-c\in\partial\mathcal{A}_{c,a}$ iff  $T_{f}(-c)$ is aperiodic.
 \end{pro}
 \begin{proof}
 If $-c\in \partial W$  for some $W\in \mathcal B_{c,a}$, then this $W$ is unique (by Proposition \ref{common-point}). For all $d\geq 1$, let $P_{d}(-c)$ be the puzzle piece of depth $d$ containing 
 $-c$. By the puzzle structure, $\partial P_{d}(-c) \cap W$  contains two sections of internal rays $R_{c,a}^W(\alpha_d)$ and
  $R_{c,a}^W(\beta_d)$ with $0\leq\alpha_d<\beta_d\leq 1$.  
   Clearly  when $d$ is large,  
  $$\alpha_d\leq\alpha_{d+1}\leq\cdots\leq\beta_{d+1}\leq\beta_{d}.$$
  If $T_{f}(-c)$ is periodic, then its period $l=sp$ for some $s\geq 1$. We have
  $$\beta_{d+s}-\alpha_{d+s}=2^{-s}(\beta_{d}-\alpha_{d})$$
 for large $d$. So $\alpha_k$'s and $\beta_k$'s have a common limit $t$, which satisfies 
 $2^s t=t \ {\rm  mod } \ \mathbb Z$. 
 The fact $\bigcap (P_k(-c) \cap \partial W)=\{-c\}$ implies that $R_{c,a}^W(t)$ lands at $-c$,
  meaning that $-c$ is a periodic point on $J(f)$. This is a contradiction.
 
 If $-c\notin \partial\mathcal{A}_{c,a}$, then by Proposition \ref{bm1}, $f$ has a parabolic point $\zeta\in\partial\mathcal{A}_{c,a}$. Clearly $T_f(\zeta)$ is periodic, with period say $n\geq 1$.
 We claim that  
 $$-c\in P_{d+n}(\zeta)\cup P_{d+n-1}(f(\zeta))\cup\cdots \cup P_{d+1}(f^{n-1}(\zeta)).$$
 In fact, if it is not true, then $f^n:  P_{d+n}(\zeta)\rightarrow P_{d}(\zeta)$ is conformal, implying that
 $\zeta$ is repelling (by Schwarz Lemma). This contradicts that $\zeta$ is parabolic.
 Hence we get the claim, which implies that $T_f(-c)$ is periodic.
 \end{proof}

   \subsection{Parameter rays}
  By Theorem \ref{parameterization}, the map 
   $$\Phi: \mathcal{H}\rightarrow \mathbb{D}, \  ({c},a)\mapsto B_{c,a}(-c)$$
   is a $d_\omega$-fold cover ramified at a single point, where $B_{c,a}$ is the B\"ottcher map of $f^p_{{c},a}$ defined near $f_{c,a}^{\ell(\mathcal{H})}(c)$.
  Taking $\Psi=\sqrt[d_\omega]{\Phi}$ yields a conformal map  $\Psi:\mathcal{H}\rightarrow \mathbb{D}$.
  There are $d_\omega$ choices of $\Psi$, we may fix one of them.

 The {\it parameter ray} $\mathcal{R}(t)$ of angle $t\in [0, 1)$ in $\mathcal{H}$ is defined by
 $$\mathcal{R}(t):=\Psi^{-1}((0,1)e^{2\pi i t}).$$
  The  {\it impression} of $\mathcal{R}(t)$  is
  $\mathcal{I}(t):=\bigcap_{k\geq 1} \overline{\mathcal{S}_k}(t)$, 
 %  defined as the  intersection of  the shrinking   closed sectors $\overline{\mathcal{S}_k}(t), k\geq 1$,
 where
$$\mathcal{S}_k(t)=\Psi^{-1}(\{re^{2\pi i \theta}; r\in (1-1/k,1), \ \theta\in(t-1/k,t+1/k)\}).$$

The following proposition is a sharper version of Proposition \ref{bm1}.

 \begin{pro}\label{dyn-para} Let $t\in [0,1)$ and $(c,a)\in \mathcal{I}(t)$. Write $V=U_{c,a}(f_{c,a}^{\ell(\mathcal{H})}(c))$.
 
 1. If $f_{c,a}$ has a parabolic point, then $R^{V}_{c,a}(d_{\omega}t)$  lands at a parabolic point.
 
 2. If  $f_{c,a}$ has no parabolic point, then $R^{V}_{c,a}(d_{\omega}t)$ lands at $-c$.
 \end{pro}

 \begin{proof}
 Let $\gamma_{c,a}=\gamma_{c,a}(\theta)$ be a graph so that $\gamma_{c,a}\cap {\rm orb}(-c)=\emptyset$ and $\gamma_{c,a}\cap J(f_{c,a})$ consists of repelling points. If $R^{V}_{c,a}(d_{\omega}t)$ lands at neither $-c$ (non parabolic case) nor a parabolic point $p_{c,a}$ (parabolic case), then by Lemma \ref{cont-ray}, there is a neighborhood $\mathcal{V}$ of $(c,a)$ and an integer $l>0$, satisfying that
 
(a).  the graph $f_{c,a}^{-l}(\gamma_{c,a})$ separates  $R_{c,a}^{V}(d_\omega t)$ from the critical point $-c$;

(b).   $\gamma_{c',a'}=\gamma_{c',a'}(\theta)$ is a graph moving continuously for   $(c',a')\in\mathcal{V}$;
 
(c).  $-c'\notin f_{c',a'}^{-l}(\gamma_{c',a'})$ for all $(c',a')\in\mathcal{V}$.

By (b) and (c), 
for all $(c',a')\in \mathcal{R}(t)\cap \mathcal{V}$,  
the graph $f_{c',a'}^{-l}(\gamma_{c',a'})$ separates  $R_{c',a'}^{V'}(d_\omega t)$ from $-c'$, where $V'=U_{c',a'}(f_{c',a'}^{\ell(\mathcal{H})}(c'))$.
But this would contradict the fact that when  $(c',a')\in \mathcal{R}(t)\cap \mathcal{V}$, one has $-c'\in R_{c',a'}^{V'}(d_\omega t)$.
\end{proof}

\begin{rmk} \label{point-char} Let $(c,a)\in \mathcal{I}(t)$ and $f=f_{c,a}$.
 Proposition \ref{dyn-para} asserts that if $f$ has a parabolic cycle,  then the landing point $z$ of  $R^{V}_{c,a}(d_{\omega}t)$ is parabolic.
 By  Proposition \ref{bm1-1},  $T_{f}(-c)$ is periodic, with period say $m$.
 Then $f^m: P_{d+m}(-c)\rightarrow P_{d}(-c)$ defines a renormalization of $f$. Its filled Julia set 
 $K=\bigcap P_{k}(-c)={\rm Imp}(-c)$ satisfies that $\partial K\cap \partial V=\{z\}$.
\end{rmk}

 \subsection{Proof of Theorem \ref{regularity}}  There are three main ingredients in the proof:
 
 $\bullet$  Characterization of the maps on $\partial \mathcal{H}$ (Proposition \ref{dyn-para}).
 
 $\bullet$  Combinatorial rigidity (Theorem \ref{c-rigidity}).
 
 $\bullet$  Holomorphic motion theory \cite{Sl}.

 \begin{proof}   Define $\mathcal{I}_0(t)\subset \mathcal{I}(t)$ by
  $$\mathcal{I}_0(t)=\{(c,a)\in\mathcal{I}(t); f_{c,a} \text{ has no parabolic cycle} \}.$$
 Fix  $(c_0,a_0)\in \mathcal{I}_0(t)$, let $\Gamma_{c_0,a_0}$ be the  $p$-admissible puzzle of $f_{c_0,a_0}$ given by 
 Theorem \ref{adm-puzzle}.  By Lemma \ref{cont-ray},  there is a neighborhood $\mathcal{U}$ of $(c_0,a_0)$, so that
 $\Gamma_{c_0,a_0}$ admits a holomorphic motion $\Gamma_{c,a}$ for  $(c,a)\in \mathcal{U}$.
 
 Precisely, there is a continuous map $h:\mathcal{U} \times ((\mathbb{C}-X_{c_0,a_0})\cup \Gamma_{c_0,a_0})\rightarrow \mathbb{C}$
 defined in the way that for any $((c,a),z)\in \mathcal{U} \times ((\mathbb{C}-X_{c_0,a_0})\cup \Gamma_{c_0,a_0})$, the point 
 $h((c,a),z)$ is in the dynamical plane of $f_{c,a}$, with the same equipotential and internal (or external) angle as that of 
 $z$ in the dynamical plane of $f_{c_0,a_0}$. In other words, $z$ and $h((c,a),z)$ have the same `dynamical position' in their
 corresponding attracting basins.
 One may verify that $h$ is a holomorphic motion parameterized by $\mathcal U$, with base point $(c_0,a_0)$ (namely $h((c_0,a_0),\cdot)=id$). 
  By  \cite{Sl}, $h$ can be extended to a holomorphic motion  $H: \mathcal{U}\times\mathbb C\rightarrow \mathbb C$. 
 In particular,  for any $(c,a)\in \mathcal U$, the map 
 $H((c,a),\cdot): \mathbb C\rightarrow \mathbb C$ is  quasiconformal.

 For any $(c,a)\in \mathcal{I}_0(t)\cap \mathcal U$, let $\phi=H((c,a),\cdot)|_{\Gamma_{c_0,a_0}}$.
 The above $\Gamma_{c,a}$ is nothing but $\phi(\Gamma_{c_0,a_0})$.
By Proposition \ref{dyn-para}, the critical point $-c$ for $f_{c,a}$ has the same 
 `dynamical position' as that of $-c_0$ for $f_{c_0,a_0}$.
  This implies that for any 
 $k\geq 1$, there is a homeomorphism $\phi_k:\Gamma^k_{c_0,a_0}
 \rightarrow  \Gamma^k_{c,a}$, which matches $\phi$ on $\Gamma^k_{c_0,a_0}\cap \Gamma_{c_0,a_0}$.
 Equivalently $f_{c_0,a_0}$ and $f_{c,a}$ have  the same combinatorics up to depth $k$. Since $k$ is arbitrary, the maps  $f_{c_0,a_0}, f_{c,a}$ are qc combinatorially equivalent.
 By Theorem \ref{c-rigidity}, we have $(c,a)=(c_0,a_0)$.
This means that $\mathcal{I}_0(t)$ is  at most a singleton. The discreteness of $\mathcal{I}(t)-\mathcal{I}_0(t)$ and the
connectivity of $\mathcal{I}(t)$ imply that $\mathcal{I}(t)$ is a singleton.
Since $t$ is arbitrary, we have that $\partial\mathcal{H}$ is locally connected.

To finish, we show that $\partial\mathcal{H}$ is a Jordan curve. If this is not true, 
 then $\mathcal{I}(t_1)=\mathcal{I}(t_2)=\{(c_0,a_0)\}$ for some $0\leq t_1< t_2<1$ (here $(c_0,a_0)$ is the same
 symbol as we used above, without assuming that $f_{c_0,a_0}$ has no parabolic cycle). Then by Lemma \ref{dyn-para} and
 Remark \ref{point-char}, 
 the internal rays $R^{V}_{c_0,a_0}(d_{\omega}t_1)$ and  $R^{V}_{c_0,a_0}(d_{\omega}t_2)$ land at the same point,
 which is exactly the unique intersection point of ${\rm Imp}(-c)\cap \partial V$,
 here $V=U_{c_0,a_0}(f_{c_0,a_0}^{\ell(\mathcal{H})}(c_0))$.
 Since $\partial V$ is a Jordan curve \cite{RY}, this implies that 
 $$d_{\omega}t_1=d_{\omega}t_2 \ {\rm mod }\ \mathbb{ Z }.$$
 
Let $\mathcal U, H$ be  defined as above. We may shrink  $\mathcal U$ if necessary so that for all
$(c,a)\in \mathcal U$, one has $f_{c,a}(-c)\notin \Gamma_{c,a}$.
It follows that $f_{c,a}^{-1}(\Gamma_{c,a})$ moves continuously with respect to $(c,a)\in \mathcal U$, and avoids $-c$ along the motion.
 Choose 
 $(c_1,a_1)\in \mathcal R(t_1)\cap \mathcal{U}$ and $(c_2,a_2)\in \mathcal R(t_2)\cap \mathcal{U}$ with
 $$\Phi(c_1,a_1)=\Phi(c_2,a_2).$$
 Note that $f_{c_1,a_1}$ and $f_{c_2,a_2}$ are hyperbolic.
 Let 
 $$\psi=H((c_2,a_2), \cdot) \circ H((c_1,a_1), \cdot)^{-1}.$$
 Clearly $\psi$ is a quasi-conformal map from $\mathbb C$
 to $\mathbb C$, 
 holomorphic in $\mathbb C-X_{c_1,a_1}$, and $\psi(\Gamma_{c_1,a_1})=\Gamma_{c_2,a_2}$.
 We may get a modification $\psi_0$ of $\psi$ so that $\psi_0$ matches $\psi$ in $(\mathbb{C}-X_{c_1,a_1})\cup \Gamma_{c_1,a_1}$, and
 $\psi_0$ is the identity map under the B\"ottcher coordinates
 defined in $Y=f_{c_1,a_1}^{-M}(\mathbb C-X_{c_1,a_1})$, here the integer $M\geq 0$ is chosen so that $f_{c_1,a_1}(-c_1)\in Y$.
  In this way $\psi_0$ gives a conjugacy between $f_{c_1,a_1}$ and $f_{c_2,a_2}$ on the postcritical set of $f_{c_1,a_1}$.

 The relation  $\Phi(c_1,a_1)=\Phi(c_2,a_2)$ implies that $f_{c_1,a_1}, f_{c_2,a_2}$ have the same critical dynamical positions.
 This allows one to get a sequence of qc-maps $\psi_k$'s by the lifting process
 $$f_{c_2,a_2}\circ \psi_{k+1}=\psi_{k}\circ f_{c_1,a_1},$$
so that $\psi_{k+1}$ and $\psi_{k}$ are isotopic rel the postcritical set of $f_{c_1,a_1}$,  holomorphic and identical on $f_{c_1,a_1}^{-k}(Y)$.
 The dilatations of $\psi_k$'s are uniformly bounded, so they have a limit $\psi$, which is  a quasi-conformal map on $\mathbb C$, holomorphic in the Fatou set of $f_{c_1,a_1}$. Since $f_{c_1,a_1}$ is hyperbolic, its Julia set has zero measure, we conclude that $\psi$ is a conformal map. One has $\psi=id$ and $(c_1,a_1)=(c_2,a_2)$.  
  This contradicts the fact that $(c_1,a_1)\neq (c_2,a_2)$.    
  \end{proof}

 \begin{rmk} By Theorem \ref{regularity}, one can state Proposition \ref{dyn-para} as follows:
 Suppose the parameter ray $\mathcal{R}(t)$ lands at $(c,a)$.
  Write $V=U_{c,a}(f_{c,a}^{\ell(\mathcal H)}(c))$. 
 
 1. If $d_\omega t$ is $2$-periodic, then $f_{c,a}$ has a parabolic point on $\partial V$.
 
 2. If  $d_\omega t$ is not $2$-periodic, then $-c\in\partial V$.
  \end{rmk}

 Lastly, the following fact has some independent interest.
 
 \begin{pro} Let ${\rm Bif}(\mathcal S_p)$ be the bifurcation locus of $\mathcal S_p$ and $\mathcal E$ be a compact subset of $\mathcal C(\mathcal S_p)$. For any component $\mathcal U$  of $\mathcal S_p-\mathcal E$,  either
 
  1. $\mathcal U$ is unbounded in $\mathcal S_p$, or
 
 2. $\mathcal U \subset \mathcal C(\mathcal S_p)$ and ${\rm Bif}(\mathcal S_p)\cap\mathcal U=\emptyset$.
 \end{pro}
 \begin{proof} Assume that $\mathcal U$ is bounded in $\mathcal S_p$, we will show that ${\rm Bif}(\mathcal S_p)\cap\mathcal U=\emptyset$. To this end, consider the holomorphic maps  $F_k: \mathcal U\rightarrow \mathbb{C}$ defined by
 $$F_k(c,a)= f^k_{c,a}(-c), \ k\geq 0.$$
 Clearly $\partial \mathcal U\subset \mathcal C(\mathcal S_p)$.
  So for any $(c,a)\in\partial \mathcal U$, one has $F_k(c,a)\in K(f_{c,a})$ (the filled Julia set).
 By univalent function theory, one has $|\zeta-c|\leq 4, \ \forall \zeta\in K(f_{c,a})$. Therefore 
 $$|F_k(c,a)|\leq |c|+4\leq \sup_{(c',a')\in \partial \mathcal{U}}|c'|+4, \ \forall (c,a)\in \partial\mathcal U.$$
 By the maximum modulus principle, the above inequality holds for all $(c,a)\in \mathcal U$. Then 
 Montel's theorem asserts that $\{F_k\}$ is a normal family in $\mathcal U$. Equivalently,  ${\rm Bif}(\mathcal S_p)\cap\mathcal U=\emptyset$.
 \end{proof}

 \section {Capture component} \label{cc}
 
  The main purpose of this section is to show
 
 \begin{thm} \label{boundary-C} Every Type-C hyperbolic component is  a Jordan disk.
 \end{thm}

 Let $\mathcal{H}$ be a Type-C component and $f_{c,a}\in \mathcal{H}$. Let
  $l>0$ be the smallest integer so that
  $f_{c,a}^l(-c)\in \mathcal{A}_{c,a}$. By Theorem \ref{parameterization}, the map
 $$\Phi: \mathcal{H}\rightarrow \mathbb{D}, \  ({c},a)\mapsto B_{c,a}(f_{c,a}^l(-c))$$
 is a conformal isomorphism, where $B_{c,a}$ is the B\"ottcher map of $f^p_{{c},a}$ defined in $U_{c,a}(f_{c,a}^l(-c))$.
Let $\kappa\in[0,p)$ be the unique integer so that  
$$U_{c,a}(f_{c,a}^l(-c))=U_{c,a}(f_{c,a}^\kappa(c)).$$
 
  The {\it parameter ray} $\mathcal{R}(t)$ of angle $t\in \mathbb{R}/\mathbb{Z}$ in $\mathcal{H}$ is 
 $\Phi^{-1}((0,1)e^{2\pi i t}).$
  The  {\it impression } $\mathcal{I}(t)$ of $\mathcal{R}(t)$  is defined as 
  $\mathcal{I}(t)=\bigcap_{k\geq 1}\overline{\mathcal{S}_k}(t)$, where
$$\mathcal{S}_k(t)=\Phi^{-1}(\{re^{2\pi i \theta}; r\in (1-1/k,1), \ \theta\in(t-1/k,t+1/k)\}).$$

Let  $v_{c,a}=f_{c,a}(-c)$ be the free critical value.
 For $(c,a)\in\mathcal{H}$, let $V_{c,a}$ be  the Fatou component of $f_{c,a}$ containing $v_{c,a}$.  Clearly,  the {\it center} $\sigma=\sigma(c,a)$ of  $V_{c,a}$, defined as  the unique point $\sigma\in  V_{c,a}$ satisfying $f_{c,a}^{l-1}(\sigma)=f_{c,a}^{\kappa}(c)$,   moves continuously with respect to   $(c,a)\in\mathcal{H}$.  The center map $(c,a)\mapsto \sigma$ has a continuous extension to $\partial\mathcal{H}$. Therefore,  when $(c,a)\in \partial\mathcal{H}$, the point $\sigma(c,a)$ and Fatou component  $V_{c,a}$ containing $\sigma(c,a)$ are  well-defined.
 
 Let  $\mathcal{H}_{AB}$ be the union of all Type-A and Type-B hyperbolic components. 
 Clearly, $\overline{\mathcal{H}_{AB}}$  is compact because it is a closed subset of $\mathcal{C}(\mathcal S_p)$.

 \begin{lem}\label{dyn-para2} For any $t\in [0,1)$ and any $(c,a)\in \mathcal{I}(t)\setminus\overline{\mathcal{H}_{AB}}$, the dynamical internal ray $R^{V_{c,a}}_{c,a}(t)$ lands at $v_{c,a}\in\partial V_{c,a}$.
 \end{lem}

 \begin{proof} Note that for any $(c_0,a_0)\in\mathcal{I}(t)\setminus\overline{\mathcal{H}_{AB}}$,  there is a disk neighborhood  $\mathcal{U}$ of $(c_0,a_0)$ contained in  $\mathcal{S}_p\setminus \overline{\mathcal{H}_{AB}}$.

 Let  $W_{c,a}:=U_{c,a}(f_{c,a}^\kappa(c))$ for $(c,a)\in \mathcal U$.  We first claim that  $\partial W_{c,a}$ moves holomorphically   with respect to $(c,a)\in\mathcal{U}$. To see this,  we define a map $h: \mathcal{U}\times W_{c_0,a_0}\rightarrow \mathbb{{C}}$ by $h((c,a), z)=(B^{W_{c,a}}_{c,a})^{-1}\circ B^{W_{c_0,a_0}}_{c_0,a_0}(z)$, where 
 $B^{W_{c,a}}_{c,a}$ is the B\"ottcher map of $f_{c,a}^p$ defined in $W_{c,a}$.
   It satisfies:

% so that for all $((c,a), z)\in \mathcal{U}\times W_{c_0,a_0}$, the points  $h((c,a), z)\in W_{c,a}$ and $z$ have the same 
% B\"ottcher coordinates under the corresponding B\"ottcher maps.
% 

 (1). Fix any $z\in W_{c_0,a_0}$, the map $(c,a)\mapsto h((c,a), z)$ is holomorphic;

 (2). Fix any $(c,a)\in \mathcal{U}$, the map $z\mapsto h((c,a), z)$ is injective;

 (3). $h((c_0,a_0), z)=z$ for all $z\in W_{c_0,a_0}$.

 These properties imply that $h$ is a holomorphic motion parameterized by $\mathcal{U}$,  with base point $(c_0,a_0)$. By the Holomorphic Motion Theorem \cite{Sl}, there is a holomorphic motion $H:\mathcal{U}\times
\mathbb{\widehat{C}}\rightarrow \mathbb{\widehat{C}}$ extending $h$ and for any $(c,a)\in \mathcal{U}$,
we have $H((c,a),\partial  W_{c_0,a_0})=\partial W_{c,a}$. Therefore  $\partial W_{c,a}$ moves holomorphically  with respect to $(c,a)\in\mathcal{U}$. The claim is proved.

It follows that  for any $k\geq 0$, the set $f_{c,a}^{-k}(\partial W_{c,a})$ moves continuously in Hausdorff topology
 with respect to $(c,a)\in\mathcal{U}$.
  By the assumption  $(c_0,a_0)\in \mathcal{I}(t)\setminus\overline{\mathcal{H}_{AB}}$, there exist a sequence of parameters $(c_n,a_n)$'s in $\mathcal{H}\cap \mathcal U$ and  a sequence of angles $t_n$'s, such that
  \bess &(c_n,a_n)\rightarrow (c_0,a_0)\text{ and } t_n\rightarrow t \text{ as } n\rightarrow\infty,&\\
   &f_{c_n,a_n}^l(-c_n)\in R^{W_{c_n,a_n}}_{c_n,a_n}(t_n)  \text{ for all } n\geq 1.&
    \eess

By the continuity of $(c,a)\mapsto \partial W_{c,a}$  (implying the continuity of the ray $R^{W_{c,a}}_{c,a}(t)$ with respect to $((c,a),t)\in \mathcal U\times \mathbb S$), we have $f_{c_0,a_0}^{l}(-c_0)\in  \partial W_{c_0,a_0}$
and the internal ray $R_{c_0,a_0}^{W_{c_0,a_0}}(t)$ lands at $f_{c_0,a_0}^{l}(-c_0)$.   By the continuity of
 $(c,a)\mapsto  f_{c,a}^{1-l}(\partial W_{c,a})$
   and the fact $f_{c,a}^{l-1}(R_{c,a}^{V_{c,a}}(t))=R_{c,a}^{W_{c,a}}(t)$ for $(c,a)\in \overline{\mathcal{H}}\cap \mathcal{U}$, we have that $v_{c_0,a_0}\in \partial V_{c_0,a_0}$ and
 $R^{V_{c_0,a_0}}_{c_0,a_0}(t)$ lands at $v_{c_0,a_0}$.
 \end{proof}

 To show that $\partial{\mathcal{H}}$ is a Jordan curve, we need some lemmas, whose proofs   are very technical. 
 For each $t$, let's define
 $$\mathcal{I}^*(t)= \mathcal{I}(t)\setminus \overline{\mathcal{H}_{AB}}.$$

 \begin{lem}\label{para-ray-c} For all $(c,a)\in \mathcal{I}^*(t)$, we have either  $f_{c,a}^{l+p}(-c)=f_{c,a}^{l}(-c),$
or
$$f_{c,a}^{l-1}(-c)\notin \partial \mathcal{A}_{c,a}  \text{ and }  f_{c,a}^{l}(-c)\in \partial \mathcal{A}_{c,a}.$$
 \end{lem}
 \begin{proof}  By Lemma \ref{dyn-para2}, we know $f_{c,a}^{l-1}(-c)\in \partial f_{c,a}^{l-2}(V_{c,a})$.
If $f_{c,a}^{l-1}(-c)\in \partial \mathcal{A}_{c,a}$, then $f_{c,a}^{l-1}(-c)\in  \partial f_{c,a}^{l-2}(V_{c,a})\cap \partial U_{c,a}(f_{c,a}^{m}(c))$ for some $m$. 
Let $W_1=f_{c,a}^{l-1}(V_{c,a})$ and $W_2=U_{c,a}(f_{c,a}^{m+1}(c))$, clearly $W_1,W_2\in \mathcal B_{c,a}$.
% they both are attracting basins.
If $W_1=W_2$, then necessarily 
$f_{c,a}^{l-1}(-c)=-c$, 
 implying that $-c$ is a periodic point on Julia set, impossible!  So we have $W_1\neq W_2$ and $f_{c,a}^l(-c)\in \partial W_1\cap \partial W_2$.  By Proposition \ref{common-point}, we get
$f_{c,a}^{l+p}(-c)=f_{c,a}^{l}(-c)$. 
 \end{proof}

\begin{lem} \label{imp-h} For each $t$, the set 
$\mathcal{I}^*(t)$  is either empty or a singleton.
\end{lem}
\begin{proof}  If it is not true, then there exist a connected and compact subset $\mathcal{E}$ of $\mathcal{I}^*(t)$  containing at least two points.

  By Lemma \ref{dyn-para2}, the internal ray $R^{V_{c,a}}_{c,a}(t)$ lands at $v_{c,a}$ for all $(c,a)\in \mathcal{E}$.
By Thurston's theorem \cite{DH}, there are only finitely many pairs $(c,a)\in\mathcal S_p$ for which $f_{c,a}^{l+p}(-c)=f_{c,a}^{l}(-c)$.  
By Lemma \ref{para-ray-c} (and shrink  $\mathcal{E}$ if necessary), we may assume all $(c,a)\in\mathcal{E}$ satisfy $f_{c,a}^{l-1}(-c)\notin \partial \mathcal{A}_{c,a}$ and $f_{c,a}^{l}(-c)\in \partial \mathcal{A}_{c,a}.$

By  continuity, there is Jordan disk
 $\mathcal{D}$ with $\mathcal{E}\subset \mathcal{D}\subset \mathcal S_p-\overline{\mathcal{H}_{AB}}$, so that for all $(c,a)\in \mathcal{D}$, we have $f_{c,a}^{l-1}(-c)\notin \overline{\mathcal{A}}_{c,a}$. 

Now take two different pairs $(c_1, a_1), (c_2, a_2)\in \mathcal{E}$. 
Let
 $$J=\{f_{c_1,a_1}^j(v_{c_1,a_1}); 0\leq j \leq l-2\}\cup \mathcal{A}_{c_1,a_1}\cup A_{c_1,a_1}^{\infty}.$$
  It's clear that ${J}$ contains the post-critical set of $f_{c_1,a_1}$.
We define a continuous map $h: \mathcal{D}\times J\rightarrow \mathbb{\widehat{C}}$ satisfying that
$$h((c,a), z)=(B^{\infty}_{c,a})^{-1}\circ B^{\infty}_{c_1,a_1}(z)$$
  for all $((c,a),z)\in \mathcal{D}\times A_{c_1,a_1}^{\infty}$;
  and
$$h((c,a), z)=(f_{c,a}^{\kappa-i}|_{U_{c,a}(f_{c,a}^i(c))})^{-1}\circ B_{c,a}^{-1}\circ B_{c_1,a_1}\circ f_{c_1, a_1}^{\kappa-i}(z)$$
  for all $((c,a),z)\in \mathcal{D}\times U_{c_1,a_1}(f_{c_1,a_1}^i(c_1)), \ 1\leq i\leq \kappa$;
  and
   $$h((c,a), z)=f_{c,a}^{i-\kappa}\circ B_{c,a}^{-1}\circ B_{c_1,a_1}\circ
    (f_{c_1, a_1}^{i-\kappa}|_{U_{c_1,a_1}(f_{c_1,a_1}^i(c_1))})^{-1}(z)$$
  for all $((c,a),z)\in \mathcal{D}\times U_{c_1,a_1}(f_{c_1,a_1}^i(c_1)), \ \kappa< i\leq p$; and
     $$h((c,a), f_{c_1,a_1}^j(v_{c_1,a_1}))=f_{c,a}^j(v_{c,a}) $$
       for all $(c,a)\in \mathcal{D}$ and  $0\leq j\leq l-2$.
%By definition, $h((c_1,a_1), \cdot)$ is the identity map. The map

One may verify that $h$ is a holomorphic motion, parameterized by $\mathcal{D}$ and with base point $(c_1,a_1)$ (i.e. $h((c_1,a_1), \cdot)=id$). By the Holomorphic Motion  Theorem \cite{Sl}, there is a holomorphic motion $H:\mathcal{D}\times
\mathbb{\widehat{C}}\rightarrow\mathbb{\widehat{C}}$ extending $h$. We consider the restriction $H_0=H|_{\mathcal{E}\times
\mathbb{\widehat{C}}}$ of $H$.
  Note that fix any  $(c,a)\in \mathcal{E}$, the map  $z\mapsto H((c,a),z)$ sends
  the post-critical set of $f_{c_1,a_1}$ to that of $f_{c, a}$, preserving the dynamics on this set.
  % between $f_{\lambda_1}$ and  $f_{\lambda_2}$,
  By the  lifting property, there is a
  unique continuous map  $H_1: \mathcal{E}\times \mathbb{\widehat{C}}\rightarrow \mathbb{\widehat{C}}$ such that
$f_{c,a}(H_1((c,a),z))=H_0((c,a), f_{c_1,a_1}(z))$ for all $((c,a),z)\in\mathcal{E}\times\widehat{\mathbb{C}}$
and
   $H_1((c_1,a_1),\cdot)\equiv id$. It's not hard to see that $H_{1}$ is also a holomorphic motion.

 Set $\psi_0=H_0((c_2,a_2), \cdot)$ and $\psi_1=H_1((c_2,a_2), \cdot)$.
  Both  $\psi_0$ and $\psi_1$ are quasi-conformal maps, satisfying  $f_{c_2,a_2}\circ \psi_{1}=\psi_0\circ f_{c_1,a_1}$.
   One may verify that $\psi_0$ and $\psi_1$ are isotopic rel ${J}$.
Again by the lifting property, there is a sequence of quasi-conformal maps $\psi_j$'s satisfying that

 (a).  $f_{c_2,a_2}\circ \psi_{j+1}=\psi_j\circ f_{c_1,a_1}$
for all $j\geq0$,

(b). $\psi_{j+1}$ and $\psi_{j}$ are isotopic rel $f_{c_1,a_1}^{-j}({J})$.

%(c). $\psi_{j+1}|_{f_{\lambda_1}^{-j}(P(f_{\lambda_1})\cup \overline{B_{\lambda_1}})}=
%\psi_{j}|_{f_{\lambda_1}^{-j}(P(f_{\lambda_1})\cup \overline{B_{\lambda_1}})}$.

The maps $\psi_j$'s form a normal family because their dilatations are
uniformly bounded above.  The limit map  $\psi_\infty$ of  $\psi_j$'s is quasi-conformal in $\mathbb{{C}}$,  holomorphic in the Fatou set $F(f_{c_1,a_1})$ of $f_{c_1,a_1}$ and  satisfies
$f_{c_2,a_2}\circ \psi_{\infty}=\psi_\infty\circ f_{c_1,a_1}$ in $F(f_{c_1,a_1})$. By continuity, $\psi_\infty$ is a conjugacy on $\mathbb{C}$.
By Theorem \ref{cubic-property}, the Julia set of $f_{c_1,a_1}$ carries no invariant line fields.
Therefore  $\psi_\infty$ is an affine map: $\psi_\infty(\zeta)=a\zeta+b$. 
Since $\psi_\infty$ is tangent to the identity map near $\infty$, we have  $\psi_\infty=id$ and $(c_1,a_1)=(c_2, a_2)$.
 This  contradicts our choice of  $(c_i,a_i)$.  
%
%
%This affine map has only two possibilities: either $\psi_\infty=id$ or $\psi_\infty(z)=-z$. If $\psi_\infty(z)=-z$ then $(c_2,a_2)=-(c_1,a_1)$, but this is not allowed by the choice of $\mathcal{E}$. So we are in the situation  $(c_1,a_1)=(c_2, a_2)$,
\end{proof}

It follows from Lemma \ref{imp-h}  that  the impression $\mathcal{I}(t)$ is either a singleton or
 contained in
   $\partial\mathcal{H}_{AB}$.  To analyze the latter case, we  shall prove

\begin{lem} \label{imp-2} If $\mathcal{I}(t)\subset \partial\mathcal{H}_{AB}$, then $\mathcal{I}(t)$ is a singleton. 
\end{lem}

\begin{proof}   
Choose $(c,a) \in \mathcal{I}(t)\subset  \partial\mathcal{H}_{AB}$ so that $f_{c,a}$ has no parabolic cycle. 
By assumption, there is a hyperbolic component  $\mathcal{H}_\omega$ of Type-$\omega\in\{A,B\}$,  so that $(c,a)\in \partial\mathcal{H}_\omega\cap\partial\mathcal{H}$. Since $\mathcal{H}_\omega$  is a Jordan disk (by Theorem \ref{regularity}), there is a parameter ray, say $\mathcal{R}_\omega(\alpha)$ in $\mathcal{H}_\omega$, landing at 
$(c,a)$.

By Proposition \ref{bm1}, the critical point $-c\in \partial Y$ for some attracting basin $Y\in \mathcal B_{c,a}$. Such $Y\in \mathcal B_{c,a}$ is unique (if not, then by Remark \ref{common-l}, we have $f_{c,a}^p(-c)=-c$, which is impossible). By Proposition \ref{dyn-para}, the internal ray $R_{c,a}^Y(d_\omega\alpha)$ lands at $-c$.
It follows that the internal ray $R_{c,a}^W( \varepsilon (Y)d_\omega\alpha)$ land at the critical value $v_{c,a}=f_{c,a}(-c)\in \partial W$,
where 
$$W=f_{c,a}(Y) \ \text{ and } \  \varepsilon (Y)=1 \text{ if }Y\neq U_{c,a}(c); \ \varepsilon (Y)=2 \text{ if }Y= U_{c,a}(c).$$ 
Let $L\geq0$ be the first integer with $f_{c,a}^L(V_{c,a})=W$, and $W'=f_{c,a}^L(W)$. 
Then   
$$f_{c,a}^L(R_{c,a}^W(\varepsilon (Y)d_{\omega}\alpha))=R_{c,a}^{W'}(\varepsilon (Y)2^M d_\omega\alpha), \  f_{c,a}^L(R_{c,a}^{V_{c,a}}(t))=R_{c,a}^{W}(2^N t),$$
where 
\bess 
&N=\#\{0\leq j<L; f^j_{c,a}(V_{c,a})=U_{c,a}(c)\},& \\
 &M=\#\{0\leq j<L; f^j_{c,a}(W)=U_{c,a}(c)\}.&\eess
It's clear that $N\in\{0,1\}$ and $M< L/p+1$.

In the following, we will show that 
$$\varepsilon (Y)2^M d_\omega\alpha=2^N t  \ {\rm mod} \ \mathbb{Z}.$$

 There are two possibilities:

If $W'\neq W$ and the rays $R_{c,a}^{W'}(\varepsilon (Y)2^M d_\omega\alpha), R_{c,a}^{W}(2^N t)$ land at the same point, then by Remark \ref{common-l}, we have 
$$\varepsilon (Y)2^M d_\omega\alpha=2^N t=0 \ {\rm mod} \ \mathbb{Z}.$$  

If  $R_{c,a}^{W'}(\varepsilon (Y)2^M d_\omega\alpha)$ and $R_{c,a}^{W}(2^N t)$ land at different points (no matter $W=W'$ or not).
Let $\gamma_{c,a}(\theta)$ (see Section \ref{cubic-graph} for definition) be a graph  avoiding the orbit of $-c$ and  so that $\gamma_{c,a}(\theta)\cap J(f_{c,a})$ consists of repelling points.
By Lemma \ref{cont-ray}, there exist an integer $K>0$  and a neighborhood $\mathcal{V}$ of $(c,a)$, with the following three properties:

(a).   $\gamma_{c',a'}(\theta)$ is well-defined and moves continuously when   $(c',a')\in\mathcal{V}$;
 
(b).  $v_{c',a'}=f_{c',a'}(-c')\notin f_{c',a'}^{-K-L}(\gamma_{c',a'}(\theta))$ for all $(c',a')\in\mathcal{V}$;

(c).   when $(c',a')=(c,a)$, the graph $f_{c,a}^{-K}(\gamma_{c,a}(\theta))$ separates the internal rays $R_{c,a}^{W'}(\varepsilon (Y)2^M d_\omega\alpha)$ and $R_{c,a}^{W}(2^N t)$.

Note that $R_{c,a}^{W}(\varepsilon (Y)d_\omega\alpha)$ lands at $v_{c,a}$. By (c),  we see that  $f_{c,a}^{-K-L}(\gamma_{c,a}(\theta))$ separates  $R_{c,a}^{W}(\varepsilon (Y)d_\omega\alpha)\cup\{v_{c,a}\}$ from $R_{c,a}^{V_{c,a}}(t)$.
Then by  properties (a) and (b), we see that
 $f_{c',a'}^{-K-L}(\gamma_{c',a'}(\theta))$ separates $v_{c',a'}$ from the set $R_{c',a'}^{V_{c',a'}}(t)$ for all $(c',a')\in \mathcal{V}
 \cap \mathcal{H}$.  This contradicts the fact that when  
 $(c',a')\in \mathcal{V}
 \cap \mathcal{R}(t)$, the critical value $v_{c',a'}\in R_{c',a'}^{V_{c',a'}}(t)$.

Therefore each  $(c,a)\in \mathcal{I}(t)$ either corresponds to a map $f_{c,a}$ with a parabolic cycle or 
a landing point of the parameter ray $\mathcal{R}_\omega(\alpha)$ of a Type-$\omega\in\{A,B\}$ component $\mathcal{H}_{\omega}$, with
 $\varepsilon (Y)2^M d_\omega\alpha=2^N t  \ {\rm mod} \ \mathbb{Z}$.
 So $\mathcal{I}(t)$ is  a discrete set. The connectedness of $\mathcal{I}(t)$ implies that it is a singleton.
 \end{proof}

\noindent {\it Proof of Theorem \ref{boundary-C}.}
By  Lemmas \ref{imp-h} and \ref{imp-2},   $\partial \mathcal{H}$ is locally connected.  Assume that 
two  parameter rays  $\mathcal{R}(t_1),  \mathcal{R}(t_2)$  with $t_1\neq t_2 \ {\rm mod }\   \mathbb{Z}$, land at the same point $(c,a)\in\partial \mathcal{H}$. Let's look at the dynamical plane of $f_{c,a}$. Note that $\partial V_{c,a}$ is a Jordan curve \cite{RY},
 the internal rays $R_{c,a}^{V_{c,a}}(t_1)$ and $R_{c,a}^{V_{c,a}}(t_2)$ would land at different points. 
  Similar as the proof as  Lemma \ref{imp-2},  let $\gamma_{c,a}(\theta)$ be a graph  avoiding the orbit of $-c$ and  so that $\gamma_{c,a}(\theta)\cap J(f_{c,a})$ consists of repelling points.
By Lemma \ref{cont-ray}, there exist  an integer $q\geq0$  and a neighborhood $\mathcal{V}$ of $(c,a)$, satisfying that

(a).  the graph $f_{c,a}^{-q}(\gamma_{c,a}(\theta))$ separates  $R_{c,a}^{V_{c,a}}(t_1)$ and $R_{c,a}^{V_{c,a}}(t_2)$.

(b).   $\gamma_{c',a'}(\theta)$ is well-defined and moves continuously when   $(c',a')\in\mathcal{V}$;
 
(c).  $-c'\notin f_{c',a'}^{-q}(\gamma_{c',a'}(\theta))$ for all $(c',a')\in\mathcal{V}$;

 The continuity of $f_{c',a'}^{-q}(\gamma_{c',a'}(\theta))$ implies that it would separated
 the sets $R_{c',a'}^{V_{c',a'}}(t_1)$ and $R_{c',a'}^{V_{c',a'}}(t_2)$ (with one possibly bifurcating).
This  implies  that $(c,a)$ can not be the landing point of  the parameter rays $\mathcal{R}(t_1)$ and $\mathcal{R}(t_2)$ simultaneously.   
  This contradicts our assumption.  \hfill\fbox

\end{document}